\documentclass[a4paper]{article}
\usepackage{amsmath,amsfonts,amsthm, amssymb,xspace,color}
\usepackage[all]{xy}
\newtheorem{theorem}{Theorem}[section] 
\newtheorem{theoremA}{Theorem}

\newtheorem{corollaryA}[theoremA]{Corollary}
\newtheorem{lemma}[theorem]{Lemma} 
\newtheorem{proposition}[theorem]{Proposition}
\newtheorem{corollary}[theorem]{Corollary}
\theoremstyle{definition}
\newtheorem{definition}[theorem]{Definition}
\newtheorem{example}[theorem]{Example}
\newtheorem{remark}[theorem]{Remark}

\newcommand{\Deflation}[2]{{\rm Defl}(#1,#2)}

\newcommand{\Q}{{\mathbb Q}}
\newcommand{\Z}{{\mathbb Z}}
\newcommand{\N}{{\mathbb N}}

\newcommand{\HH}{{\mathcal H}}
\newcommand{\TT}{{\mathcal T}}
\newcommand{\LL}{{\mathcal L}}

\newcommand{\pc}{\sf pc}
\newcommand{\gam}{\Gamma} 
\newcommand{\dE}{{\vec{E}}} 
\newcommand{\be}{{\overline{e}}} 
\newcommand{\GG}{\mathcal{G}}
\newcommand{\lam}{\Lambda} 
\newcommand{\artgraph}{\Sigma} 
\newcommand{\Geo}{{\hbox{\sc{Geo}}}}
\newcommand{\SL}{{\hbox{\sc{sl}}}}
\newcommand{\SLex}{{\hbox{\sc{SL}}}}
\newcommand{\LHGeo}{{L^H({\hbox{\footnotesize{\Geo}}})}}
\newcommand{\groupid}{1}
\newcommand{\emptyword}{\epsilon}

\newcommand{\Rel}{{\sf Rel}\xspace} 
\newcommand{\SACA}{{\sf SACA}\xspace} 
\newcommand{\SSCA}{{\sf SSCA}\xspace} 
\newcommand{\gspan}[1]{\left\langle #1 \right\rangle}

\newcommand{\Yi}{\ensuremath{Y_i}}

\newcommand{\elen}[2]{\ensuremath{|#1|_{#2}}}  
\newcommand{\wlen}[1]{\ensuremath{|#1|}}     

\newenvironment{mylist}{\begin{list}{}{
\setlength{\parskip}{0mm}
\setlength{\topsep}{2mm}
\setlength{\parsep}{0mm}
\setlength{\itemsep}{0.5mm}
\setlength{\labelwidth}{7mm}
\setlength{\labelsep}{3mm}
\setlength{\itemindent}{0mm}
\setlength{\leftmargin}{12mm}
\setlength{\listparindent}{6mm}
}}{\end{list}}

\newcommand{\Addresses}{{
  \bigskip

  \textsc{Department of Mathematics, University of Nebraska,
    Lincoln, NE 68588-0130, USA}\par\nopagebreak
  \textit{E-mail address}: \texttt{hermiller@unl.edu}

  \medskip

  \textsc{Mathematics Institute,
    University of Warwick,
    Coventry CV4 7AL, UK}\par\nopagebreak
  \textit{E-mail address}: \texttt{D.F.Holt@warwick.ac.uk}

  \medskip

  \textsc{School of Mathematics, Statistics and Physics,
    University of Newcastle,
    Newcastle NE1 7RU,
    UK}\par\nopagebreak
  \textit{E-mail address}: \texttt{Sarah.Rees@ncl.ac.uk}

  \medskip

  \textsc{Hall College Center,
    84 Alford Rd,
    Bard College at Simon’s Rock,
    Great Barrington, MA}\par\nopagebreak
  \textit{E-mail address}: \texttt{tsusse@simons-rock.edu}

}}

\title{Automaticity for graphs of groups}
\author{Susan Hermiller, Derek F. Holt, Sarah Rees and Tim Susse}
\date{21 June 2020}

\begin{document}

\maketitle
\begin{abstract}
In this article we construct 
asynchronous and sometimes synchronous automatic
structures for amalgamated products and HNN extensions
of groups that are
strongly asynchronously (or synchronously) 
coset automatic with respect to the associated
automatic subgroups, subject to further geometric conditions.
These results are proved in the  general context
of fundamental groups of graphs of groups. 
The hypotheses of our closure results
are satisfied in a variety of examples such as
Artin groups of sufficiently large type, Coxeter groups, virtually
abelian groups, 
and groups that are hyperbolic relative to virtually
abelian subgroups.
\end{abstract}

\noindent 2010 Mathematics Subject Classification: 20F65, 20F10, 20E06, 20F36

\noindent Key words: Automatic group, automatic coset system,
graph of groups, relatively hyperbolic group,
3-manifold group, Artin group

\bigskip

\section{Introduction}
\label{sec:intro}

Closure properties for the classes of automatic and
asynchronously automatic groups are known for 
a variety of group constructions.
The class of automatic groups is
closed with respect to finite index supergroups and 
subgroups, 
direct products, 
free products~\cite[Chapter 12]{ECHLPT}, 
and graph products~\cite{HermillerMeier}.
For amalgamated products and HNN extensions, closure for
automaticity is also known in  some special cases.
Epstein et al.\ show in~\cite[Theorems~12.1.4,~12.1.9]{ECHLPT}
that an amalgamated product or HNN extension
of automatic groups along a finite subgroup is automatic,
while Baumslag et al. [4] show closure 
for automaticity under amalgamated products
with other technical restrictions. It is
proved in particular in \cite[Theorems~E,B,D]{BGSS} that amalgamated 
products of two finitely generated free groups over a finitely generated subgroup  are asynchronously automatic, and that amalgamated products of
two finitely generated abelian groups over a subgroup or of two
negatively curved groups over a cyclic subgroup are automatic.

This article derives automatic structures for new families of groups
as a consequence of constructive proofs of closure properties for the
class of groups with strong automatic coset systems.
The study of groups that are automatic relative to a subgroup was
introduced by Redfern in~\cite{Redfern}. 
But Redfern had slightly weaker conditions on the associated structures 
than we need, and our definition of
strong automatic coset systems comes from later work of Holt and Hurt
in~\cite{HoltHurt}, where some fellow travelling conditions were added. 
With either variant of the definition, an automatic coset system provides a
quadratic time algorithm for reducing an element of a coset of $H$ in $G$ to a
normal form representative of that coset and, in particular, a quadratic time
solution for the membership problem of elements of $G$ in $H$.

Our article constructs strong automatic coset systems
for fundamental groups of graphs of groups,
given that the vertex groups have such systems with respect to corresponding 
edge groups, and given certain geometric conditions.
In particular the construction can be applied to amalgamated
products and HNN extensions, given those conditions.
We build new automatic structures out of the automatic
coset systems we are now able to build.
Our results also generalise those of \cite{ECHLPT} relating to amalgamated
products and HNN extensions.

The main results of the article are
Theorems~\ref{thm:gog_coset},~\ref{thm:gog_coset_synch},
and~\ref{thm:concat_synch}.
The first two of these provide our most general combination theorems,
using only conditions introduced in Section~\ref{sec:basic}; the third 
deals with groups satisfying a particular condition on their geodesics.
Those three main results are
stated briefly at the end of this introduction, together with
a pair of corollaries, relating specifically to the fundamental groups
of compact 3-manifolds, and to Artin groups.

Related results on the construction of synchronous and asynchronous
automatic structures on amalgamated products and HNN extensions were
developed in \cite{BGSS}, and their generalisations to graphs of groups
were the topic of Shapiro's paper~\cite{Shapiro}. Some of our results
are very similar to 
some of these, but others are distinct. 
The earlier papers~\cite{BGSS} and~\cite{Shapiro} do
not involve automatic coset systems, but their hypotheses are related to
typical properties of automatic cosets systems (such as the subgroup
$H$ being quasiconvex in $G$), and their hypotheses are also related to our
conditions of crossover and stability, which we discuss below.
Another difference between their approach and ours is that
their normal forms are defined using left cosets of the subgroups (so subgroup
elements are situated at the right hand end of normal form words)
whereas ours use right cosets. Since the definition of automaticity involves
multiplying words by generators on the right, this difference is significant.

In Section~\ref{sec:basic} we give definitions and notation of basic
concepts used throughout the article.
The definitions of strong asynchronous and synchronous automatic coset 
systems for a group, subgroup pair $(G,H)$
are given in Section~\ref{subsec:cosetautdef}. In that section
we also prove Theorem~\ref{thm:coset2auto}, which constructs
asynchronous or synchronous automatic structures
for a group $G$, given a strong,  asynchronous or synchronous, automatic
coset system for $(G,H)$ and an automatic structure, asynchronous or synchronous, for $H$. We can combine this result with our combination
theorems for coset automatic systems to derive combination theorems
for automatic structures.

Section~\ref{subsec:xover} introduces the geometric conditions of limited
crossover and stability that are used in our main results,
and studies basic properties of these conditions.

The condition of limited crossover on a language $L^H$ of coset
representatives for $H$ in $G$, with respect to a given generating set $Y$
of $H$ and a generating set $Z$ of a possibly different subgroup of $G$,
limits the $Y$-length of $ugv^{-1}$ as a function of the $Z$-length of $g$
when $g \in \langle Z \rangle$ and $ugv^{-1} \in H$, with $u,v \in L^H$.
In our results about graphs of groups, we need to assume this condition when
$\langle Y \rangle$ and $\langle Z \rangle$ are the edge subgroups of two
edges with the same target vertex.

The condition of stability on an isomorphism $\phi$ between two
groups $H_1$ and $H_2$ relates the lengths of an element $h \in H_1$
and its image $\phi(h)$ over specified generating sets of $H_1$ and $H_2$,
respectively.

Section~\ref{sec:gog} is devoted to our first two main results
for graphs of groups $\GG$, Theorems~\ref{thm:gog_coset} and \ref{thm:gog_coset_synch}. 
These are the most general
combination theorems of this article, building respectively strong 
asynchronous and
strong synchronous coset systems for $\pi_1(\GG)$, given appropriate conditions of
crossover and stabliity on vertex, edge subgroup pairs (and in the second case
some further conditions).
Necessary definitions and
background on
graphs of groups $\GG$ are provided in Section~\ref{subsec:gogbackground},
including a description of Higgins'~\cite{Higgins} 
normal forms for $\pi_1(\GG)$.
Section \ref{subsec:gog_coset} contains the statement of
Theorem~\ref{thm:gog_coset}, together with asynchronous combination results 
Corollary~\ref{cor:amalgprod_coset} 
and Theorem~\ref{thm:amalgprod_coset2} 
specifically for amalgamated products $G_1 \ast_H G_2$,
and Corollary~\ref{cor:hnn_coset} 
specifically for HNN-extensions $G\ast_\phi$,
of strongly asynchronously coset automatic
group, subgroup pairs.
The proof of 
Theorem~\ref{thm:gog_coset} is given in Section~\ref{subsec:gog_coset_proof}. 
In Section~\ref{subsec:synchsub}, we prove 
Proposition~\ref{prop:synchsub} which derives
strong synchronous coset automaticity from its asynchronous form,
given a particular geometric condition.
We apply this to derive Theorem~\ref{thm:gog_coset_synch}, our general closure
result for graphs of groups that are strongly synchronously
coset automatic.

The remainder of the article is devoted to finding applications
of Theorems~\ref{thm:gog_coset} and~\ref{thm:gog_coset_synch}.
For some of these, such as
Theorem~\ref{thm:concat_synch},
some additional technical results are needed to derive them.
In general, we find applications by providing proofs of strong 
asynchronous or synchronous coset automaticity,
as well as crossover conditions and more, for 
various pairs $(G,H)$.
For many of our examples, we derive strong synchronous (rather than asynchronous)
coset systems.

Section~\ref{sec:relhyp} is devoted to relatively hyperbolic groups;
Section~\ref{subsec:relhypbackground}
contains definitions and a number of technical results that
we need.
Given a group $G$  hyperbolic relative to a set of subgroups, 
and a specified such subgroup $H$, the technical
result Proposition~\ref{prop:relhyp} provides
conditions under which we can find a strong synchronous automatic coset system
for $(G,H) $ that satisfies crossover and some other conditions we need.
Theorem~\ref{thm:relhyp_synch} uses a combination of 
Proposition~\ref{prop:relhyp},
and Theorems~\ref{thm:gog_coset_synch} and \ref{thm:coset2auto} to deduce 
strong synchronous coset automaticity relative to any peripheral subgroup 
of the fundamental group of a graph of relatively hyperbolic groups, 
under appropriate conditions on relevant subgroups and coset systems.

Results of Dahmani~\cite[Theorem~0.1]{Dahmani} and 
Antolin and Ciobanu~\cite[Corollary~1.8]{AC}
show that the fundamental group $G$ of an acylindrical 
graph of finitely generated groups
that are hyperbolic relative to
abelian subgroups, in which all of the edge groups are peripheral,  
is automatic;
in Corollary~\ref{cor:gogrelhypaut} we use
Theorem~\ref{thm:relhyp_synch} to give a new proof of this result.
The special case of Corollary~\ref{cor:gogrelhypaut}
in which the graph of groups arises from the
JSJ decomposition of a 3-manifold yields
Corollary~\ref{cor:relhyp3mfd},
which gives new automatic structures for fundamental groups of 
these 3-manifolds with respect to a Higgins language of normal forms
(that is, normal forms derived from the JSJ composition).
These fundamental groups were first shown to be automatic 
by Epstein et al.~\cite[Thm.~12.4.7]{ECHLPT} and Shapiro~\cite{Shapiro}, 
but the structure of the associated languages is not transparent from the proofs;
the results of Dahmani and of 
Antolin and Ciobanu provide a shortlex automatic structure.

Section~\ref{sec:concatenate} considers groups for which
geodesics of a subgroup $H$ of $G$ `concatenate up' to geodesics of $G$;
such a pair ($G,H)$ is
also referred to in the literature~\cite{Alonso, Antolin,chiswellHNN} as an 
`admissible pair'.
We observe in Section~\ref{subsec:coxartin} that this property
holds for Coxeter groups and for Artin groups of sufficiently large type, 
relative to their parabolic subgroups.
This property also holds for graph products of groups, relative to
sub-graph products~\cite{chiswellgp},~\cite[Prop.~14.4]{mann}.
We apply Theorem~\ref{thm:gog_coset} and Proposition~\ref{prop:synchsub} to deduce
Theorem \ref{thm:concat_synch}. We note that we can apply this to prove automaticity
of a variety of examples of Artin groups that were not previously known
to be automatic; a family of such examples is described in
Corollary~\ref{cor:newartin}.

Finally, in Section~\ref{sec:misc} we derive several results
involving abelian and virtually abelian groups.
Proposition~\ref{prop:ab} establishes strong coset automaticity 
with limited crossover in finitely generated abelian groups.
In Proposition~\ref{prop:vab} we prove 
that finitely generated virtually abelian groups $G$ are strongly coset
automatic with respect to any subgroup $H$; however, the
question of whether any crossover conditions hold in this case remains open. 
We also prove assorted results on the strong synchronous coset automaticity
of various types of amalgamated free products $G_1 *_H G_2$ for which
$(G_1,H)$ and $(G_2,H)$ are both strongly coset automatic;
in particular Proposition~\ref{prop:amalg_relhyp_ab} proves the strongly synchronous
coset automaticity of an amalgamated product of a finitely generated
abelian group and a group that is hyperbolic relative to a 
collection of abelian subgroups, where amalgamation is over one of those subgroups.

\subsection{Statement of main results}
Each of these results refers to a graph of groups
$\GG(\lam) = (\lam=(V,\dE), \{ G_v: v \in V \}, \{ G_e: e \in \dE \},
       \{\phi_e \mid e \in \dE\} )$
over a finite connected directed graph 
$\lam$.
(For more details see Definitions~\ref{def:gog} and~\ref{def:pigog} 
of Section~\ref{subsec:gogbackground};
but note in particular that we denote the initial and terminal
vertices of an edge $e \in \dE$ by $\iota(e)$ and $\tau(e)$ respectively
and that for each edge $e \in \dE$, there is an oppositely oriented
edge $\be$ with $\iota(\be) = \tau(e)$ and $\tau(\be)=\iota(e)$.)
We suppose that the groups $G_v$, $G_e$ are finitely generated,
with generating sets $X_v$ and $Y_e$,
respectively.

We refer the reader to Definitions~\ref{def:saca},~\ref{def:crossover},~\ref{def:stable}
for the meanings of strong coset automaticity, limited crossover and stablility, respectively.
\medskip

\noindent{\bf Theorem~\ref{thm:gog_coset}}
\emph{
Let $\GG=\GG(\lam)$ be a graph of groups as above,
and let $e_0$ be an edge of $\lam$.
Suppose that the following conditions hold for each $e \in \dE$.
\begin{mylist}
	\item[(i)]  The pair $(G_{\tau(e)},G_e)$ is strongly asynchronously 
	coset automatic with coset language 
	$L_{\tau(e)}^e \subseteq (X_{\tau(e)}^{\pm})^*$
	containing the empty word.
	\item[(ii)] 
	The triple $(G_e,G_\be,\phi_e)$ is stable 
            with respect to $(Y_e,Y_\be)$.
	\item[(iii)] For each $f \in \dE$ with $\tau(e)=\tau(f)$, 
	the coset language  
	$L_{\tau(e)}^e$ has limited crossover with respect to $(Y_e,Y_f)$.
\end{mylist}
Then the pair $(\pi_1(\GG),G_{e_0})$ is strongly asynchronously coset automatic.\\ 
}

\medskip

\noindent{\bf Theorem~\ref{thm:gog_coset_synch}}
\emph{
Let $\GG=\GG(\lam)$ be a graph of groups as above,
and let $e_0$ be an edge of $\lam$.
Suppose that the following conditions hold for each $e \in \dE$.
\begin{mylist}
        \item[(i)]  $Y_e \subseteq X_{\tau(e)}$.
	\item[(ii)]  The pair $(G_{\tau(e)},G_e)$ is strongly synchronously coset
	automatic with coset language 
        $L_{\tau(e)}^e$
       satisfying
       $
       L_{\tau(e)}^e \subset \Geo(G_{\tau(e)},X_{\tau(e)}) \cap
       [(X_{\tau(e)}^{\pm})^* \setminus Y_e^\pm(X_{\tau(e)}^{\pm})^*],
       $
       the only representative in $L_{\tau(e)}^e$ of the identity coset is $\emptyword$,
       and each element $g \in G_{\tau(e)}$ is represented by a 
       word $y_gz_g \in \Geo(G_{\tau(e)},X_{\tau(e)})$ 
       with $y_g \in (Y_e^\pm)^*$ and $z_g \in L_{\tau(e)}^e$.
	\item[(iii)] 
	The triple $(G_e,G_\be,\phi_e)$ is 1--stable 
         with respect to $(Y_e,Y_\be)$.
	\item[(iv)] For each $f \in \dE$ with $\tau(e)=\tau(f)$, 
	the coset language  
	$L_{\tau(e)}^e$ has limited crossover with respect to $(Y_e,Y_f)$.
\end{mylist}
Then the pair $(\pi_1(\GG),G_{e_0})$ is strongly synchronously coset automatic. \\
}

\medskip

\noindent{\bf Corollary}~\ref{cor:relhyp3mfd}
\emph{
Let $M$ be an orientable,
connected, compact 3-manifold with incompressible toral boundary
whose prime factors have JSJ decompositions containing only hyperbolic pieces.
Then the group $\pi_1(M)$ is automatic, with
respect to a Higgins language of normal forms.\\
}

\medskip

\noindent{\bf Theorem~\ref{thm:concat_synch}}
\emph{
Let $\GG=\GG(\lam)$ be a graph of groups as above.
Suppose that the following conditions hold for each $e \in \dE$.
\begin{mylist}
\item[(i)]  $Y_e \subseteq X_{\tau(e)}$.
\item[(ii)] $\Geo(G_e,Y_e)$ concatenates up to 
$\Geo(G_{\tau(e)},X_{\tau(e)})$.
\item[(iii)] The triple $(G_e,G_\be,\phi_e)$ is 1-stable 
  with respect to $(Y_e,Y_\be)$.
\item[(iv)]  $G_{\tau(e)}$ is shortlex automatic
with respect to an ordering of $X_{\tau(e)}$ in which all letters of $Y_e^\pm$
precede all letters of $X_{\tau(e)}^\pm \setminus Y_e^\pm$.
\end{mylist}
Let $\LL$ be the set of coset languages $\SLex_{G_{\tau(e)}}^{G_e}$,
for $e \in \dE$, and let $\TT$ be any maximal tree in $\lam$.
Then, for each $e_0 \in \dE$, the pair $(\pi_1(\GG),G_{e_0})$ is strongly synchronously coset automatic, 
with the Higgins coset language $L:= L(\GG,\LL,e_0,\TT)$.
Furthermore
$L \subseteq \Geo^{G_{e_0}}$,
and the group $\pi_1(\GG)$ is automatic.\\
}

\medskip

\noindent{\bf Corollary~\ref{cor:newartin}}
\emph{
Let $\artgraph$ be a Coxeter graph of sufficiently large type.
Given arbitrary subgraphs $\Lambda_1,\Lambda_2,\ldots,\Lambda_k$ 
of $\artgraph$, suppose that the
Coxeter graph $\artgraph'$ is formed by adjoining new vertices 
$v_1,v_2,\ldots,v_k$ to $\artgraph$ together with 
the following edges from each $v_i$:
\begin{mylist}
\item[] to each vertex of $\Lambda_i$, with the label $2$;
\item[] to each vertex of $\artgraph \setminus \Lambda_i$, with the label $\infty$;
\item[] to each vertex
$v_j$ with $j \neq i$, with the label $\infty$.
\end{mylist}
Then the Artin group $A_{\artgraph'}$ is automatic. 
}
\subsection*{Acknowledgments}


The first author was partially supported by grants from
the National Science Foundation (DMS-1313559) and
the Simons Foundation (Collaboration Grant number 581433).

\section{Coset Automaticity and Crossover} \label{sec:basic}

\subsection{Notation}
\label{subsec:notation}

Let $G$ be a group.
Throughout this article, all generating sets for all groups
will be assumed to be finite.
Let $X$ be a finite generating set for $G$.
We write $X^\pm$ for $X \cup X^{-1}$.
We denote the length of a word $w \in X^\pm$ by $|w|$.

We denote by $\gam(G)$  (or $\gam(G,X)$ if it is necessary to specify $X$)
the Cayley graph of $G$
and let $d_{\gam(G)}$ be
the path metric in $\gam(G)$.
For each $g \in G$, we denote the length of a shortest word over
$X^\pm$ that represents $g$ by $\elen{g}{X}$, and call that the 
{\em $X$-length} of $g$.
For any $g \in G$ and $w \in X^\pm$, let ${}_g w$ denote
the path in $\gam(G)$ starting at the vertex $g$ and labelled by the
word $w$.

We write $\groupid$ for the identity element of $G$, and $\emptyword$
for the empty word in $(X^\pm)^*$.
For two words $w,x \in (X^\pm)^*$, we write $w=x$ if $w$ and $x$ are
the same word, and $w =_G x$ if $w$ and $x$ represent the same element of $G$.

\subsection{Automatic coset systems and automatic structures}
\label{subsec:cosetautdef}

As before, let $G = \langle X \rangle$ with $|X| < \infty$. 
We define a {\em language for $G$} (over $X$) to be a set of words over $X^\pm$
that contains at least one representative of each element of $G$.
Examples are provided by $\Geo(G,X)=\Geo$, the set of all words $w$ over $X^\pm$
that are minimal length representatives of the elements of $G$ they define, 
and $\SLex \subseteq \Geo$, the set of all words $w$ over $X^\pm$ that are minimal
representatives of the elements they define with respect to the shortlex
ordering (defined using some fixed ordering of $X^\pm$).

Let $H$ be a finitely generated subgroup of $G$.
A {\em coset language} for $(G,H)$ is a set $L^H$ (or $L_G^H$ if
it is necessary to specify $G$) of words over $X^\pm$ that contains 
at least one representative of each right coset $Hg$ of $H$ in $G$.

Examples of coset languages are provided by $\Geo^H$ 
(sometimes called $\Geo^H_G$ or even $\Geo^H_G(X)$), 
the set of all words $w$ over $X$
for which $w$ is of minimal length as a representative of $Hw$,
and $\SLex^H \subseteq \Geo^H$ (sometimes called $\SLex^H_G$ or $\SLex^H_G(X)$), 
the set of all words $w$ over $X$ for which
$w$ is minimal with respect to the shortlex ordering as a representative of
$Hw$ (with respect to some fixed ordering of $X^\pm$). 

Given a word $w$ in $(X^\pm)^*$ and a natural number $t$, let
$w(t)$ denote the element of $G$
represented by the prefix of $w$ of length $t$; in the case that $t > \wlen{w}$,
let $w(t) =_G w$.  Two paths ${}_1w$ and ${}_hw'$
in the Cayley graph $\gam(G,X)$
are said to {\em synchronously $K$-fellow travel}
if for all $t \in \N$ we have $d_{\gam(G)}(w(t),hw'(t)) \le K$.  
The paths ${}_1w$ and ${}_hw'$ are said to {\em asynchronously $K$-fellow
travel} if there exists nondecreasing surjective functions
$\phi_1,\phi_2:\N \rightarrow \N$ such that
for all $t \in \N$ we have $d_{\gam(G)}(w(\phi_1(t)),hw'(\phi_2(t))) \le K$.

\begin{definition}\label{def:saca}
A {\em strong asynchronous automatic coset system}
for $(G,H)$ is defined to be a coset language $L^H \subseteq X^\pm$
together with a constant $K$, such that
\begin{mylist}
\item [(i)] 
$L^H$ is a regular language (that is, the language of a finite state automaton),
\item [(ii)] if $v,w \in L^H$ and $h \in H$ with $d_{\gam(G,X)}(v,hw) \leq 1$,
then the paths ${}_1v$ and ${}_hw$ in $\Gamma(G)$ asynchronously $K$-fellow travel.
(So, in particular, we have $|h|_X \le K$.)
\end{mylist}
If $(G,H)$ has a strong asynchronous automatic coset system as above,
then we say that $(G,H)$ is {\em strongly asynchronously coset automatic}
(or \SACA for short), with
{\em coset language} $L^H$, and {\em fellow traveller constant} $K$.
If the fellow traveller condition above can be replaced by a synchronous 
fellow traveller condition, then we say that $G$ is {\em strongly synchronously
coset automatic} (or \SSCA) or just {\em strongly coset automatic}.
  \end{definition}

We note that our definition of $\SSCA$ matches the definition of coset automaticity 
in \cite{HoltHurt}.  Moreover, in the case that
the subgroup $H$ is the trivial group, the definition of $\SSCA$ is equivalent 
to the definition of automatic~\cite{ECHLPT}. 
We refer the reader to~\cite{ECHLPT} 
or~\cite{HRRbook} for
further information on fellow traveller properties, 
regular languages, finite state automata, 
and automatic groups.

The following result allows us to construct automatic structures for groups
from automatic coset systems.  

\begin{theorem}
\label{thm:coset2auto}
Let $G=\langle X \rangle$ be a group, and $H=\langle Y \rangle$ a subgroup of
$G$. Suppose that $(G,H)$ is strongly asynchronously coset automatic with
language $L^H$ with respect to $X$, and $H$ is asynchronously
automatic with language $L_H$ with respect to $Y$.  
Then
\begin{mylist}
\item[(i)] the group $G$ is asynchronously automatic over $X \cup Y$, with
language $L:= L_HL^H$ (the concatenation of $L_H$ and $L^H)$;
\item[(ii)] if $(G,H)$ is strongly coset automatic and 
$H$ is automatic (that is, both structures are synchronous), 
then $G$ is automatic.
Furthermore, if $L_H$ and $L^H$ are both synchronous structures and
$L_H$ contains only finitely many representatives of each element
of $H$, then the language $L:= L_HL^H$ is the language of a synchronous
automatic structure for $G$. 
\end{mylist}
\end{theorem}

Note that we do not require $X$ and $Y$ to be disjoint.
\begin{proof}
Suppose first that $L_H$ and $L^H$ satisfy the 
asynchronous $K$--fellow traveller property.
Since regular languages are closed under concatenation,
the language $L$ is regular.
We shall verify an asynchronous fellow traveller property for $L$
with constant $\kappa K^2$, where $\kappa$ is the maximum $X$-length of any
$y \in Y$.

Suppose that $wv$, $w'v'$ are words in $L$ with $w,w' \in L_H$ and
$v,v' \in L^H$. First suppose that $wv=_G w'v'$.
In this case $v$ and $v'$ represent the same coset of $H$ in $G$, and so
pair of paths ${}_1 v$ and ${}_{w^{-1}w'} v'$, and hence also the pair
${}_w v$ and ${}_{w'} v'$, asynchronously $K$-fellow travel. In particular
$d_{\Gamma(G)}(w,w') \leq K$, and hence, applying the fellow property for $L_H$
we deduce that $w,w'$ asynchronously fellow travel at distance $K^2$.

Similarly, if $x \in X$ and $wvx=_G w'v'$,
then the same argument shows that the paths 
${}_wv$ and $_{w'}v'$ $K$-fellow travel, 
and $w$ and $w'$ $K^2$-fellow travel. 

Finally, suppose that $y \in Y$, and that $wvy=_G w'v'$. 
Writing $y=x_1\cdots x_{\kappa}$ with each $x_i \in X^\pm$,
we have that the paths $_wv$ and $_{w'}v'$ asynchronously 
$\kappa K$--fellow travel, and so
$w$ and $w'$ fellow travel at distance $\kappa K^2$.

In all cases the paths $_1w$ and $_1w'$ asynchronously 
$\kappa K^2$--fellow travel, and the paths $_wv$ and $_{w'}v'$ 
asynchronously $\kappa K$--fellow travel. 
Thus, the paths $_1wv$ and $_1w'v'$ asynchronously 
$\kappa K^2$--fellow travel, as desired.
This proves (i).

To prove (ii),
suppose that $L_H$ and $L^H$ are the languages of synchronous structures.
If $L_H$ contains infinitely many representatives 
of some elements of $H$ then, by~\cite[Thm~2.5.1]{ECHLPT}, we can replace
it by a language consisting of unique representatives. So, in proving
the first assertion of (ii), we may
assume that $L_H$ contains only finitely many such representatives.
Let $K$ be a synchronous fellow traveller constant for both structures.
Then by~\cite[Thm 2.3.9]{ECHLPT}, there is a constant $N$ such
that whenever $u,u' \in L_H$ and the paths ${}_1u$ and ${}_1u'$ end
a distance at most 1 apart, the lengths of the words
$u$ and $u'$ differ by at most $N$.

Let $w, w'\in  L_H$ and $v, v'\in L^H$ satisfy 
$wvx =_G w'v'$ for some $x\in X\cup Y\cup\{\emptyword\}$. 
As in the proof above, we see that the paths ${}_wv$, ${}_{w'}v'$ synchronously 
$\kappa K$--fellow travel and so the paths ${}_1w$ and ${}_1w'$ end 
at distance at most $\kappa K$ apart in $\gam(G,X)$.  Hence
the paths ${}_1w$ and ${}_1w'$ synchronously $\kappa K^2$--fellow travel,
and their lengths differ by at most $\kappa K N$. 
Thus, the paths ${}_1wv$ and ${}_1w'v'$ synchronously 
$\kappa K^2+\kappa K N$--fellow travel.
\end{proof}

\subsection{Crossover and stability}\label{subsec:xover}

The properties of crossover and stability for coset systems are fundamental
for us in order to state and prove the results of Section~\ref{sec:gog}.

\begin{definition}\label{def:crossover} 
Let $Y$ and $Z$ be finite subsets of a group $G$, and let $H=\langle Y \rangle$.
Let $1 \leq \lambda \in \N$.
We say that the coset language $L^H$ for $(G,H)$ has 
{\em $\lambda$-limited crossover}
with respect to $(Y,Z)$ if,
for any $g \in \langle Z \rangle$ with $|g|_{Z} \leq \lambda$, and any 
$u,v \in L^H$  with $ug \in Hv$, 
we have $|ugv^{-1}|_{Y} \leq \lambda$.
We say that $L^H$ has {\em limited crossover} with respect to $(Y,Z)$ 
if it has $\lambda$-limited crossover for some $\lambda$.
If $Y=Z$, we use the term limited crossover {\em with respect to $Y$}.
\end{definition}

As we shall show in Sections~\ref{sec:relhyp},~\ref{sec:concatenate}, 
and~\ref{sec:misc}, 
the limited crossover property is satisfied, for example, by Coxeter
groups and by Artin groups of large type where $Y$ and $Z$ are arbitrary
subsets of the standard generating sets, by finitely generated abelian groups
with $Y=Z$ an arbitrary finite subset, and by groups that are relatively hyperbolic
with respect to virtually abelian or hyperbolic parabolic subgroups.

The following result will be useful for finding common crossover constants
for a collection of languages and generating sets.

\begin{lemma}
\label{lem:multiply_crossover}
Suppose that for finite subsets $Y,Z$ of a group $G=\langle X \rangle$, 
$H=\langle Y \rangle$,
$1 \leq \lambda \in \N$, and
coset language $L^{H} \subseteq (X^\pm)^*$ for $(G,H)$, the set
$L^H$ has $\lambda$-limited 
crossover with respect to $(Y,Z)$. Then 
$L^H$ has $k\lambda$-limited crossover with respect to $(Y,Z)$, for
any $k \in \N$.
\end{lemma}

\begin{proof}
Suppose that $g\in \langle Z \rangle$ with $|g|_Z \leq k\lambda$ and that
$u,v \in L^H$ satisfy $ug \in Hv$. Then we can decompose $g$ as a product 
$g=g_1\cdots g_k$ of elements $g_i \in \langle Z \rangle$, with each
$|g_i|_Z \leq\lambda$.  We choose $v_1,\ldots,v_k \in L^H$ such that, for each
$i$, $ug_1\cdots g_i \in Hv_i$; in particular we choose $v_k=v$.
Then we have
$ug_1 \in Hv_1$, and for each $i=2,\ldots,k$ $v_{i-1}g_i \in Hv_i$.
We deduce from the limited crossover condition that 
$ug_1v^{-1}$ and all of the elements $v_{i-1}g_iv_i^{-1}$ for $2 \leq i \leq
k$ have $Y$-length at most $\lambda$. The product of these elements has
$Y$-length at most $k\lambda$ and is equal to $ugv^{-1}$.
\end{proof}

The following stronger version of crossover leads to a variant of the result of
Corollary~\ref{cor:amalgprod_coset} for amalgamated products, proved in 
Theorem~\ref{thm:amalgprod_coset2}.

\begin{definition}\label{def:maxxover}
Let $Y,Z$ be finite subsets of $G$, $H=\langle Y \rangle$, and $1 \leq \lambda \in \N$.
We say that a coset language $L^H$ for $(G,H)$ has 
{\em $\lambda$-maximal crossover}
with respect to $(Y,Z)$ if, for any $g \in \langle Z \rangle$ and any
$u,v \in L^H$  with
$u \not\in H$ and $ug \in Hv$, we have $|ugv^{-1}|_Y \leq \lambda$.
If $Y=Z$, then we use the term $\lambda$-maximal crossover {\em with respect to} $Y$.
We say that $L^H$ has {\em maximal crossover} with respect to the subgroup
$\langle Z \rangle$ of $G$ 
if $L^H$ has $\lambda$-maximal crossover with respect to
$(Y',Z)$ for some generating set $Y'$ of $H$ and $1 \le \lambda \in N$.
\end{definition}

We note that the maximal crossover property would not hold in the case 
that $H$ and $G/H$ are infinite
if we did not impose 
the condition $u\not\in H$, 
but we do not need that condition in the definition of limited crossover.
It is straightforward to show that if the maximal crossover property holds for some
finite generating set $Y$ of $H$, then it holds (but probably with a different 
parameter $\lambda$) for any other finite generating set. Further,
for some finite generating set of $H$, $L^H$ has 1-maximal crossover.

This stronger property of maximal crossover is unusual but, 
as we shall show in Section~\ref{subsec:relhypxover}, it holds for groups
that are hyperbolic relative to a collection of finitely generated groups that
are either virtually abelian or hyperbolic, where $Y$ and $Z$ are suitably
chosen generating sets of two of the parabolic subgroups.

\begin{definition}\label{def:stable}
 Suppose that $H_1=\langle Y_1\rangle$ and $H_2=\langle Y_2 \rangle$
are isomorphic groups with isomorphism $\phi:H_1 \to H_2$.
We say that $(H_1,H_2,\phi)$ is \emph{$\mu$-stable} with respect to $(Y_1,Y_2)$ 
if, whenever $h \in H_1$ with
$|h|_{Y_1}\leq \mu$, we have $|\phi(h)|_{Y_2} \leq \mu$.
Provided that each of $H_1,H_2$ is associated with just one generating set, we may
omit the phrase `with respect to $(Y_1,Y_2)$', and in general we shall do that.
We say that $(H_1,H_2,\phi)$ is \emph{stable} if it is $\mu$-stable for some $\mu$.
\end{definition}

We have the following results.

\begin{lemma}
\label{lem:multiply_stability}
For groups $H_1=\langle Y_1 \rangle, H_2=\langle Y_2 \rangle$, related by an isomorphism $\phi:H_1 \rightarrow H_2$ of $G$,
if  $(H_1,H_2,\phi)$ is $\mu$-stable, then it
is also $k\mu$-stable for any $k \in \N$.
\end{lemma}

We omit the proof of this, which is nearly identical 
to the proof of Lemma~\ref{lem:multiply_crossover}.

\begin{lemma} \label{lem:basic_properties}
\begin{mylist}
\item[(i)]
Let $X_1$ and $X_2$ be finite generating sets for a group $G$, and
let $L_1^H$ be a coset language for the pair $(G,H)$ with respect to $X_1$.
Then there is a coset language $L_2^H$ for $(G,H)$ with
respect to $X_2$, such that the subset of $G$ represented by the words in
$L_2^H$ is the same as for $L_1^H$ and such that, if any of the properties
\SACA, \SSCA, limited crossover, or maximal crossover
hold in $L_1^H$, then they also hold in $L_2^H$.
\item[(ii)] If $H \le F \le G$ with $|G:F|$ finite, and $(G,H)$ has
a coset language $L_G^H$ satisfying any of	
\SACA, \SSCA, or 
limited crossover or maximal crossover
relative to a pair of finite generating sets for $H$,
then there is a coset language $L_F^H$ for $(F,H)$ with the same properties.
Furthermore, we can choose $L_F^H$ such that the subset of the group
that it represents is the intersection with $F$ of the subset represented
by $L_G^H$.
\item[(iii)] If $H \le G \le F$ with $|F:G|$ finite and 
$(G,H)$ is strongly asynchronously (resp. synchronously) coset automatic, 
then so is $(F,H)$.
\end{mylist}
\end{lemma}

We note that in (iii) it is not clear that either of the
crossover properties is preserved.

\begin{proof}
(i) Let $\pi_1\colon X_1^*\to G$ and $\pi_2\colon X_2^*\to G$ be 
the natural projection maps.
For each generator $x\in X_1$ choose 
a word $w(x)\in X_2^*$ so that $\pi_1(x)=\pi_2(w(x))$, and let 
let $\varphi\colon X_1^*\to X_2^*$ be the corresponding 
semigroup homomorphism.
Then we imitate the construction of $L_2^H$ from $\varphi(L_1^H)$
exactly as in \cite[Theorem 2.4.1]{ECHLPT} for automatic structures,
and then, by the same argument as in \cite{ECHLPT},
$L_2^H$ is $\SACA$ or $\SSCA$ if $L_1^H$ is.
Furthermore, we have $\pi_2(L^H_2)=\pi_1(L^H_1)$. 

Now let $Y\subseteq H$ be a finite generating set for $H$, 
and $Z\subseteq G$ be any finite subset. 
Suppose that $L^H_1\subseteq X_1^*$ 
has $\lambda$--limited (resp. $\lambda$--maximal) crossover 
with respect to $(Y,Z)$,
Since both the limited and maximal crossover properties 
depend only on the image of the language in $G$ and the generating set 
of $H$, it follows that $L^H_2$ has $\lambda$--limited
(resp. $\lambda$--maximal) 
crossover with respect to $(Y,Z)$, as desired.
	
We omit the proofs of (ii) and (iii), which are straightforward adaptations
of the proof of \cite[Theorem 4.1.4]{ECHLPT}.
\end{proof}

Next we note that any coset
language containing the empty word can have all other
representatives of the same coset removed,
without altering \SACA or crossover conditions.

\begin{lemma}\label{lem:sacaxoverlanguagechange}
Suppose that the group $G=\langle X \rangle$ is \SACA with
language $L^H$ containing the empty word $\emptyword$,
with respect to a finitely generated subgroup 
$H = \langle Y \rangle$.
Suppose also that 
$Z_1,...,Z_m$ are finitely many finite subsets of $G$
such that $L^H$ has limited crossover with respect to $(Y,Z_i)$
for all $i$.  Let $\widetilde L^H := L^H \setminus S$, where $S$
is the set of nonempty words in $L^H$ that represent the identity
coset $H$ in $G$.
Then the pair $(G,H)$ is \SACA with language
$\widetilde L^H$,
and $\widetilde L^H$ has limited crossover with respect to $(Y,Z_i)$
for all $i$ as well.
\end{lemma}

\section{Automatic structures for graphs of groups}
\label{sec:gog}

Our goal in this section is to prove 
Theorem~\ref{thm:gog_coset},
that free products with amalgamation,
HNN extensions, and more generally fundamental groups of graphs of groups of 
asynchronously automatic groups with well-behaved coset automatic structures,
are also asynchronously automatic; the proof is given in
Section~\ref{subsec:gog_coset_proof}.
The resulting structure is asynchronous, but under certain circumstances a strong
asynchronous coset system contains a synchronous system as a substructure,
as is proved in
Proposition~\ref{prop:synchsub}. We apply the proposition to deduce a 
synchronous closure result Theorem~\ref{thm:gog_coset_synch} for graphs of groups
with particular conditions on associated coset automatic structures.

We begin this section with definitions and notation for
graphs of groups.

\subsection{Background on graphs of groups and Higgins normal forms}\label{subsec:gogbackground}

For a directed graph $\lam$ with vertex set $V$ and directed edge
set $\dE$ (written $\lam=(V,\dE)$), 
we denote the initial and terminal
vertices of an edge $e \in \dE$ by $\iota(e)$ and $\tau(e)$ respectively.
We assume that associated with each edge $e \in \dE$, 
there is an oppositely oriented
edge $\be$ with $\iota(\be) = \tau(e)$ and $\tau(\be)=\iota(e)$.
We define $P(\lam)$ to be the set of directed paths of $\lam$, and 
where $p=e_1\cdots e_k \in P(\lam)$, we define $\iota(p)=\iota(e_1)$, $\tau(p)=\tau(e_k)$.

\begin{definition}\label{def:gog}
A \emph{graph of groups} is a quadruple 
$\GG=(\lam, \{G_v \mid v \in V\}, \{G_e \mid e \in \dE\}, 
\{\phi_e \mid e \in \dE\})$, 
where $\lam=(V,\dE)$ is a directed graph, each $G_v$ is a group,
and for each $e \in \dE$, $G_e$ is a subgroup of $G_{\tau(e)}$,
$\phi_e:G_e \rightarrow G_\be$ is an isomorphism, and $\phi_e^{-1}=\phi_\be$.

We call the subgroups $G_v$, $G_e$ the \emph{vertex} and \emph{edge} groups of $\GG$, respectively.
\end{definition}

Following standard practice,
we assume that the $G_v$ have pairwise trivial intersections.
Whenever we refer to a graph of groups $\GG=\GG(\lam)$ in this article
we will use the notation of Definition~\ref{def:gog}. In addition we will 
use the notation $X_v$, $Y_e$ for specified generating sets for $G_v$, $G_e$,
respectively.

\begin{definition}\label{def:pigog}
Let $\GG=(\lam, \{G_v \mid v \in V\}, \{G_e \mid e \in \dE\}, 
  \{\phi_e \mid e \in \dE\})$
be a graph of groups with connected graph $\lam$, 
and let $\TT$ be a maximal tree in $\lam$.
The {\em fundamental group of $\GG$ at $\TT$}, denoted 
$\pi_1(\GG)=\pi_1(\GG,\TT)$, is the group generated by the disjoint union of all of 
the groups $G_v$ and the set  of symbols
$\{s_e \mid e \in \dE\}$, subject to the relations:
\begin{mylist}
\item[(i)] $s_\be= s_e^{-1}$ for all $e \in \dE$,
\item[(ii)] $s_e = 1$ for all directed edges $e$ in $\TT$, and
\item[(iii)] $s_egs_e^{-1} = \phi_e(g)$ for all $e \in \dE$ and $g \in G_e$.
\end{mylist}
\end{definition}

When $\lam$ consists of two vertices joined by an edge,
or of a single vertex together with a loop, then the fundamental group is
a free product with amalgamation or HNN-extension, respectively.
We refer the reader to~\cite{ScottWall,Serre} for
basic facts about graphs of groups. 
In particular, up to isomorphism,
$\pi_1(\GG,\TT)$ does not depend on the choice of the
maximal tree $\TT$.

Now we provide a description of the language for $\pi_1(\GG)$
that we use in our proof of Theorem~\ref{thm:gog_coset}.
This is a set of words representing 
normal forms provided by Higgins in~\cite{Higgins},
but modified to work with right rather than left cosets,
and to provide words over generating sets rather than normal forms that are
products of elements; see~\cite[Prop.~3.3]{BHS}
for more details.

For each $v \in V$, let $X_v$ be a finite generating set
for the vertex group $G_v$; 
note that the $X_v$ are pairwise disjoint.
We consider the generating sets 
$$
\widehat X := \bigcup_{v \in V} X_v \cup \{s_e: e \in \dE\}
\hspace{.2in} \mbox{ and } \hspace{.2in}  
X := \bigcup_{v \in V} X_v \cup \{s_e: e \in \dE \setminus \dE_\TT\}
$$
for $\pi_1(\GG)$.
For any product $w \in (\widehat X^\pm)^*$,
we define its {\em deflation} $\Deflation{w}{\TT} \in (X^\pm)^*$  
to be the word derived from $w$ by
omitting from it every $s_e$ for which $e$ is in the 
set $\dE_\TT$ of directed edges of the tree $\TT$.
 
Choose a vertex $v_0 \in V$, and let $L_{v_0} \subset (X_{v_0}^\pm)^*$ 
be a language for $G_{v_0}$.
For each $e \in \dE$, let $L_{\tau(e)}^e\subset (X_{\tau(e)}^\pm)^*$ 
be a coset language for
$(G_{\tau(e)},G_e)$, and
let $\LL$ be the collection of languages 
$ \{L_{\tau(e)}^e\}$.  

Now define
$\widehat L(\GG,\LL,L_{v_0},v_0,\TT) \subset (\widehat X^\pm)^*$ 
to be the set of all words of the form:
\[ 
w_0s_{e_1}u_1\cdots s_{e_k}u_k,
\]
where
\begin{mylist}
\item[(i)]
$p=e_1\cdots e_k \in P(\lam)$ with $\iota(p)=v_0$;
\item[(ii)]
$w_0 \in L_{v_0}$ and $u_i \in L_{\tau(e_i)}^{e_i}$ for $1 \le i \le k$;
\item[(iii)]
if $e_{i+1}=\be_i$, then $u_i$ does not represent an element of $G_{e_i}$;
\item[(iv)]
if $k > 0$ and $e_k \in \dE_\TT$,
then $u_k$ does not represent an element of $G_{e_k}$.
\end{mylist}
Then every element of $\pi_1(\GG)$ has at least one representative of this form.
We refer to this set as the {\em  inflated Higgins language}
for the group $\pi_1(\GG,\TT)$
with respect to the triple $(\LL,v_0,L_{v_0})$.
The language 
$$
L(\GG,\LL,L_{v_0},v_0,\TT) := 
\{\Deflation{w}{\TT} \mid w \in \widehat L(\GG,\LL,L_{v_0},v_0,\TT)\}
$$ over $X$ is the associated {\em Higgins language} for $\pi_1(\GG)$.

Next suppose that $e_0$ is any directed edge in $\lam$.
We define $\widehat L(\GG,\LL,e_0,\TT)$ to be the of all words
over $\widehat X$ of the from
of the form 
\[ 
u_0s_{e_1}u_1\cdots s_{e_k}u_k 
\]
where $u_0 \in  L_{\tau(e_0)}^{e_0}$ and the conditions (i)--(iv) above hold
with $v_0=\tau(e_0)$.
Then each coset of $G_{e_0}$ in $\pi_1(\GG)$ has at least one representative
in this language.
we call this the {\em inflated Higgins coset language}
for the pair $(\pi_1(\GG,\TT),G_{e_0})$ with respect to $(\LL,e_0)$.
Similarly, the language 
$$
L(\GG,\LL,e_0,\TT) := 
\{\Deflation{w}{\TT} \mid w \in \widehat L(\GG,\LL,e_0,\TT)\}
$$ over $X$
is the associated {\em Higgins coset language} for $(\pi_1(\GG),G_{e_0})$.

\begin{remark}\label{rmk:higginsnormalform}
We note that in the case when $L_{v_0}$ and $L^e_{\tau(e)}$
are languages
with unique representatives of $G_{v_0}$ and of cosets of $G_e$
in $G_{\tau(e)}$, respectively,
then the corresponding Higgins languages are sets
with unique representatives of $\pi_1(\GG)$
and cosets of $G_{e_0}$ in $\pi_1(\GG)$, respectively.
\end{remark}

\subsection{Strong asynchronous automatic coset systems for graphs of groups}
\label{subsec:gog_coset}

This section is devoted to the
statement of our main graph of groups result
Theorem~\ref{thm:gog_coset}, together with
Corollaries~\ref{cor:amalgprod_coset} and \ref{cor:hnn_coset},
for amalgamated products and HNN extensions,
which are special cases of this.
We defer the proof of the theorem to the following section,
Section~\ref{subsec:gog_coset_proof}.
We conclude this section with a variant of the result for amalgamated free products, 
in the case that one group has maximal crossover.

\begin{theoremA}
\label{thm:gog_coset}
Let $\GG = (\lam=(V,\dE), \{ G_v: v \in V \}, \{ G_e: e \in \dE \},
       \{\phi_e \mid e \in \dE\} )$
be a graph of groups over a finite connected directed graph 
$\lam$ with an edge $e_0$.
Let $X_v$ and $Y_e$ be finite generating sets of the groups
$G_v$ and $G_e$, respectively.
Suppose that the
following conditions hold for each $e \in \dE$.
\begin{mylist}
	\item[(i)]  The pair $(G_{\tau(e)},G_e)$ is strongly asynchronously 
	coset automatic with coset language 
	$L_{\tau(e)}^e \subseteq (X_{\tau(e)}^{\pm})^*$
	containing the empty word $\emptyword$.
	\item[(ii)] 
	The triple $(G_e,G_\be,\phi_e)$ is stable 
            with respect to $(Y_e,Y_\be)$.
	\item[(iii)] For each $f \in \dE$ with $\tau(e)=\tau(f)$, 
	the coset language  
	$L_{\tau(e)}^e$ has limited crossover with respect to $(Y_e,Y_f)$.
\end{mylist}
Then the pair $(\pi_1(\GG),G_{e_0})$ is strongly asynchronously coset automatic. 
\end{theoremA}

In the special cases of amalgamated products and HNN extensions,
we immediately obtain the following results as corollaries of the
above result together with
Theorem~\ref{thm:coset2auto}.

\begin{corollary}[{\bf Amalgamated products}]
\label{cor:amalgprod_coset}
Let $G_1$ and $G_2$ be groups with a common subgroup
$H = G_1 \cap G_2$, and suppose that $G_1$, $G_2$ and $H$ are all
finitely generated.
Suppose that the pairs $(G_1,H)$ and $(G_2,H)$ are both 
strongly asynchronously coset automatic and that, 
for some finite generating set $Y$ of $H$,
each of the associated coset languages has
limited crossover with respect to $(Y,Y)$.
Then $(G_1 \ast_H G_2,H)$ has a strong asynchronous automatic
coset system.  Moreover, if the group $H$ is asynchronously
automatic, then so is the amalgamated product $G_1 \ast_H G_2$.
\end{corollary}

\begin{corollary}[{\bf HNN extensions}]
\label{cor:hnn_coset}
Let $G$, $H_1$, $H_2$ be finitely generated with $H_1, H_2\le G$ and 
$\phi\colon H_1\to H_2$ an isomorphism. Further, let
$H_i=\gspan{\Yi}$ for $|Y_i|<\infty$, for $i=1,2$. 
Suppose that:
\begin{mylist}
\item[(i)] the pairs $(G, H_j)$ are strongly asynchronously coset automatic with
coset language $L^{H_j}$ for $j=1,2$;
\item[(ii)] $L^{H_j}$ has 
limited crossover with respect to
$(Y_i,Y_j)$ for each of $i,j$ in $\{1,2\}$; and
\item[(iii)] the triples $(H_1, H_2, \phi)$ %
and $(H_2,H_1,\phi^{-1})$ 
are stable.
\end{mylist}
Then $(G\ast_\phi, H_i)$ is strongly asynchronously coset automatic for
$i=1, 2$.
Moreover, if the (isomorphic) groups $H_i$ are asynchronously
automatic, then so is the HNN extension $G_1 \ast_\phi$.
\end{corollary}

In the presence of maximal crossover, the following variation of
Theorem~\ref{thm:gog_coset} holds for amalgamated free products.
The proof is analogous to the proof of
Theorem~\ref{thm:gog_coset} in Section~\ref{subsec:gog_coset_proof},
although maximal crossover
allows us to simplify the argument somewhat.

\begin{theorem}\label{thm:amalgprod_coset2}
    Let $H \le G_1 \cap G_2$ be  a finitely
    generated group within the intersection of groups
    $G_1=\langle X_1\rangle$, $G_2 = \langle X_2 \rangle$,
    Suppose that 
    $(G_1,H)$ and $(G_2,H)$ are both strongly asynchronously coset automatic,
    and that the language for $(G_1,H)$ has maximal crossover.
    Then $(G_1 \ast_H G_2,H)$ is strongly asynchronously coset automatic.
\end{theorem}

One situation in which we could apply this result is when
$G_1$ is hyperbolic relative to a collection of virtually abelian groups
with $H$ a peripheral subgroup (cf Proposition~\ref{prop:relhyp}) and
when $H$ is an arbitrary subgroup of the virtually abelian group $G_2$
(cf Proposition~\ref{prop:vab}).

\subsection{Proving Theorem~\ref{thm:gog_coset}}
\label{subsec:gog_coset_proof}

In order to prove Theorem~\ref{thm:gog_coset}, we need to 
define a procedure that we call {\em cascading}, that will
convert a given word
of a particular form over $\widehat X^\pm$
into another word representing 
the same group element, which (as we shall show in 
Lemma~\ref{lem:higginspath}) is in the
inflated Higgins language.

\begin{definition}\label{def:cascade}
Let $\GG=(\lam=(V,\dE),\{G_v=\langle X_v\rangle\},\{G_e\},\{\phi_e\})$ 
be a graph of groups,
and let $v_0 \in V$.  

Let $L_{v_0} \subset (X_v^\pm)^*$ be a language for $G_{v_0}$
and let $\LL=\{L_{\tau(e)}^e: e \in \dE\}$ be a set of coset languages for the pairs
$(G_{\tau(e)},G_e)$, for which each
$L_{\tau(e)}^e$ is a language over $X_{\tau(e)}$, and
for which
the only representative in each
$L_{\tau(e)}^e$ of the identity coset $G_e$ is the empty word.

Let $w =w_0s_{e_1}w_1\cdots s_{e_k}w_k \in (\widehat X^\pm)^*$,
where $p=e_1 \cdots e_k \in P(\lam)$ with $\iota(p)=v_0$, 
$w_0 \in (X_{v_0}^\pm)^*$,
and $w_i \in (X_{\tau(e_i)}^\pm)^*$ for $i=1,\ldots,k$.

An {\em $(\LL,L_{v_0})$--cascade} of $w$ is a word 
$u \in (\widehat X^\pm)^*$  
satisfying $u =_{\pi_1(\GG)} w$
that is obtained as follows.  
\begin{mylist}
\item[(i)]
Select $u_k \in L_{\tau(e_k)}^{e_k}$ 
with $w_k =_{G_{\tau(e_k)}} h_k u_k$ for some $h_k \in G_{e_k}$.
\item[(ii)] For $j=k-1,\ldots,1$, 
select $u_j \in L_{\tau(e_j}^{e_j}$, 
with
$w_j\phi_{e_{j+1}}(h_{j+1}) =_{G_{\tau(e_j)}} h_j u_j$ 
for some $h_j \in G_{e_j}$. 
\item[(iii)]
Select $u_0 \in L_{v_0}$
representing the element $w_0\phi_{e_1}(h_1)$ in $G_{v_0}$.
\item[(iv)]
Remove from $u_0s_{e_1}u_1\cdots s_{e_k}u_k$
the maximal suffix of the form $s_{e_j}s_{e_{j+1}} \cdots s_{e_k}$
for which $e_i \in \dE_\TT$ for all $j \le i \le k$, to obtain $u$.
\end{mylist}
\end{definition}

The proof of the following lemma 
is basically the proof in~\cite[Proposition 3.4]{BHS}.

\begin{lemma}\label{lem:higginspath}
Let $\GG$ be a graph of groups over a finite connected graph $\lam=(V,\dE)$,
and assume the notation of Section~\ref{subsec:gogbackground}.
Let $\widehat L =\widehat L(\GG,\LL,L_{v_0},v_0,\TT)$ be an inflated Higgins language 
over $\widehat X = \cup_{v \in V} X_v \cup \{s_e: e \in \dE\}$ 
for which 
the only representative in each
$L_{\tau(e)}^e$ of the identity coset $G_e$ is the empty word.
Let  $w=w_0s_{e_1}w_1\cdots s_{e_k}w_k$
be a word over $\widehat X$ as in Definition~\ref{def:cascade},
and suppose that for all $1 \le i \le k-1$ either $e_{i+1} \neq \be_i$ or $w_i$
does not represent an element of $G_{e_i}$.
\begin{mylist}
\item[(i)] If $u$ is an $(\LL,L_{v_0})$--cascade of $w$, 
then $u \in \widehat L$.  
\item[(ii)]  Suppose that if both $k>0$ and  $e_k \in \dE_\TT$ 
then $w_k$ does not represent
an element of $G_{e_k}$.  
Then any $w' \in \widehat L$ with $w' =_{\pi_1(\GG)} w$
is an $(\LL,L_{v_0})$--cascade of $w$ 
of the form
$w'=w_0's_{e_1}w_1' \cdots s_{e_k}w_k'$; that is, 
the paths in $\lam$ associated with $w$ and $w'$ are the same, and
there exist elements $h_i \in G_{e_i}$ for $1 \le i \le k$ such that
$w_k =_{G_{\tau(e_k)}}  h_kw_k'$,
and for $k> i \geq 1$,
$w_i\phi_{e_{i+1}}(h_{i+1}) =_{G_{\tau(e_i)}} h_i w'_i$,
and $w_0\phi_{e_1}(h_1) =_{G_{v_0}} w_0'$.
\end{mylist}
\end{lemma}

\begin{proof}
An $(\LL,L_{v_0})$--cascade word $u$ of $w$ is obtained by removing 
a suffix $s_{e_j}s_{e_{j+1}} \cdots s_{e_k}$ of letters
corresponding to edges in the tree $\TT$
from a word of the form
$u_0s_{e_1}u_1\cdots s_{e_k}u_k$.
For each index $\ell$, we have
$w_{\ell} =_{G_{\tau(e_{\ell})}}
 h_{\ell} u_{\ell}\phi_{e_{\ell+1}}(h_{\ell+1})^{-1}$ and,
if  $e_{\ell+1} = \be_\ell$, then
$h_{\ell},\phi_{e_{\ell+1}}(h_{\ell+1}) \in G_{e_{\ell}}$ and
$w_{\ell}$ does not represent an element of $G_{e_\ell}$,
so the word $u_\ell$ also cannot represent an element of $G_{e_\ell}$.
Moreover, since the only representative of the identity coset
in each of the coset languages is the empty word, then 
after removing the maximal suffix of letters associated with edges
in $\TT$ from the word $u_0s_{e_1}u_1\cdots s_{e_k}u_k$,
either the resulting word $u$
is in $L_{v_0}$, or $u$ ends with a letter $s_{e_{j-1}}$ 
with $e_{j-1} \notin \dE_\TT$, or $u$ ends with a word $s_{e_{j-1}}u_{j-1}$ 
in which $u_{j-1}$ does not represent the identity coset.
Hence $u$ is in the inflated Higgins set $\widehat L$.

Suppose further that the additional hypothesis of (ii) holds.
If $k>0$ and $e_k \in \dE_\TT$ then,
since $w_k =_{G_{\tau(e_k)}} h_k u_k$ with $h_k \in G_{e_k}$,
and $w_k$ does not represent an element of $G_{e_k}$, 
again we see that the word $u_k$ cannot represent an element of $G_{e_k}$.
Hence in any case no suffix of letters is removed in the last step of 
the cascade procedure,
and the $(\LL,L_{v_0})$--cascade $u$ of $w$ has the form
$u=u_0s_{e_1}u_1\cdots s_{e_k}u_k$ with the same associated path
in $\lam$ as $w$.

Now let $\widetilde L_{v_0} \subseteq (X_{v_0}^\pm)^*$ 
be a set of unique representatives for the elements
of $G_{v_0}$ containing the empty word $\emptyword$, and
for each $e \in \dE$ let $\widetilde L^e \subseteq L_{\tau(e)}^e$ 
be a set of unique
representatives of the right cosets of $G_e$ in $G_{\tau(e)}$,
containing $\emptyword$.
Let $\widetilde L := \widehat L(\GG,\{\widetilde L^e\},\widetilde L_{v_0},v_0,\TT)$
be the associated inflated Higgins language.
By Remark~\ref{rmk:higginsnormalform}, each element of $\pi_1(\GG)$
is represented by a unique element of $\widetilde L$.

Let $\widetilde w$ and $\widetilde w'$ be 
$(\{\widetilde L^e\},\widetilde L_{v_0})$--cascades of
$w$ and $w'$, respectively.  
The words $w$ and $w'$ both satisfy the hypotheses in (ii), and
so the proof above shows that $\widetilde w,\widetilde w' \in \widetilde L$.
Now $\widetilde w =_{\pi_1(\GG)} w =_{\pi_1(\GG)} w' =_{\pi_1(\GG)} \widetilde w'$,
and so the uniqueness of representatives in $\widetilde L$ implies that
$\widetilde w$ and $\widetilde w'$ are the same word over $\widetilde X$.
Moreover, the argument above shows that the paths in $\lam$ associated with $w$,
$\widetilde w$, and $w'$ are the same. In particular, we can write
$\widetilde w = \tilde w_0s_{e_1}\tilde w_1\cdots s_{e_k}\tilde w_k$,
and there are elements 
$\tilde h_i,h_i' \in G_{e_i}$ for $1 \le i \le k$ such that
$w_k =_{G_{\tau(e_k)}} \tilde h_k\tilde w_k$,
$w_k' =_{G_{\tau(e_k)}}  h_k'\tilde w_k$,
and for $k>i \geq 1$,
$w_i\phi_{e_{i+1}}(\tilde h_{i+1}) =_{G_{\tau(e_i)}} 
    \tilde h_i \tilde w_i$,
$w_i'\phi_{e_{i+1}}(h_{i+1}') =_{G_{\tau(e_i)}} h_i' \tilde w_i$,
$w_0\phi_{e_1}(\tilde h_1) =_{G_{v_0}} \tilde w_0$, and
$w_0'\phi_{e_1}(\tilde h_1) =_{G_{v_0}} \tilde w_0$.
Hence the elements $h_i$ defined by $h_ih'_i=\tilde h_i$  for $1 \leq i \leq k$ satisfy the properties
needed for the conclusion in~(ii).
\end{proof}

We note that in the hypotheses of Lemma~\ref{lem:higginspath}(ii), the word $w$
is an arbitrary element of the inflated Higgins language
$L'=\hat L(\GG,\{(X_{\tau(e)}^\pm)^*\},(X_{v_0}^\pm)^*,v_0,\TT)$; that is,
$w$ is in the inflated Higgins language with respect to the largest possible sets
of coset and vertex group representatives.  
Thus when the $(\LL,L_{v_0})$--cascade process is
applied to a word in this maximal inflated Higgins language, a word is produced
in the inflated Higgins language with respect to more restricted coset and
vertex group representatives in $(\LL,L_{v_0})$.

Given a word $w$ in a Higgins normal form 
$w = \Deflation{w_0s_{e_1}u_1\cdots s_{e_k}u_k}{\TT}$ or
in coset normal form $w = \Deflation{u_0s_{e_1}u_1\cdots s_{e_k}u_k}{\TT}$, 
the {\em $\lam$--path associated with $w$} is the directed path $p=e_1 \cdots e_k$ in
the graph $\lam$.  
An immediate consequence of Lemma~\ref{lem:higginspath} is that any
two words in the inflated Higgins language
$\widehat L(\GG,\LL,L_{v_0},v_0,\TT)$ that represent the
same element of $\pi_1(\GG)$, or any two words in the inflated
Higgins coset language $\widehat L(\GG,\LL,e_0,\TT)$ that represent the
same coset of $G_{e_0}$ in $\pi_1(\GG)$, have the same associated $\lam$--path,
and so the $\lam$--path associated with a (deflated) Higgins normal form is
well-defined.

\begin{proof}[Proof of Theorem~\ref{thm:gog_coset}]
Let $\TT$ be a maximal tree in $\lam$ and let $G := \pi_1(\GG,\TT)$.
Applying Lemma~\ref{lem:sacaxoverlanguagechange}, for each $e \in \dE$ 
we modify the coset language $L_{\tau(e)}^e$ by
removing all representatives of the identity coset $G_e$ other than
the empty word $\emptyword$.
Let 
$$
X := \cup_{v \in V} X_v \cup \{s_e: e \in \dE \setminus \dE_\TT\},
$$
let $\LL:=\{L_{\tau(e)}^e\}$ be the collection of (modified) edge coset languages,
and let
$$
L:=L(\GG,\LL,e_0,\TT)
$$
be the associated Higgins coset language over $X$.
We shall prove that $L$ is the language of a \SACA structure for
the pair $(G,G_{e_0})$ over the generating set $X$.

Now $L$ is a coset language for $(\pi_1(\GG),G_{e_0})$ (as
discussed in Section~\ref{subsec:gogbackground}). Let $L_{\tau(e_0)}$ 
be a regular language of normal forms for $G_{\tau(e_0)}$. 
It is shown in~\cite[Prop.~3.3]{BHS} that  
the Higgins language for $\pi_1(\GG)$, with respect to 
$\LL,L_{\tau(e_0)},\tau(e_0)$, is a regular language; the same proof
shows that the Higgins coset language $L$ is also regular.

It now remains to verify fellow traveller properties for $L$.
Let $K$ be a common fellow traveller constant for the coset automatic 
structures for the pairs $(G_e,G_{\tau(e)})$. 
Applying Lemmas~\ref{lem:multiply_crossover} and~\ref{lem:multiply_stability}
to Hypotheses (ii) and (iii),
we can choose $\lambda \in \N$ to be large enough such that:
\begin{mylist}
\item[(a)] $K \le \lambda$;
\item[(b)] 
for each $e \in \dE$,
the triple $(G_e,G_\be,\phi_e)$ is $\lambda$-stable with respect to
$(Y_e,Y_\be)$; 
\item[(c)] 
for each $e,f \in \dE$ with $\tau(e)=\tau(f)$,
the coset language $L_{\tau(e)}^e$ has $\lambda$-limited crossover with respect to $(Y_e,Y_f)$;
\item[(d)] for each $e \in \dE$, any element $g \in G_e$ with
$|g|_{X_{\tau(e)}} \le K$ satisfies $|g|_{Y_e} \le \lambda$.
\end{mylist}
We define a further constant $N$ to be the maximum value of $|y|_{X_{\tau(e)}}$
for any $e \in \dE$ and any $y \in Y_e$.

Now suppose that $w,w' \in L$ are related by an equation
$wx=_G hw'$, where $h \in G_{e_0}$ and $x \in X \cup \{\emptyword\}$; that is,
either $x$ is in the generating set of $G$, or $x$ represents the identity of $G$.

Let $\widehat X := \cup_{v \in V} X_v \cup \{s_e: e \in \dE\}$,
let $\widehat L := \widehat L(\GG,\LL,e_0,\TT)$ be the inflated Higgins coset language, 
and let $L_{e_0} \subseteq (X_{\tau(e_0)}^\pm)^*$ be a 
set of words representing the elements of the group $G_{e_0}$.
We suppose that 
$\widehat w, \widehat{w'} \in \widehat L$
satisfy $w=\Deflation{\widehat w}{\TT}$ and $w'=\Deflation{\widehat{w'}}{\TT}$,
and write $\widehat{w} =u_0s_{e_1}u_1\cdots s_{e_k}u_k$.
Let $p:=e_1 \cdots e_k$ denote the path in $\lam$ determined by $w$, and
let $z_h \in L_{e_0}$
be any word representing the element $h$ in $G_{\tau(e_0)}$.

{\bf Case (1):}  Suppose that $x=\emptyword$.
Then $\widehat{w} =_G z_h\widehat{w'}$.  
The words $\widehat{w}$ and $z_h\widehat{w'}$ are both in the inflated Higgins
language 
$\widehat L(\GG,\LL,L_{e_0}L_{\tau(e_0)}^{e_0},\tau(e_0),\TT)$, and so
by Lemma~\ref{lem:higginspath}(ii),
 $z_h\widehat{w'}$
is an $(\LL,L_{e_0}L_{\tau(e_0)}^{e_0})$--cascade of 
$\widehat{w}$ also associated with the path $p$ in $\lam$.  
Thus we can write 
$\widehat{w'} =u'_0s_{e_1}u'_1\cdots s_{e_k}u'_k$ with
each $u_i' \in L_{\tau(e_i)}^{e_i}$.  Moreover, 
for $1 \le i \le k$ there are elements $h_i \in G_{e_i}$ 
for which if
$h_i':=\phi_{e_i}(h_i) \in G_{\be_i}$ and $h_0:=h$, then
\begin{eqnarray*}
u_kx =_{G_{\tau(e_k)}} h_ku'_k,&\quad & h'_k=_G s_{e_k}h_ks^{-1}_{e_k},\\
u_{k-1}h'_k=_{G_{\tau(e_{k-1})}} h_{k-1}u'_{k-1},
&&h'_{k-1}=_G s_{e_{k-1}}h_{k-1}s^{-1}_{e_{k-1}},\\
\ldots &&\ldots \hspace{2in} (*)\\
u_1h'_2=_{G_{\tau(e_1)}} h_1u'_1,
&& h'_1=_G s_{e_1}h_1s^{-1}_{e_1},\\
u_0h'_1 =_{G_{\tau(e_0)}}h_0u'_0. 
\end{eqnarray*}

An illustration of the paths 
${}_1w$ and ${}_hw'$ in the Cayley graph $\gam(G,\widehat X)$,
along with the connector $h_i$ and $h'_i$ paths in this
array of equations, is given in Figure~\ref{wdfig1}.  
We note that in this illustration, for each
index $j$ for which $e_j \in \dE_\TT$, the 
edges $s_{e_j}$ along the top and bottom paths 
actually label loops in $\gam(G,\widehat X)$.

\begin{figure}
\setlength{\unitlength}{0.9pt}
\begin{center}
\begin{picture}(380,53)(-5,-8)
\qbezier(0,20)(40,40)(80,20)
\qbezier(80,20)(95,10)(120,12)

\qbezier(0,0)(40,-25)(80,-10)
\qbezier(80,-10)(95,-5)(120,-1)
\put(0,20){\circle*{3}}
\put(0,0){\circle*{3}}
\put(0,20){\vector(0,-1){20}}
\put(3,9){\footnotesize $h_0$}
\put(12,33){$u_0$}
\put(12,-22){$u'_0$}
\put(40,30){\circle*{3}}
\put(40,-14.8){\circle*{3}}
\put(40,30){\vector(0,-1){44.4}}
\put(43,6){\footnotesize $h'_1$}
\put(52,32){$s_{e_1}$}
\put(52,-23){$s_{e_1}$}
\put(80,20){\circle*{3}}
\put(80,-10){\circle*{3}}
\put(80,20){\vector(0,-1){29.5}}
\put(83,4){\footnotesize $h_1$}
\put(120,12){\circle*{3}}
\put(120,-1){\circle*{3}}

\put(135,12){\circle*{2}}
\put(135,-1){\circle*{2}}
\put(145,12){\circle*{2}}
\put(145,-1){\circle*{2}}
\put(155,12){\circle*{2}}
\put(155,-1){\circle*{2}}

\qbezier(170,12)(270,50)(370,20)
\qbezier(170,-1)(270,-30)(370,20)
\put(170,12){\circle*{3}}
\put(170,-1){\circle*{3}}
\put(210,24){\circle*{3}}
\put(210,-9.2){\circle*{3}}
\put(210,24){\vector(0,-1){32.8}}
\put(213,8){\footnotesize $h'_{k-1}$}
\put(217,36){$s_{e_{k-1}}$}
\put(217,-20){$s_{e_{k-1}}$}
\put(250,31.8){\circle*{3}}
\put(258,38){$u_{k-1}$}
\put(250,-11){\circle*{3}}
\put(258,-21){$u'_{k-1}$}
\put(250,31.8){\vector(0,-1){42.2}}
\put(253,10){\footnotesize $h_{k-1}$}
\put(290,32.6){\circle*{3}}
\put(290,-7.5){\circle*{3}}
\put(304,37){$s_{e_k}$}
\put(290,32.6){\vector(0,-1){39.2}}
\put(293,13){\footnotesize $h'_k$}
\put(304,-13.5){$s_{e_k}$}
\put(330,29){\circle*{3}}
\put(330,3.3){\circle*{3}}
\put(347,29){$u_k$}
\put(330,29){\vector(0,-1){25.6}}
\put(335,14){\footnotesize $h_k$}
\put(347,0){$u'_k$}
\put(370,20){\circle*{3}}
\end{picture}
\end{center}
\caption{Fellow-travelling words representing the same coset}\label{wdfig1}
\end{figure}
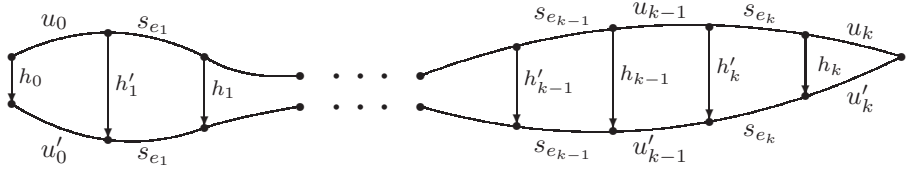

We have $|h_k|_{X_{\tau(e_k)}} \leq K$, from the fellow traveller property
on $L_{\tau(e_k)}^{e_k}$. Our condition (d) ensures that
$|h_k|_{Y_{e_k}} \leq \lambda$.
Then condition (b) ensures that $|h'_k|_{Y_{\be_k}}\leq \lambda$.
Then condition (c) ensures that $|h_{k-1}|_{Y_{e_{k-1}}} \leq \lambda$.
Repeated application of conditions (b) and (c), ensures that, for each $i$,
$|h_i|_{Y_{e_i}} \leq \lambda$ and $|h'_i|_{Y_{\be_i}} \leq \lambda$.
The definition of the constant $N$ shows that 
$|h'_i|_{X_{\tau(e_{i-1})}} \leq \lambda N$ for each $i$.
Application of the fellow traveller properties for the languages
$L_{\tau(e_i)}^{e_i}$ now ensures that ${}_1\widehat{w}$ and
${}_h\widehat{w'}$ asynchronously fellow travel at distance $KN\lambda$
in $\gam(G,\widehat X)$.

The paths ${}_1w$ and ${}_hw'$ in the Cayley graph
$\gam(G,X)$ are obtained from the paths ${}_1\widehat{w}$ and
${}_h\widehat{w'}$ by skipping the $s_{e_j}$ edges in both paths
whenever $e_j \in \dE_\TT$.  
Thus the paths ${}_1w$ and ${}_hw'$ also
asynchronously fellow travel at distance $KN\lambda$ in $\gam(G,X)$.

{\bf Case (2):} Suppose that $x \in X_{\tau(e_k)}$. 
Let $u$ be an $(\LL,L_{e_0}L_{\tau(e_0)}^{e_0})$--cascade of
$wx$.  Then the $\lam$--path associated with $u$ is a prefix
$p'=e_1 \cdots e_\ell$ of $p$.  Since the word $z_h\widehat{w'}$ is in
the inflated Higgins language 
$\widehat L(\GG,\LL,L_{e_0}L_{\tau(e_0)}^{e_0},\tau(e_0),\TT)$
and satisfies $u =_G z_h\widehat{w'}$, it follows from
by Lemma~\ref{lem:higginspath} that the word $z_h\widehat{w'}$ is 
an $(\LL,L_{e_0}L_{\tau(e_0)}^{e_0})$--cascade of
$u$, and the path in $\lam$ associated with $z_h\widehat{w'}$ is
also $p'$.
So $\widehat{w'} =u'_0s_{e_1}u'_1\cdots s_{e_\ell}u'_\ell$ 
with $\ell \le k$ and each $u'_i \in L_{\tau(e_i)}^{e_i}$.
In the case that $\ell<k$, 
let $u'_{\ell+1}= \cdots = u'_k = \emptyword$.

Now a composition of two cascades is again a cascade,
and so $z_h\widehat{w'}$ is an $(\LL,L_{e_0}L_{\tau(e_0)}^{e_0})$--cascade
of the word $wx$. 
Hence there are elements $h_i \in G_{e_i}$ for $1 \le i \le k$
for which if
$h_i':=\phi_{e_i}(h_i) \in G_{\be_i}$ and $h_0:=h$, then
the array in Equation~($*$) holds.
The corresponding paths in the Cayley graph $\gam(G,\widehat X)$
are illustrated in Figure~\ref{wdfig2}.

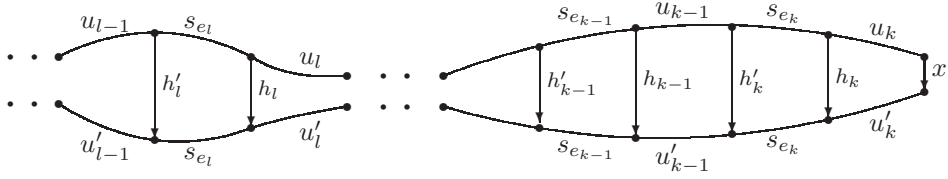
\begin{figure}
\setlength{\unitlength}{0.9pt}
\begin{center}
\begin{picture}(380,60)(-12,-15)
\put(-10,20){\circle*{2}}
\put(-10,0){\circle*{2}}
\put(0,20){\circle*{2}}
\put(0,0){\circle*{2}}
\qbezier(10,20)(50,40)(90,20)
\qbezier(90,20)(105,10)(130,12)

\qbezier(10,0)(50,-25)(90,-10)
\qbezier(90,-10)(105,-5)(130,-1)
\put(10,20){\circle*{3}}
\put(10,0){\circle*{3}}
\put(19,32){$u_{l-1}$}
\put(19,-20){$u'_{l-1}$}
\put(50,30){\circle*{3}}
\put(50,-14.8){\circle*{3}}
\put(50,30){\vector(0,-1){44.4}}
\put(53,6){\footnotesize $h'_l$}
\put(62,32){$s_{e_l}$}
\put(62,-23){$s_{e_l}$}
\put(90,20){\circle*{3}}
\put(90,-10){\circle*{3}}
\put(90,20){\vector(0,-1){29.5}}
\put(93,4){\footnotesize $h_l$}
\put(110,17){$u_l$}
\put(110,-15){$u'_l$}
\put(130,12){\circle*{3}}
\put(130,-1){\circle*{3}}

\put(145,12){\circle*{2}}
\put(145,-1){\circle*{2}}
\put(155,12){\circle*{2}}
\put(155,-1){\circle*{2}}

\qbezier(170,12)(270,50)(370,20)
\qbezier(170,-1)(270,-30)(370,5)
\put(170,12){\circle*{3}}
\put(170,-1){\circle*{3}}
\put(210,24){\circle*{3}}
\put(210,-10){\circle*{3}}
\put(210,25){\vector(0,-1){32.8}}
\put(213,7){\footnotesize $h'_{k-1}$}
\put(217,36){$s_{e_{k-1}}$}
\put(217,-20){$s_{e_{k-1}}$}
\put(250,31.8){\circle*{3}}
\put(250,-14){\circle*{3}}
\put(258,38){$u_{k-1}$}
\put(258,-25){$u'_{k-1}$}
\put(250,31.8){\vector(0,-1){45}}
\put(253,8){\footnotesize $h_{k-1}$}
\put(290,32.6){\circle*{3}}
\put(290,-12.6){\circle*{3}}
\put(304,37){$s_{e_k}$}
\put(290,32.6){\vector(0,-1){44.6}}
\put(293,8){\footnotesize $h'_k$}
\put(304,-18){$s_{e_k}$}
\put(330,29){\circle*{3}}
\put(330,-6.6){\circle*{3}}
\put(347,29){$u_k$}
\put(330,29){\vector(0,-1){35}}
\put(333,9){\footnotesize $h_k$}
\put(347,-11){$u'_k$}
\put(370,20){\circle*{3}}
\put(370,5){\circle*{3}}
\put(370,20){\vector(0,-1){14.5}}
\put(373,12){$x$}
\end{picture}
\end{center}
\caption{Fellow-travelling words, case (2)}\label{wdfig2}
\end{figure}

Then just as in Case~(1) we can bound the lengths of each $h_i,h'_i$ over appropriate
generating sets by $\lambda$, and see that ${}_1w$ and ${}_hw'$
asynchronously $KN\lambda$--fellow travel in $\gam(G,X)$.

{\bf Case (3):} Suppose that $x \in X_v$, for some vertex 
$v$, but $x \not \in X_{\tau(e_k)}$.
In this case  we extend the path $p$ in $\lam$ that corresponds to $w$
to a path $p'$ by appending the unique minimal path $e_{k+1}\cdots e_\ell$
within the tree $\TT$ from $\tau(e_k)$ to $v$; then $\tau(e_\ell)=v$. 
We consider the word 
$\widetilde w := \widehat{w}s_{e_{k+1}}u_{k+1}\cdots s_{e_\ell}u_\ell x$,
with $u_{k+1}=\cdots = u_\ell=\epsilon$.
Since $\widehat w$ does not end with a letter $s_e$ for an edge $e$ in $\TT$,
the word $\widetilde w$ satisfies the hypotheses of the word
$w$ in Lemma~\ref{lem:higginspath}(i).

Let $u$ be 
an $(\LL,L_{e_0}L_{\tau(e_0)}^{e_0})$--cascade of $\widetilde w$; by applying
both parts of Lemma~\ref{lem:higginspath}, we see that the word
$z_h\widehat{w'}$ is an $(\LL,L_{e_0}L_{\tau(e_0)}^{e_0})$--cascade of $u$,
and hence also of $\widetilde w$.
Now the $\lam$--path associated with both $u$ and $z_h\widehat{w'}$ may be 
a prefix $e_1 \cdots e_j$ of $p'$; we can write
$\widehat{w'} =u'_0s_{e_1}u'_1\cdots s_{e_j}u'_j$ 
with $j \le \ell$ and each $u'_i \in L_{\tau(e_i)}^{e_i}$, and
in the case that $j<\ell$, 
let $u'_{j+1}= \cdots = u'_\ell = \emptyword$.
The cascade from $\widetilde w$ to $z_h\widehat{w'}$ now yields
elements $h_i \in G_{e_i}$ for $1 \le i \le \ell$, 
which together with the elements
$h_i':=\phi_{e_i}(h_i) \in G_{\be_i}$ and $h_0:=h$
satisfy the array in Equation~(*).
The corresponding paths in the Cayley graph $\gam(G,\widehat X)$
are illustrated in Figure~\ref{wdfig3}.

\begin{figure}
\setlength{\unitlength}{0.9pt}
\begin{center}
\begin{picture}(380,60)(-12,-15)
\put(-10,20){\circle*{2}}
\put(-10,0){\circle*{2}}
\put(0,20){\circle*{2}}
\put(0,0){\circle*{2}}
\qbezier(10,20)(50,40)(90,20)
\qbezier(90,20)(105,10)(130,12)

\qbezier(10,0)(50,-25)(90,-10)
\qbezier(90,-10)(105,-5)(130,-1)
\put(10,20){\circle*{3}}
\put(10,0){\circle*{3}}
\put(19,32){$u_{k-1}$}
\put(19,-20){$u'_{k-1}$}
\put(50,30){\circle*{3}}
\put(50,-14.8){\circle*{3}}
\put(50,30){\vector(0,-1){44.4}}
\put(53,6){\footnotesize $h'_k$}
\put(62,32){$s_{e_k}$}
\put(62,-23){$s_{e_k}$}
\put(90,20){\circle*{3}}
\put(90,-10){\circle*{3}}
\put(90,20){\vector(0,-1){29.5}}
\put(93,4){\footnotesize $h_k$}
\put(110,17){$u_k$}
\put(110,-15){$u'_k$}
\put(130,12){\circle*{3}}
\put(130,-1){\circle*{3}}

\put(145,12){\circle*{2}}
\put(145,-1){\circle*{2}}
\put(155,12){\circle*{2}}
\put(155,-1){\circle*{2}}

\qbezier(170,12)(270,50)(370,20)
\qbezier(170,-1)(270,-30)(370,5)
\put(170,12){\circle*{3}}
\put(170,-1){\circle*{3}}
\put(210,24){\circle*{3}}
\put(210,-10){\circle*{3}}
\put(210,25){\vector(0,-1){32.8}}
\put(213,7){\footnotesize $h'_{l-1}$}
\put(217,36){$s_{e_{l-1}}$}
\put(217,-20){$s_{e_{l-1}}$}
\put(250,31.8){\circle*{3}}
\put(250,-14){\circle*{3}}
\put(258,38){$u_{l-1}$}
\put(258,-25){$u'_{l-1}$}
\put(250,31.8){\vector(0,-1){45}}
\put(253,8){\footnotesize $h_{l-1}$}
\put(290,32.6){\circle*{3}}
\put(290,-12.6){\circle*{3}}
\put(304,37){$s_{e_l}$}
\put(290,32.6){\vector(0,-1){44.6}}
\put(293,8){\footnotesize $h'_l$}
\put(304,-18){$s_{e_l}$}
\put(330,29){\circle*{3}}
\put(330,-6.6){\circle*{3}}
\put(347,29){$u_l$}
\put(330,29){\vector(0,-1){35}}
\put(333,9){\footnotesize $h_l$}
\put(347,-11){$u'_l$}
\put(370,20){\circle*{3}}
\put(370,5){\circle*{3}}
\put(370,20){\vector(0,-1){14.5}}
\put(373,12){$x$}
\end{picture}
\end{center}
\caption{Fellow-travelling words, case (3)}\label{wdfig3}
\end{figure}
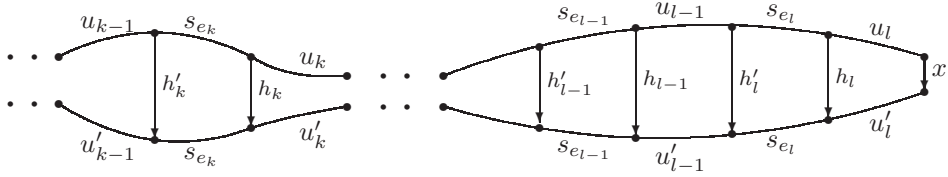

Then, as in Case~(1), we deduce that ${}_1w$ and ${}_hw'$
asynchronously fellow travel in $\gam(G,X)$ at a distance bounded by $KN\lambda$.

{\bf Case (4):} Suppose that $x=s_e$ for some  $e \in \dE \setminus \dE_\TT$.
If $u_k \neq \epsilon$ or $e \neq \be_k$,
then the word $\widehat{w}s_{e_{k+1}} \cdots s_{e_\ell}x$ is in the inflated
Higgins coset language $\widehat L$,
where $e_{k+1} \cdots e_\ell$ is the unique minimal path (possibly empty)
in the tree $\TT$ from $\tau(e_k)$ to the initial vertex of $e$.
Then $wx=\Deflation{\widehat{w}s_{e_{k+1}} \cdots s_{e_\ell}x}{\TT}$
is in the Higgins coset language $L$.  
In this subcase the proof in Case~(1) shows 
that the paths ${}_1wx$ and ${}_hw'$ in $\gam(G,X)$
$KN\lambda$--fellow travel.

So now suppose that $u_k =\epsilon$ and $e=\be_k$.
In that case we can write $\widehat{w} = \widetilde w s_{e_k}$ with
$\widetilde w \in (\widehat X^\pm)^*$, and also write
$\widetilde w = \widetilde{w}''s_{e_j} \cdots s_{e_{k-1}}$, where
$s_{e_j} \cdots s_{e_{k-1}}$ is the maximal suffix of $\widetilde w$
lying in $\{s_e \mid e \in \dE_\TT\}^*$.
That is, $\widetilde w''$ is obtained from $\widehat{w}$ by
removing the letter $s_{e_k}$ at the end, and then removing 
any resulting suffix of generators $s_{e_j}$ associated with edges
lying in tree $\TT$.
Now $\widetilde{w}''$ is in the inflated Higgins coset language 
$\widehat L$, and so the word
$w'' := \Deflation{\widetilde{w}''}{\TT}$ is in the Higgins
coset language $L$.
Moreover, $w'' =_G wx =_G hw'$, and so
Case~(1) applies to show that the paths
${}_1w''$ and ${}_hw'$ in $\gam(G,X)$ 
asynchronously $KN\lambda$--fellow travel.  
Since $w=w''s_{e_k}$, the paths 
${}_1w$ and ${}_hw'$ 
asynchronously $KN\lambda+1$--fellow travel in $\gam(G,X)$.
\end{proof}

\subsection{Finding synchronous subsystems}
\label{subsec:synchsub}

Note that we might expect that an argument analogous to the proof
of Theorem~\ref{thm:gog_coset} would allow us to
derive a synchronous fellow traveller property for $L$ from synchronous
fellow traveller properties for the coset languages $L_{\tau(e)}^e$.
However it is not clear that this is possible, since 
it seems likely that for words 
$\Deflation{u_0s_{e_1}u_1\cdots s_{e_k}u_k}{\TT}$ and 
$\Deflation{u'_0s_{e'_1}u'_1\cdots s_{e'_{k'}}u'_{k'}}{\TT}$ 
 as above, representing the same coset,
the lengths of the corresponding subwords $u_j$ and $u'_j$ could differ.

But, as we prove in Proposition~\ref{prop:synchsub} below,
under certain conditions a strong asynchronous automatic coset system
must contain a synchronous system as a substructure.
We shall use this result to derive Theorem~\ref{thm:gog_coset_synch} and other 
synchronous results relating to Theorem~\ref{thm:gog_coset}.
Our proof of the proposition emulates the proof
of~\cite[Lemma~1]{Elder}, which shows that two geodesic paths
that start at 1 in a Cayley graph and
and  asynchronously $K$-fellow travel must also synchronously $2K$-fellow travel.

\begin{proposition}\label{prop:synchsub}
Suppose that $(G,H)$ has a strong asynchronous automatic coset system
$L^H$ for which $\LHGeo := L^H \cap \Geo^H$ is a coset language (contains at
least one representative of each coset).
Then $\LHGeo$ is a strong synchronous automatic coset system for $(G,H)$. 
\end{proposition}

\begin{proof}
We first prove regularity of $\LHGeo$ by proving regularity of its complement in
$L^H$.  Let $X$ be the generating set for $G$ and 
let $K$ be the asynchronous fellow traveller constant associated with 
$L^H \subseteq (X^\pm)^*$, and let
$N$ be the number of states in the automaton recognising $L^H$.

Suppose that $w \in L^H \setminus \LHGeo$, and 
let $w'$ be the shortest prefix of $w$ that is not of minimal length within
its coset. Then there exists $v \in \LHGeo$ with $|v|<|w'|$ and $v \in Hw'$,
and there exists a word $w'u \in L^H$, with $|u|<N$.
Let $h \in H$ with $w'=_G hv$.
Then, since $v,w'u \in L^H$ with $(Hv)u =_G Hw'u$, the fellow traveller
condition on $L^H$ implies that the paths ${}_1 w'u$ and
${}_hv$ in $\gam(G,X)$
asynchronously $KN$-fellow travel.
Note that this implies in particular that $\elen{h}{X}  \le KN$.

We shall now show that ${}_1 w'$ and ${}_hv$ synchronously fellow travel
with constant $2KN+2$.
Take any vertex $g_1$ of $\gam(G)$ on the path ${}_1 w'$, and let $g_2$
be a vertex of $\gam(G)$ on the path ${}_h v$ that is closest to $g_1$.
Let $u_1$ be the prefix of $w'$ labeling the subpath of ${}_1 w'$
from 1 to $g_1$, and $u_2$ the label of the subpath of
${}_h v$ from $h$ to $g_2$.

Now, both $v$ and the maximal proper prefix of $w'$ 
are shortest representatives of their cosets of $H$, and any prefix
of a word in $\Geo^H$ is also in $\Geo^H$.  Then $u_2 \in \Geo^H$,
and either $u_1=w'$ or $u_1 \in \Geo^H$.  Then we have
$|u_2| \le |u_1| +KN$ and $|u_1| \le |u_2| +KN+2$,
and hence $||u_1|-|u_2|| \le KN+2$. 
So now the vertex $g_3$ of ${}_hv$ that is at distance $|u_1|$ from $h$
is at distance at most $KN+2$ from $g_2$, and hence distance at
most $(KN+2)+KN=2KN+2$ from $g_1$ (see Figure \ref{synchfig}). This verifies
our claim that ${}_1w'$ and ${}_hv$ synchronously $(2KN+2)$-fellow travel.

\begin{figure}
\begin{center}
\begin{picture}(300,180)(-30,-10)
\qbezier(0,72)(44,72)(105,60)
\qbezier(135,54)(185,42)(240,12)
\qbezier(0,60)(110,60)(240,0)
\put(116,57){$w'$}
\put(240,12){\vector(3,-2){0}}
\put(0,0){\circle*{3}}
\put(0,60){\circle*{3}}
\put(240,0){\circle*{3}}
\put(0,0){\line(1,0){240}}
\put(0,-12){\line(1,0){97}}
\put(127,-12){\vector(1,0){113}}
\put(110,-14){$v$}
\put(0,56){\vector(0,-1){52}}
\put(-10,28){$h$}
\put(110,0){\circle*{3}}
\put(110,46){\circle*{3}}
\put(64,0){\circle*{3}}
\put(110,46){\line(0,-1){46}}
\put(110,46){\line(-1,-1){46}}
\put(48,48){$u_1$}
\put(28,-7){$u_2$}
\put(113,36){$g_1$}
\put(64,-7){$g_2$}
\put(113,5){$g_3$}
\put(61,21){\scriptsize $\le KN$}
\put(112,21){\scriptsize $\le 2KN\!\!+\!\!2$}
\put(74,5){\scriptsize $\le KN\!\!+\!\!2$}

\qbezier(0,172)(44,172)(105,160)
\qbezier(135,154)(185,142)(240,112)
\qbezier(0,160)(110,160)(240,100)
\put(116,157){$w'$}
\put(240,112){\vector(3,-2){0}}
\put(0,100){\circle*{3}}
\put(0,160){\circle*{3}}
\put(240,100){\circle*{3}}
\put(0,100){\line(1,0){240}}
\put(0,88){\line(1,0){97}}
\put(127,88){\vector(1,0){113}}
\put(110,86){$v$}
\put(0,156){\vector(0,-1){52}}
\put(-10,128){$h$}
\put(110,100){\circle*{3}}
\put(110,146){\circle*{3}}
\put(156,100){\circle*{3}}
\put(110,146){\line(0,-1){46}}
\put(110,146){\line(1,-1){46}}
\put(48,148){$u_1$}
\put(77,93){$u_2$}
\put(98,105){$g_3$}
\put(98,140){$g_1$}
\put(154,93){$g_2$}
\put(138,122){\scriptsize $\le KN$}
\put(82,122){\scriptsize $\le 2KN$}
\put(118,105){\scriptsize $\le KN$}
\end{picture}
\end{center}
\caption{${}_1 w'$ and ${}_hv$ synchronously fellow travel}\label{synchfig}
\end{figure}
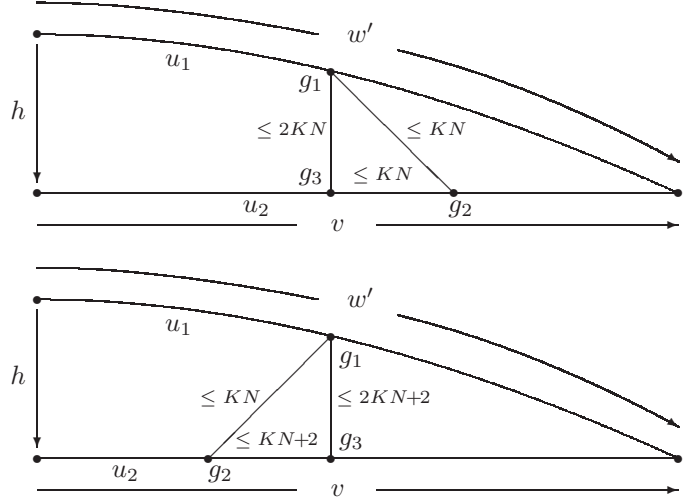

Using the elements of $G$ in the ball of
radius $2KN+2$ centred at 1 (or ``word differences'')
in constructing a finite
set of states, we can construct a finite state automata to recognise
the languages (of padded pairs)
\begin{eqnarray*}
L_h &:=& \{(w,v) \mid w,v \in L^H \text{ and } w \text{ has a prefix } w' 
   \text{ with } w' =_G hv,
   \wlen{w'}>\wlen{v},\\ 
   && \text{ and }
    {}_1 w', {}_hv \text{ synchronously }(2KN+2)\text{--fellow travel}\}
\end{eqnarray*}
for each $h \in H$ with $\elen{h}{X} \le KN$.
(See~\cite[Definition 2.3]{ECHLPT} or~\cite[Section 5.2.1]{HRRbook} for more details on word difference
machines and padding to convert a language of pairs of words to a
language of words over a product alphabet.)
The language $L^H \setminus \LHGeo$ is the union of the
projections onto the first coordinate of the sets $L_h$.
Since regularity is preserved by projection, 
we see that the complement of $\LHGeo$ is indeed regular,
and hence so is $\LHGeo$.

Now suppose that $v,w \in \LHGeo$ and $h \in H$, with $d_{\gam(G)}(w,hv) \leq 1$.
Then, much as above, we see that ${}_1 w$ and ${}_hv$ synchronously fellow
travel with constant $2K$.  For now,
if $g_1$ is any vertex of $\gam(G)$ on the path ${}_1 w$, $g_2$
a vertex of $\gam(G)$ 
that is closest to $g_1$ on the path ${}_h v$, 
$u_1$ the label of the subpath of ${}_1 w'$ from 1 to $g_1$, 
$u_2$ the label of the subpath of ${}_h v$ from $h$ to $g_2$, and
$g_3$ the vertex of ${}_hv$ at distance $|u_1|$ from $h$, then
$|u_2| \le |u_1| +K$ and $|u_1| \le |u_2| +K$,
and hence $||u_1|-|u_2|| \le K$, and $d_{\gam(G)}(g_1,g_3) \leq 2K$. 
\end{proof}

\subsection{A synchronous result for graphs of groups}
\label{subsec:gog_coset_synch}

\begin{theoremA}\label{thm:gog_coset_synch}
Let $\GG = (\lam=(V,\dE), \{ G_v: v \in V \}, \{ G_e: e \in \dE \},
       \{\phi_e \mid e \in \dE\} )$
be a graph of groups over a finite connected directed graph 
$\lam$  with an edge $e_0$.
Let $X_v$ and $Y_e$ be finite generating sets of the groups
$G_v$ and $G_e$, respectively.
Suppose that the
following conditions hold for each $e \in \dE$.
\begin{mylist}
         \item[(i)]  $Y_e \subseteq X_{\tau(e)}$.
	\item[(ii)] The pair $(G_{\tau(e)},G_e)$ is strongly synchronously coset 
	automatic with coset language $L_{\tau(e)}^e$
       satisfying
       $
       L_{\tau(e)}^e \subset \Geo(G_{\tau(e)},X_{\tau(e)}) \cap
       [(X_{\tau(e)}^{\pm})^* \setminus Y_e^\pm(X_{\tau(e)}^{\pm})^*],
       $
       the only representative in $L_{\tau(e)}^e$ of the identity coset is $\emptyword$,
       and each element $g \in G_{\tau(e)}$ is represented by a 
       word $y_gz_g \in \Geo(G_{\tau(e)},X_{\tau(e)})$ 
       with $y_g \in (Y_e^\pm)^*$ and $z_g \in L_{\tau(e)}^e$.
	\item[(iii)] 
	The triple $(G_e,G_\be,\phi_e)$ is 1--stable 
         with respect to $(Y_e,Y_\be)$.
	\item[(iv)] For each $f \in \dE$ with $\tau(e)=\tau(f)$, 
	the coset language  
	$L_{\tau(e)}^e$ has limited crossover with respect to $(Y_e,Y_f)$.
\end{mylist}
Then the pair $(\pi_1(\GG),G_{e_0})$ is strongly synchronously coset automatic. 
\end{theoremA}

\begin{proof}
Let $G:= \pi_1(\GG)$, $H:=G_{e_0}$, and $X := \cup_{v \in V} X_v$, 
and let $\TT$ be any tree in $\lam$.  
Let $\LL := \{L_{\tau(e)}^e \mid e \in \dE\}$.
Then Theorem~\ref{thm:gog_coset} shows that the pair
$(G,H)=(\pi_1(\GG),G_{e_0})$ is \SACA, with respect to the Higgins
coset language
$L^H:=L(\GG,\LL,e_0,\TT)\subseteq (X^\pm)^*$.

By Proposition~\ref{prop:synchsub}, it suffices to show
that $L^H \cap \Geo_{G}^H(X)$ contains at least one representative of each coset.
Note that the empty word is in $L^H \cap \Geo_{G}^H(X)$.

Let $w$ be any nonempty element of $\Geo_G^H(X)$; that is, $w$ is of
minimal length as a representative over $X$ of the right coset $Hw$ in $G$. 
Write $w=x_1 \cdots x_m$ with each $x_i \in X^\pm$.
For each index $i$ such that
$x_i =s_e^{-1} \in \{s_{e'}^{-1} \mid e' \in \dE \setminus \dE_\TT\}$,
replace $x_i$ by $s_\be$, to obtain a word 
$w'=\tilde x_1 \cdots \tilde x_m$ over
$(\cup_{v \in V} X_v^\pm)\cup\{s_{e'} \mid e' \in \dE \setminus \dE_\TT\}$.
For $1 \le i \le m$, if $\tilde x_i \in \cup_{v \in V} X_v^\pm$ then let
$v_i=v_i'$ be the unique vertex in $\lam$ for which $\tilde x_i \in X_{v_i}$ and if 
$\tilde x_i \in \{s_{e'} \mid e' \in \dE \setminus \dE_\TT\}$ then let
$v_i$ and $v_i'$ be the initial and terminal
vertices, respectively, of the edge $e$ for which $\tilde x_i=s_e$.
Let $t_0 \in \{s_e \mid e \in \dE_\TT\}^*$
be the (possibly empty)
word corresponding to the geodesic path in the tree $\TT$ from
$\tau(e_0)$ to the vertex $v_1$.  Similarly for $1 \le i \le m-1$ let
$t_i \in \{s_e \mid e \in \dE_\TT\}^*$ be
the word corresponding to the geodesic in $\TT$ from
$v_{i}'$ to $v_{i+1}$.  
Let $\widehat w:= t_0\tilde x_1t_1 \cdots t_{m-1}\tilde x_m$.
Then 
$w'=\Deflation{\widehat w}{\TT}$.

Repartitioning the subwords of $\widehat w$, we can write 
$\widehat w = u_0s_{e_1}u_1 \cdots s_{e_k}u_k$ with 
$u_i \in (X_{\tau(e_i)}^\pm)^*$ for each $i$,
and $e_1 \cdots e_k$ is a path in $\lam$ starting at $\tau(e_0)$.
Since the original word $w$ is a geodesic over $X$,
and each $u_i$ is a subword of $w$, 
we have $u_i \in \Geo(G_{\tau(e_i)},X_{\tau(e_i)})$ for all $i$.

Next we construct a choice of
$(\LL,(Y_{e_0}^\pm)^*L_{\tau(e_0)}^{e_0})$--cascade
of $\widehat w$.
By hypothesis (ii), 
the element of $G_{\tau(e_k)}$ represented by $u_k$
is also represented by a word of the form
$y_ku_k' \in \Geo(G_{\tau(e_k)},X_{\tau(e_k)})$ with
$y_k \in (Y_{e_k}^\pm)^*$ and $u_k' \in L_{\tau(e_k)}^{e_k}$.
Note that if $u_k$ represents an element of $G_{e_k}$, then $u_k'=\emptyword$.
Then $\wlen{y_k}+\wlen{u_k'}=\wlen{u_k}$.
Now the 1--stability condition says that there is
a word $y_k' \in (Y_{e_{k-1}}^\pm)^*$ with
$y_k' =_{G_{e_k}} \phi_{e_k}(y_k) =_G s_{e_k}y_ks_{e_k}^{-1}$ 
and $\wlen{y_k'}=\wlen{y_k}$.
Next there is word $y_{k-1}u_{k-1}' \in \Geo(G_{\tau(e_{k-1})},X_{\tau(e_{k-1})})$
representing the element $u_{k-1}y_k'$ of $G_{\tau(e_{k-1})}$,
with 
$y_{k-1} \in (Y_{e_{k-1}}^\pm)^*$ and $u_{k-1}' \in L_{\tau(e_{k-1})}^{e_{k-1}}$
(and again $u_{k-1}' =\emptyword$ if $u_{k-1}y_k'$, and hence also $u_{k-1}$,
represents an element of $G_{e_{k-1}}$).
Then $\wlen{y_{k-1}}+\wlen{u_{k-1}'} \le \wlen{u_{k-1}}+\wlen{y_k'}$.
Repeating this process, we obtain the word
$\widehat w'=y_0u_0's_{e_1}u_1' \cdots s_{e_k}u_k'$ satisfying
\begin{eqnarray*}
\wlen{y_0} + \sum_{j=0}^k \wlen{u_j'} 
&\le&  \wlen{u_0}+\wlen{y_1'} +\sum_{j=1}^k \wlen{u_j'} 
\le \cdots 
\le \left(\sum_{j=0}^{i-1} \wlen{u_j}\right) + \wlen{y_i'}
+\left(\sum_{j=i}^k \wlen{u_j'}\right) \\
&\le&  \cdots \le
\sum_{j=0}^{k} \wlen{u_j}. 
\end{eqnarray*} 
Let $\widehat w''=u_0''s_{e_1''}u_1'' \cdots s_{e_\ell''}u_\ell''$
be the word obtained from $\widehat w'$ by removing the $y_0$ prefix,
removing the maximal
suffix in $\{s_e \mid e \in \dE_\TT\}^*$, and (iteratively) removing
any subwords of the form $s_es_\be$ with $e \in \dE_\TT$.

Let $w'':=\Deflation{\widehat w''}{\TT}$.
Then $w''$ represents
the same coset of $H$ in $G$ as the word $w \in \Geo_G^H(X)$.
Since the words $w,w''$ contain the same
number of letters in $\{s_{e} \mid e \in \dE \setminus \dE_\TT\}$ (because
the inflation, cascade, and deflation processes don't alter those letters),
we have
$\sum_{j=0}^{k} \wlen{u_j} \le \sum_{j=0}^\ell \wlen{u_j''} =
\sum_{j=0}^k \wlen{u_j'}$.  Hence these sums are equal,
$y_0 = \emptyword$, and $w'' \in \Geo_G^H(X)$ as well.

The fact that $\emptyword$ is 
the only representative of the identity coset $G_e$ in each $L_{\tau(e)}^e$
guarantees that 
either $\ell=0$ or $e_\ell \notin \dE_\TT$ or $u_\ell''$ does
not represent an element of $G_{e_\ell}$, and
similarly guarantees that, for each subword
$s_{e_i''}u_i''s_{e_{i+1}''}$ of $\widehat w''$, either
$u_i''=\emptyword$ or $u_i''$ does not represent an element of $G_{e_i}$.
By construction, $\widehat w''$ doesn't contain a subword
of the form $s_es_{\be}$ for any $e \in \dE_\TT$.
If the word $\widehat w''$ contains a subword of the form $s_es_\be$
with $e \in \dE \setminus \dE_\TT$, then so does the deflated word $w''$,
contradicting the fact that $w''$ is a geodesic word over $X$.
Then $\widehat w''$ is in the inflated Higgins coset language
$\widehat L(\GG,\LL,e_0,\TT)$, and so its deflation $w''$ is in the
language $L^H$.
  
Therefore $w''$ is an element of $L^H \cap \Geo_{G}^H(X)$ 
representing the same coset as the original word $w$.
Hence $L^H \cap \Geo_{G}^H(X)$ is a coset language for $H$ in $G$,
as required.
\end{proof}

\section{Automaticity for graphs of relatively hyperbolic groups}
\label{sec:relhyp}

In this section, we prove that certain relatively hyperbolic groups have
strong synchronous automatic coset systems that satisfy the crossover
conditions that we need for the application of Theorem~\ref{thm:gog_coset}.
We begin in Section~\ref{subsec:relhypbackground} with some
background and an account of relevant existing results for
relatively hyperbolic groups. 
In Section~\ref{subsec:relhypxover} we discuss crossover and
\SSCA for relatively hyperbolic groups. 
Finally, in Section~\ref{subsec:synchrelhyp} we 
prove (in Corollary~\ref{cor:gogrelhypaut})
automaticity for any fundamental group of a graph of groups in
which the vertex groups are hyperbolic relative to abelian groups, the
edge groups are peripheral subgroups of the vertex groups,
and a further hypothesis holds on paths in the graph. 
Then in Corollary~\ref{cor:relhyp3mfd} we give an application
to 3-manifold groups.

\subsection{Background on relatively hyperbolic groups and biautomaticity}\label{subsec:relhypbackground}

Background and details on relatively hyperbolic groups and biautomatic
structures used in this
paper can be found in~\cite{AC,osin,ECHLPT}.

Let $G=\langle X\rangle$ be a group with finite generating
set $X$.
For any path $p$ in $\gam(G,X)$, let $\iota(p)$ denote the initial vertex,
and let $\tau(p)$ denote the terminal vertex, of $p$.
Given $\lambda \ge 1$ and $c \ge 0$,
the path $p$ is a \emph{$(\lambda,c)$--quasigeodesic}
if for every subpath $r$ of $p$, the inequality
$l(r) \le \lambda d_{\gam(G,X)}(\iota(r),\tau(r))+c$ holds.

The group $G$ is \emph{biautomatic} if there is a 
regular language $L$ for $G$ (over $X$) and a constant
$K \ge 0$ satisfying the property
that whenever $u,v \in L$ and $x,y \in X^{\pm} \cup \{\emptyword\}$
satisfy $ux =_G yv$, the paths ${}_1 y^{-1}ux$ and ${}_1 v$ 
synchronously $K$--fellow travel~\cite[Lemma 2.5.5]{ECHLPT}.

Let $\{H_\omega \mid \omega \in \Omega\}$ be a collection of subgroups of $G$, and
let $\HH=\cup_{\omega \in \Omega} (H_\omega\setminus\{\groupid\})$. 
The graph $\gam(G,X\cup\HH)$ is called the \emph{relative Cayley graph} of $G$.

Given a path $p$ in $\gam(G,X\cup\mathcal{H})$, the path $p$ 
\emph{penetrates} the coset $gH_\omega$ if $p$ contains an edge 
labelled by an element of $H_\omega$ that connects two vertices of
$gH_\omega$. 
An $H_\omega$--\emph{component} of such a path is a non-empty maximal subpath 
of $p$ that is labelled by a word in $H_\omega^*$. 
The path $p$ is said to be \emph{without backtracking} if, whenever 
$p=p'srs'p''$ with two $H_\omega$--components $s,s'$, the initial vertices 
of $s$ and $s'$ lie in different left cosets of $H_\omega$ 
(intuitively, $p$ penetrates every left coset at most once).
The path $p$ is \emph{without vertex backtracking} if each subpath
of $p$ of length at least 2 is labelled by a word that does not represent
an element of an $H_\omega$ subgroup.
In particular, if a path does not vertex backtrack,
then it does not backtrack and all components are edges.

Following~\cite{osin} we say that $G$ is \emph{hyperbolic relative to} 
$\{H_\omega\}$ if the following two conditions hold.

\begin{mylist}
   \item[(i)]\label{def:weakrelhyp} $\gam(G,X\cup\HH)$ is Gromov-hyperbolic.
	\item[(ii)]\label{def:BCP} Given any $\lambda \ge 1$,
	there is a constant $B(\lambda)$ with the following property.
        Let $p$ and $q$ be any two $(\lambda, 0)$--quasigeodesic paths
        without backtracking in $\gam(G,X\cup\HH)$ with $\iota(p) = \iota(q) = 1$
        and $d_{\gam(G,X)}(\tau(p),\tau(q)) \le 1$. Then:
  \begin{mylist}
    \item[(a)] if $s$ is an $H_\omega$--component of $p$ penetrating the 
	coset $gH_\omega$, and $q$ does not penetrate $gH_\omega$, then the 
     distance between the initial and terminal vertices of $s$ in $\gam(G,X)$ 
	is at most $B(\lambda)$;
   \item[(b)] if $s$ is an $H_\omega$--component of $p$ penetrating the
  coset $gH_\omega$ and $s'$ is an $H_\omega$--component of $q$ penetrating 
	the same coset, then in $\gam(G,X)$ the distance between the initial 
	vertices of $s$ and $s'$, and the distance between the
           terminal vertices of $s$ and $s'$, are both at most 
		$B(\lambda)$.
  \end{mylist}
\end{mylist}

Property (i) is frequently called \emph{weak relative hyperbolicity} and
Property (ii) is frequently called \emph{bounded coset penetration}. 
The groups $H_\omega$ are called the \emph{peripheral subgroups} of the 
hyperbolic group $G$.

\begin{remark}\label{rmk:bcpqg}
In a finitely generated
relatively hyperbolic group $G$, bounded coset penetration also holds for 
$(\lambda,c)$--quasigeodesics with $\lambda \ge 1$ and $c \ge 0$, 
with a constant 
$B(\lambda,c)$~\cite[Theorem~3.23]{osin}. 
\end{remark}

Relatively hyperbolic groups with a finite generating set satisfy several
further fundamental properties that we shall use.

\begin{lemma}\label{lem:relhypfg}
Let $G$ be a finitely generated group hyperbolic relative to a collection
$\{H_\omega\}$ of subgroups. Then
\begin{mylist}
\item[(i)]~\cite[Corollary~2.48]{osin} there are only finitely
many groups $H_\omega$; that is, $|\Omega| < \infty$;
\item[(ii)]~\cite[Proposition~2.36]{osin} for all $\omega,\mu \in \Omega$
with $\omega \neq \mu$, the intersection $H_\omega \cap H_\mu$ is finite;
\item[(ii)]~\cite[Proposition~2.29]{osin}  each $H_\omega$ is finitely generated.
\end{mylist}
\end{lemma}


\begin{definition}\label{def:factor}\cite[Construction~4.1]{AC}
	Let 
	$w$ be a word in $(X^\pm)^*$; we define the \emph{factorisation}
        of $w$ to be its decomposition as
		$w=w_0u_1w_1 \cdots u_nw_n$ where 
	\begin{mylist}
		\item[(i)] each $w_k$ is in
     $(X^\pm \setminus (\cup_{\omega \in \Omega} (X^\pm \cap H_\omega)))^*$,
	\item[(ii)]
	each $u_k$ is a nonempty word in $(X^\pm \cap H_{\omega_k})^*$ 
	for some $\omega_k \in \Omega$, 
	\item[(iii)] if $w_k=\emptyword$ and
	$x$ is the first letter of $u_{k+1}$, then $u_kx$
	is not in $(X^\pm \cap H_\omega)^*$ for any $\omega$.
	\end{mylist}

We define the \emph{derived word} $\hat{w}$ of $w$ to be the word
$\hat{w} := w_0h_1w_1 \cdots h_nw_n$ over $X^\pm \cup \HH$, where 
each $h_k$ is the element of $\HH$ represented by $u_k$
(or $h_k = \emptyword$ if $u_k=_G 1$). Similarly, if $p$ is a path in
$\gam(G,X)$ labelled by $w$, then the \emph{derived path} $\hat{p}$ is
the corresponding path in  $\gam(G,X\cup\HH)$ labelled by $\hat{w}$.
\end{definition}

Following the notation in~\cite[Definition~4.5]{AC}, 
given subsets $L_{H_{\omega}} \subseteq (X^\pm \cap H_\omega)^*$ for each 
$\omega \in \Omega$,
let $\Rel(X,\{L_\omega\}^{\pc})$ denote the set of all words $w$ in $(X^\pm)^*$
such that, in the factorisation
$w=w_0u_1w_1 \cdots u_nw_n$ of $w$,
each $u_k$ is a prefix of a word in $\cup_{\omega \in \Omega} L_\omega$.

The following result, which we state here as a lemma, 
is a combination of several results in~\cite{AC}. 
We use it to prove Proposition \ref{prop:relhyp}.

\begin{lemma}\label{lem:AC}
Let $G= \langle X_1 \rangle$ (with $|X_1|<\infty$) be hyperbolic relative to a
family of subgroups $\{H_\omega\}_{\omega \in \Omega}$.
Then there exist constants $\lambda \ge 1$ and $c \ge 0$
and a  finite subset $\HH'$ of 
$\HH = \cup_{\omega \in \Omega} (H_\omega \setminus \{\groupid\})$ such that, 
whenever $X$ is a finite set satisfying
\begin{mylist}
\item[(i)] $X_1 \cup \HH' \subseteq X \subseteq X_1 \cup \HH$, and
\item[(ii)] for all $\omega \in \Omega$, 
the group $H_\omega$ has a geodesic biautomatic structure over $H_\omega \cap X$
with language $L_{H_\omega}$,
\end{mylist}
the following hold.
\begin{mylist}
\item[(a)] For every $\omega,\mu \in \Omega$ with $\omega \neq \mu$,
the intersection $H_\omega \cap H_\mu$ is contained in 
$X$.
\item[(b)] For every $\omega \in \Omega$, the set
$X \cap H_\omega$
generates $H_\omega$.
\item[(c)] For any word $w \in \Geo(G,X)$, the word
$\hat{w} \in (X^\pm \cup \HH)^*$ derived from $w$ labels a
$(\lambda,c)$--quasigeodesic path in $\gam(G,X \cup \HH)$ without
vertex backtracking.
\item[(d)] For every $\omega \in \Omega$, if $w \in \Geo(G,X)$
represents an element of $H_\omega$, then $w \in \Geo(H_\omega,X \cap H_\omega)$. 
\item[(e)] The group $G$ has a biautomatic structure over $X$
with language $\Geo(G,X) \cap \Rel(X,\{L_{H_\omega}\}^{\pc})$.
\end{mylist}
\end{lemma}

\begin{proof}
By Lemma~\ref{lem:relhypfg}, there are 
finitely many peripheral subgroups, and they have pairwise
finite intersections; hence the subset 
$\HH_1 := \cup_{\omega \neq \mu \in \Omega} 
  (H_\omega \cap H_\mu \setminus \{\groupid\})$ of $\HH$ is finite.
It also follows from Lemma~\ref{lem:relhypfg} that each peripheral subgroup
$H_\omega$ has a finite generating set $Y_\omega$, and
so the subset $\HH_2 := \cup_{\omega \in \Omega} Y_\omega$ of $\HH$
is finite.  
Let $\HH_3$ be the finite subset $\HH'$ of~\cite[Lemma~5.3]{AC},
and let $\HH_4$ be the finite subset $\HH'$ of~\cite[Theorem~7.6]{AC}.
Then the finite subset $\HH' := \cup_{i=1}^4 \HH_i$ satisfies (a)--(b),
and (c) and (e) follow  from the two results of Antolin and Ciobanu.
Suppose that $w \in \Geo(G,X)$ represents an element of $H_\omega$,
and let $w=w_0u_1w_1 \cdots u_nw_n$ be the factorisation
of $w$.  Since, by (c), the word $\hat{w}$ derived from
$w$ has no vertex backtracking, the word $\hat{w}$ must have length
at most $1$; hence $w \in (X^\pm \cap H_\omega)^*$, proving (d). 
\end{proof}

\subsection{Crossover properties for relatively hyperbolic groups}\label{subsec:relhypxover}

In this section we use Lemma~\ref{lem:AC} to show that a group that is
hyperbolic relative to geodesically biautomatic subgroups is coset automatic 
relative to each peripheral subgroup with maximal crossover.
We note that a similar but weaker
\SSCA result is shown in~\cite[Theorem~5.4]{BHS}.

\begin{proposition}\label{prop:relhyp}
Let $G=\langle X_1 \rangle$ (with $|X_1| < \infty$) 
be a group that is hyperbolic relative to subgroups
$\{H_\omega \mid \omega \in \Omega\}$. Suppose that, for each $\omega$,
any finite generating set for $H_\omega$ can be extended to one over which
$H_\omega$ has a geodesic biautomatic structure.
Let $\omega_0 \in \Omega$, and let $H:= H_{\omega_0}$.
Then there are constants $\lambda \ge 1$ and $c \ge 0$ 
and a finite generating set $X$ for $G$ satisfying the following.
\begin{mylist}
\item[(1)] The set $X$ satisfies properties (a)--(e) of Lemma~\ref{lem:AC},
and hence the subgroup $H$ is generated by $Y:=X \cap H_{\omega_0}$.
\item[(2)] The pair $(G,H)$ is strongly synchronously coset automatic with respect
to a coset language 
$L^H$
satisfying 
$
       L^H \subset \Geo(G,X) \cap
       [(X^{\pm})^* \setminus Y^\pm(X^{\pm})^*],
       $
       the only representative in $L^H$ of the identity coset is $\emptyword$,
       and each element $g \in G$ is represented by a 
       word $y_gz_g \in \Geo(G,X)$ 
       with $y_g \in (Y^\pm)^*$ and $z_g \in L^H$.
\item[(3)] For all $\omega \in \Omega$ the language
$L^H$ has maximal crossover with respect to $(Y,X \cap H_{\omega})$.
\end{mylist}
\end{proposition}

We note that the condition on finite generating sets of the $H_\omega$ holds
when each subgroup $H_\omega$ is either virtually abelian
(by \cite[Prop 10.1]{AC}) or hyperbolic.

\begin{proof}
Given the finite generating set $X_1$ of $G$,
let $\lambda \ge 1$ and $c \ge 0$ be the constants
and let $\HH'$ be the subset of $\HH$ from (the proof of) Lemma~\ref{lem:AC}.  
Let $X_2 := X_1 \cup \HH'$.
For each $\omega \in \Omega$, the set
$X_2 \cap H_\omega$ generates $H_\omega$ (by Lemma~\ref{lem:AC}(b)),
and so there is another finite generating set
$Y_\omega \supseteq X_2 \cap H_\omega$ over which
$H_w$ has a biautomatic structure, with a language $L_{H_\omega}$.
Let $X:= X_2 \cup (\cup_{\omega \in \Omega} Y_\omega )$.
Then $X_1 \cup \HH' \subseteq X \subset X_1 \cup \HH$.
Moreover, since $H_\omega \cap H_\mu \subseteq \HH' \subseteq X_2$ for all
$\omega \neq \mu$, 
we have $X \cap H_\omega = Y_\omega$ for
all $\omega$.  Now $X$ is a  finite generating set for $G$
satisfying (i)--(ii) of Lemma~\ref{lem:AC}, and so 
properties (a)--(e) of that lemma hold, which proves (1).

Let $H := H_{\omega_0}$ and $Y := X \cap H_{\omega_0}$.  Let 
\[
L := \Geo(G,X) \cap \Rel(X,\{L_{H_\omega}\}^{\pc}),
\]
be the language
of the biautomatic structure for $G$ over $X$ (from Lemma~\ref{lem:AC}(e)).
Finally, let 
\[
L^H := L \cap [(X^{\pm})^* \setminus Y^\pm(X^{\pm})^*];
\]
that is, $L^H$ is the set of
words in the geodesic biautomatic structure for $G$ that do not begin
with a letter in $Y^\pm$. Since $L$ is regular, and the class
of regular languages is closed under intersection,
complementation, and concatenation, the language $L^H$ is also regular.

For any element $g \in G$, there is a word 
$w \in L \subseteq \Geo(G,X)$ representing $g$, and we can write
$w = y_gz_g$ where $y_g$ is the maximal
prefix of $w$ lying in $(Y^\pm)^*$ and $z_g$ does not start
with a letter in $Y^\pm$.  The factorisation of $w$ is 
$y_g$ followed by the factorisation of $z_g$; in particular,
the suffix $z_g$ of $w$
is also a geodesic over $X$ for which the components lie 
in the prefix closures of the geodesic biautomatic structures 
of the component subgroups, and so $z_g \in L^H$.
Moreover, $z_g$ is a representative in $L^H$ of the coset $Hg$.
Thus $L^H$ is a coset language for $(G,H)$.

Let $w \in L^H$ be a representative of the identity coset; that is,
$w \in H$.  Then it follows from Lemma~\ref{lem:AC}(d) that 
$w \in \Geo(H,Y) \subset (Y^\pm)^*$, but no word in $L^H$ begins
with a letter in $Y^\pm$.  Thus $w = \emptyword$ in this case.

Before proving that the language $L_H$ satisfies the requisite
fellow traveller and crossover properties, we prove two lemmas.

\begin{lemma}\label{lem:lem1}
 Let $v \in L^H$ and let $\hat{v}$ be the derived word
defined in Definition~\ref{def:factor}. Then any path in
$\gam(X \cap \HH)$ labelled by $\hat{v}$ is a $(\lambda,c)$--quasigeodesic
that does not vertex backtrack, and no such path of the form
${}_h \hat{v}$ with $h \in H$ penetrates the coset $H$.
\end{lemma}
\begin{proof}
Since $L^H \subseteq \Geo(G,X)$, the first claim follows immediately from
Lemma~\ref{lem:AC}(c).
For any $h \in H$, if the path $p:={}_h \hat{v}$ were to penetrate 
the identity coset $H$,
then we could write $p=rst$, where $s$ is an edge labelled by a
letter in $H$, and the initial and terminal vertices of $s$ lie in $H$.  
However, since (by the definition of $L^H$) the first letter of $v$ cannot
lie in $H$, the path $r$ is nonempty, and so $rs$ is a path of length at
least 2 labelled by a word representing an element of $H$, contradicting the
fact that $p$ has no vertex backtracking.
So $p$ cannot penetrate the coset $H$.
\end{proof}

\begin{lemma}\label{lem:lem2}
Suppose that the word $y \in \Geo(G,X)$ represents an element
of $H = H_{\omega_0}$ that does not lie in $H_\omega$ for any
$\omega \ne \omega_0$, and let $v \in L^H$. Then the path
$p = {}_1\widehat{yv}$
in $\gam(X \cup\HH)$ labelled by the derived word $\widehat{yv}$
is a $(\lambda,c+\lambda+1)$--quasigeodesic that does not backtrack.
\end{lemma}
\begin{proof}
Let $v = v_0s_1v_1 \cdots s_nv_n$ be the factorisation of $v$ and
$\hat{v} = v_0h_1 \cdots h_nv_n$.
We have $y \in \Geo(H,Y)$ by Lemma~\ref{lem:AC}(d) and,
since $y \notin (X^\pm \cap H_\omega)^*$
for any $\omega \neq \omega_0$ and
the first letter $x$ of $v$ does not lie in $Y$,
the word $yx$ is not in any $(X^\pm \cap H_\mu)^*$.
Thus the factorisation of $yv$ is $yv_0s_1v_1 \cdots s_nv_n$,
and $\widehat{yv} = h\hat{v}$, where $y$ represents $h \in H$.
Since, by Lemma~\ref{lem:lem1}, $\hat{v}$ labels a $(\lambda,c)$--quasigeodesic
path, the path $p := {}_1\widehat{yv}$ is a
$(\lambda,c+\lambda+1)$--quasigeodesic. 

Suppose that $p$ backtracks.
Then for some $\omega$ there are two $H_\omega$--components of $p$
whose initial vertices lie in the same coset of $H_\omega$ and,
since by Lemma~\ref{lem:lem1} the subpath ${}_h \hat{v}$ of $p$
is without backtracking, one of those two components must
be the first edge $e:={}_1 h$ of $p$.
By our choice of $y$, the edge $e$ is an $H$--component
of $p$ but not an $H_\omega$--component for any other index
$\omega \neq \omega_0$, and so,
for some index $i$, the edge ${}_{h \cdots v_{i-1}} h_i$ of
$p$ also has initial vertex in the same coset $H$
as the initial vertex $1$ of $e$, and $s_i =_G h_i$
represents an element of $H$.  
But then the nonempty prefix $v_0 s_1 \cdots v_{i-1} s_i$ of the
geodesic $v$ represents an element of $H$ and so, by  Lemma~\ref{lem:AC}(d),
this nonempty prefix lies in $(Y^\pm)^*$, contradicting the fact that
$v \in L^H$.
\end{proof}

Returning to the proof of Proposition~\ref{prop:relhyp},
we next apply these two lemmas to establish the fellow traveller property.
Suppose that $u,v \in L^H$, $x \in X^\pm \cup \{\emptyword\}$,
and $h \in H$ satisfy $ux =_G hv$.  Let $y \in \Geo(G,X)$ be a geodesic
representative of $h$. Then, by Lemma~\ref{lem:AC}, we have $y \in \Geo(H,Y)$.  
If $h$ is in the finite set
$\cup_{\omega \neq \omega_0} (H \cap H_\omega)$ then it follows from the
definition of the generating set $X$ of $G$ that
$y \in X^\pm \cup \{\emptyword\}$, and so $\wlen{y} \le 1$.

Suppose, on the other hand, that $h$ does not lie in $H_\omega$
for any $\omega \neq \omega_0$.  Then, by Lemma~\ref{lem:lem1} applied to $u$,
the path $p:={}_1 \hat{u}$ in $\gam(G,X \cup \HH)$
is a $(\lambda,c)$--quasigeodesic
without backtracking that does not penetrate the coset $H$.
Since increasing the constants preserves the quasigeodesic property,
$p$ is also a $(\lambda,c+\lambda+1)$--quasigeodesic.
By Lemma~\ref{lem:lem2}, the path $q:={}_1 \widehat{yv}$ is 
a $(\lambda,c+\lambda+1)$--quasigeodesic as well.  Since $h \neq \groupid$,
the path $q$ penetrates the coset $H$ in its first edge ${}_1 h$.
Moreover, the paths $p$ and $q$ both start at $1$,
and the group elements at their terminal vertices, represented by
$u$ and $yv$, are connected by a single edge labelled $x$ in $\gam(G,X)$.
Now,
by the Bounded Coset Penetration property
of Remark~\ref{rmk:bcpqg}, the distance between $1$ and $h =_G y$
in $\gam(G,X)$ is at most the constant $B(\lambda,c+\lambda+1)$.

So in either case we have
$\wlen{y} = \elen{h}{X} \le M := \max\{1,B(\lambda,c+\lambda+1)\}$.
Now, by a standard argument, if $K$ is the fellow-traveller constant
of the biautomatic structure $L$ for $G$ over $X$,
then the paths ${}_1 u$ and ${}_h v$ synchronously
$K'$--fellow travel, with $K' := MK+M$.
This completes the proof of (2).

Finally we turn to the crossover property.
Suppose that $u,v \in L^H$, $\omega \in \Omega$,
$g \in H_\omega$, and $h \in H$ satisfy $ug =_G hv$, where $u$
does not represent an element of $H$.  
Let $x$ and $y$ be elements of $\Geo(G,X)$ representing
$g$ and $h$, respectively, and note from Lemma~\ref{lem:AC} that 
$x \in \Geo(H_\omega,X \cap H_\omega)$ and $y \in \Geo(H,Y)$.  
If $h$ is in the finite set
$\cup_{\omega \neq \omega_0} (H \cap H_\omega)$ 
then, as above, we have $\wlen{y} = \elen{h}{X} \le 1$.

Suppose instead that $h$ is not in this finite set.
Then as above, Lemmas~\ref{lem:lem1} and~\ref{lem:lem2}
show that the path
$p:={}_1 \hat{u}$ is a $(\lambda,c)$--quasigeodesic without
vertex backtracking that does not penetrate the coset $H$, and
the path $q:={}_1 \widehat{yv}$ is a $(\lambda,c+\lambda+1)$--quasigeodesic
without backtracking.
Now consider the path $p' := {}_1 \hat{u}g$, where we consider $g$
to be a single letter in the generating set $\HH$.  
The path $p'$ is also a $(\lambda,c+\lambda+1)$--quasigeodesic, since
it consists of the path $p$ together with one more edge $e := {}_{\hat{u}} g$.
Since $p$ does not penetrate $H$, and the word $u$ does not represent
an element of $H$, the initial vertex
of $e$ is not in $H$, and so $p'$ also does not penetrate the coset $H$.

If the path $p'$ does not backtrack then, by the Bounded Coset
Penetration property of Remark~\ref{rmk:bcpqg}, we have
$\wlen{y} = \elen{h}{X} \le B(\lambda,c+\lambda+1)$.

Suppose instead that $p'$ does backtrack; then the final edge $e$
of $p'$ penetrates the same $H_{\omega'}$-coset $uH_{\omega'}$
as one of the edges of $p$, for some index $\omega'$, and since
$g \in H_\omega$, we may take $\omega'=\omega$.
Let $u = u_0s_1u_1 \cdots s_nu_n$ be the factorisation of $u$, and
$\hat{u} = u_0h_1 \cdots h_nu_n$, and suppose that the edge of $p$
labelled by $h_k$ penetrates the coset $uH_\omega$. Then the suffix
$s_ku_k \cdots s_nu_n$ of $u$ represents an element of $H_\omega$ and then
by Lemma~\ref{lem:AC}(d) this suffix is a word in $\Geo(H_\omega,X \cap H_w)$,
and so we must have $k=n$ and $u_n = \emptyword$.

So $s_ng$ represents an element $h' \in H_\omega$,
and the path $p''$ labelled by $u_0h_1 \cdots s_{n-1}u_{n-1}h'$
is a $(\lambda,c+\lambda+1)$--quasigeodesic without
backtracking that does not penetrate $H$, and with the same initial
and terminal vertices as $q$.  Now we can apply
Remark~\ref{rmk:bcpqg} as before to the paths $p''$ and $q$ to conclude that
$\elen{h}{X} \le B(\lambda,c+\lambda+1)$.

Hence $L^H$ has $M$-maximal crossover
with respect to $(Y,X \cap H_\omega)$,
where $M = \max\{1,B(\lambda,c+\lambda+1)\}$.
\end{proof}

\subsection{Synchronous automatic structures for graphs of relatively
hyperbolic groups}\label{subsec:synchrelhyp}

The following is now an immediate corollary of
Proposition~\ref{prop:relhyp} and
Theorems~\ref{thm:gog_coset_synch} and~\ref{thm:coset2auto}.
 
\begin{theorem}
\label{thm:relhyp_synch}
Let $\GG = (\lam=(V,\dE), \{ G_v: v \in V) \}, \{ G_e: e \in \dE \},
  \{ \phi_e: e \in \dE \} )$
be a graph of groups over a finite connected directed graph 
$\lam$.
Suppose that the following conditions hold.
\begin{mylist}
\item[(i)] Each vertex group $G_v$ is finitely generated and hyperbolic 
relative to a
collection of subgroups, and each edge group $G_e$
with $\tau(e)=v$ is one of those peripheral subgroups.
\item[(ii)] Any finite
generating set of any peripheral subgroup $H$ of 
a vertex group $G_v$ can be extended to
one over which the peripheral subgroup has a geodesic biautomatic structure.
\item[(iii)] For each edge $e$, the triple $(G_e,G_\be,\phi_e)$ is 1-stable
  with respect to $(X_{\tau(e)} \cap G_e,X_{\tau(\be)} \cap G_\be)$
  where for each $v \in V$ the set $X_v$ is a finite
  generating set for $G_v$ satisfying
   properties (1)-(3) of Proposition~\ref{prop:relhyp}.  
\end{mylist}
Then for each edge $e_0 \in \dE$ the pair $(\pi_1(\GG),G_{e_0})$ is strongly
synchronously coset automatic.
Moreover, the fundamental group $\pi_1(\GG)$ is automatic. 
\end{theorem}

Once again, we observe that  the condition (ii) holds in particular
when each subgroup $H_\omega$ is either virtually abelian (by \cite[Prop 10.1]{AC})
or hyperbolic.

In general we cannot dispense with the $1$-stability
assumption in condition (iii) of this theorem
even in the case that the vertex groups are hyperbolic
relative to abelian subgroups, as the following example shows. 

\begin{example}
Let $G = \langle a,b,c \mid ab=ba \rangle \cong \Z^2 * \Z$. Then
$G$ is hyperbolic relative to $\{H\}$ with $H := \langle a,b \rangle$.
Let $\GG$ be the graph of groups with a single vertex,
and a single edge $e$ from the vertex group $G$ to itself (so $\pi_1(\GG)$
is  an HNN extension).
We define $\phi_e:H \to H$ by $\phi_e(a)=ab$, $\phi_e(b)=b$.
Then the resulting fundamental group is isomorphic to
$K * \Z$, where $K$ is the Heisenberg group. Since $K$ is not automatic by
\cite[Theorem 8.1.3]{ECHLPT}, the group $K * \Z$ is not automatic by
\cite[Theorem 12.1.8]{ECHLPT}.
\end{example}

However, in Corollary~\ref{cor:gogrelhypaut}
we show that, in the case when  the peripheral subgroups are abelian
and have sufficiently limited interaction,
we can dispense with the $1$-stability
assumption in Theorem~\ref{thm:relhyp_synch}.

\begin{corollary}\label{cor:gogrelhypaut}
Let $\GG$
be a graph of groups associated with a finite connected graph $\lam$
and  finitely generated vertex
groups that are hyperbolic relative to abelian subgroups,
and suppose that each edge group is peripheral 
in its adjacent vertex group.  Suppose further that 
$\lam$ contains no nonempty directed circuit $p$
for which, whenever $e \cdot f$
is a pair of consecutive edges in $p$, the edge groups
corresponding to the terminal vertex of $e$ and the initial
vertex of $f$ are equal.
Then $\pi_1(\GG)$ is an automatic group
with respect to a Higgins language of normal forms.
\end{corollary}

\begin{proof}
Let $\GG = (\lam=(V,\dE), \{ G_v: v \in V) \}, \{ G_e: e \in \dE \},
  \{ \phi_e: e \in \dE \} )$
be this graph of groups.
In~\cite[Theorem~4.3.1]{ECHLPT}, it is shown that every finitely generated
abelian group is shortlex automatic over every generating set; 
moreover, the structure is also biautomatic.

For each $v \in V$, let $X_{v,1}$ be a finite generating set of $G_v$,
let $\{H_{v,\omega} \mid \omega \in \Omega_v\}$ be the collection
of peripheral subgroups for $G_v$, and let 
$\HH_v := \cup_{\omega \in \Omega_v} (H_{v,\omega} \setminus \{\groupid\})$.
Let $\HH'_v$ be the finite subset of $\HH_v$ associated to
$G_v$ and $X_{v,1}$ from Lemma~\ref{lem:AC}, and let 
$X_{v,2} := X_{v,1} \cup \HH'_v$.
Then for each $\omega \in \Omega_v$, the set
$X_{v,2} \cap H_{v,\omega}$ generates the group $H_{v,\omega}$.

Let $\widehat P(\lam)$ be the set of all directed paths
in $\lam$ of the form $p=e_1 \cdots e_k$ such that
$e_{i+1} \neq \be_i$ and
$G_{e_i}=G_{\be_{i+1}}$ for all $1 \le i \le k-1$; that is,
the path $p$ in $\lam$ does not backtrack and
the (peripheral) edge subgroups in the vertex group 
$G_{\tau(e_i)}=G_{\iota(e_{i+1})}$
corresponding to the edges $e_i$ and $\be_{i+1}$ are the same for all $i$.
The hypotheses show that the set $\widehat P(\lam)$
is a finite set.

For any $p=e_1 \cdots e_k \in \widehat P(\lam)$, let
$\bar p := \be_k \cdots \be_1$ be the reverse path,
let $G_p:=G_{e_k}$,
and define $\phi_p: G_{e_k} \rightarrow G_{\bar{e_1}}$
by $\phi_p:=\phi_{e_1} \circ \cdots \circ \phi_{e_k}$.
Note that the hypotheses show that 
$G_{\be_1} \neq G_{e_k}$.
Also define $Y_p := 
\phi_{\bar{p}}(X_{\tau(\bar{p}),2} \cap G_{\bar{p}})$.

For each vertex $v$ of $\lam$, let 
\[
X_{v} := X_{v,2} \cup (\cup_{p \in \widehat P(\lam), \tau(p) =v} Y_p).
\]

Now let $e$ be any edge of $\lam$, and let $v:=\tau(e)$ and $\widetilde v:=\iota(e)$.
The set $X_{v} \cap G_e$ is again a generating
set for the peripheral subgroup $G_e$, over which $G_e$ is geodesic biautomatic.
The proof of Proposition~\ref{prop:relhyp} shows that the 
generating set $X_{v}$ of $G_{v}$ 
satisfies properties (1)-(3) (with respect to the pair $(G_v,G_e)$)
of that Proposition.

The fact that
$H_{v,\omega} \cap H_{v,\mu} \subseteq X_{v,2}$ for all distinct
$\omega,\mu \in \Omega_v$ implies that
\begin{eqnarray*}
X_{v} \cap G_e 
&=& 
\left(X_{v,2} \cup (\cup_{p \in \widehat P(\lam), \tau(p) =v} Y_p)\right) \cap G_e\\
&=&
(X_{v,2} \cap G_e) \cup Y_e \cup  
  (\cup_{p \in \widehat P(\lam) \setminus \{e\}, G_p=G_e} Y_p),
\end{eqnarray*}
and similarly
\begin{eqnarray*}
X_{\widetilde v} \cap G_\be 
&=&
(X_{\widetilde v,2} \cap G_\be) \cup Y_\be \cup 
  (\cup_{q \in \widehat P(\lam) \setminus \{\be\}, G_q=G_e} Y_q).
\end{eqnarray*}
Now $\phi_e(X_{v,2} \cap G_e) =Y_\be$ and 
$\phi_e(Y_e) = \phi_e(\phi_\be(X_{\widetilde v,2} \cap G_\be))
  =X_{\widetilde v,2} \cap G_\be$.
Suppose that $p \in \widehat P(\lam) \setminus \{e\}$ satisfies $G_p=G_e$.
If the last edge of the path $p$ is $e$, then we can write
$p=q \cdot e$ for some path $q \in \widehat P(\lam) \setminus \{\be\}$ 
satisfying
$G_q=G_\be$, and so 
$$
\phi_e(Y_p)=\phi_e(\phi_{\bar p}(X_{\tau(\bar{p}),2} \cap G_{\bar{p}}))
=\phi_e((\phi_{\be} \circ \phi_{\bar q})(X_{\tau(\bar{p}),2} \cap G_{\bar{p}}))
= Y_q.
$$
On the other hand, if the last edge of $p$ is not $e$, then
the path $q := p \cdot \be$ lies in $P(\lam) \setminus \{\be\}$ 
and satisfies $G_q=G_\be$, and
the argument in the previous sentence shows that
$\phi_\be(Y_q)=Y_p$; hence $\phi_e(Y_p) = Y_q$.
Hence $\phi_e$ maps $X_{v} \cap G_e$ to $X_{\widetilde v} \cap G_\be$.
Similarly $\phi_\be$ maps $X_{\widetilde v} \cap G_\be$ to $X_{v} \cap G_e$;
that is, $\phi_e$ is a bijection from the generating set
$X_{v} \cap G_e$ of $G_e$ to the generating set $X_{\widetilde v} \cap G_\be$
of $G_\be$.  Hence the triple $(G_e,G_\be,\phi_e)$ is 1-stable
  with respect to this pair of generating sets.
The result now follows from Theorems~\ref{thm:relhyp_synch} and~\ref{thm:coset2auto}.
\end{proof}

We already noted in Section~\ref{sec:intro} that the automaticity of
$\pi_1(\GG)$ in the above result was previously known, with respect to a different normal form.
in particular, it follows from
Dahmani's Combination Theorem~\cite[Theorem~0.1]{Dahmani} 
that $\pi_1(\GG)$ is hyperbolic relative to a family of abelian groups, 
and then application of~\cite[Corollary~1.8]{AC} shows that $\pi_1(\GG)$ is 
shortlex biautomatic.

We can apply Corollary~\ref{cor:gogrelhypaut} to
the construction of automatic structures for fundamental groups of
3-manifolds.  Although fundamental groups of closed
3-manifolds with JSJ decomposition pieces that do not have Nil or Sol
geometry have been shown 
by Epstein et al.~\cite[Thm.~12.4.7]{ECHLPT} to be automatic, 
the normal forms for the automatic structure are difficult to
determine from the construction in that proof.
The proofs of Theorems~\ref{thm:gog_coset} and~\ref{thm:relhyp_synch}
were partly inspired by the 
proof in~\cite{BHS} that all fundamental
groups of closed 3-manifolds have the related property of
 autostackability, and as in that earlier proof,
our proofs of those theorems use the set of
Higgins normal forms 
described in Section~\ref{subsec:gogbackground}.
We now show that when the pieces of the
JSJ decomposition are hyperbolic,  the fundamental group of
the 3-manifold is also automatic over those normal forms.

\begin{corollaryA}\label{cor:relhyp3mfd}
Let $M$ be an orientable,
connected, compact 3-manifold with incompressible toral boundary
whose prime factors have JSJ decompositions containing only hyperbolic pieces.
Then the group $\pi_1(M)$ is automatic, with
respect to a Higgins language of normal forms.
\end{corollaryA}

\begin{proof}
The manifold $M$ is a connected sum of finitely many prime
manifolds, $M = M_1 \# \cdots \# M_k$, and
the fundamental group $\pi_1(M)$ is the free product
of the groups $\pi_1(M_i)$.  

For each index $i$, the group
$\pi_1(M_i)$ is a fundamental group of a
graph of finitely generated 
groups that are hyperbolic relative to (free) abelian subgroups, 
over a finite connected graph $\lam_i$.
Moreover, this graph of groups satisfies the properties that 
each edge group is a peripheral subgroup in its vertex group,
and for any two edges $e,f$ of $\lam_i$ with 
the same terminal vertex $\tau(e)=\tau(f)$, 
the intersection of the corresponding edge groups is
$G_e \cap G_f = \{\groupid\}$.  Hence conditions~(i) and~(ii) of
Corollary~\ref{cor:gogrelhypaut} are satisfied,
and so $\pi_1(M_i)$ is automatic with respect to a
Higgins language $L_i$ of normal forms.

The free product $\pi_1(M)=\pi_1(M_1) \ast \cdots \ast \pi_1(M_k)$ 
is automatic with respect to 
the standard normal form set $L$ of a free product~\cite[Theorem~12.1.4]{ECHLPT},
constructed using the languages $L_i$ of normal forms for the factor groups above.
Then $\pi_1(M)$ is also the fundamental group of a graph
of groups built from the graphs of groups defining
the groups $\pi_1(M_i)$, by joining the graph $\lam_i$
to the graph $\lam_{i+1}$ by an edge whose associated
edge groups are the trivial group for each $i$, and the language $L$ is
a Higgins language for this graph of groups.
\end{proof}

\begin{remark}
For a nonorientable, connected, compact 3-manifold $M$ with
incompressible toral boundary, whose
JSJ pieces have interiors with hyperbolic geometry,
there is an orientable 2-sheeted cover $M'$ of $M$ 
satisfying the hypotheses of Corollary~\ref{cor:relhyp3mfd},
and $\pi_1(M')$ is an index 2 subgroup of $\pi_1(M)$.
Hence in this case, by~\cite[Theorem~4.1.4]{ECHLPT} and 
Corollary~\ref{cor:relhyp3mfd}, the group $\pi_1(M)$ has 
an automatic structure with a language that is the concatenation 
of the Higgins normal forms for $\pi_1(M')$ with a transversal 
for $\pi_1(M')$ in $\pi_1(M)$.
\end{remark}

\section{Synchronous automaticity when geodesics concatenate up}
\label{sec:concatenate}

In this section we introduce the property for a pair of groups $(G,H)$ that
geodesics in $G$ 
`concatenate up' from the subgroup $H$; such a pair $(G,H)$ is
known in the literature as an {\em admissible pair}.
In Section~\ref{subsec:concatmain} we study crossover
properties for shortlex automatic groups in which geodesics concatenate up
from subgroups, and use this to prove that strong synchronous coset automaticity 
is preserved by the graph of groups construction
when geodesics for all edge groups $G_e$ 
concatenate up to geodesics for their adjacent vertex groups $G_{\tau(e)}$.

Let $G=\langle X \rangle$ be a group and, for some $Y \subseteq X$,
let $H=\langle Y \rangle$ be the subgroup of $G$  generated by $Y$. 

\begin{definition}\label{def:concatenate}
We say that geodesics for $H$ over $Y$ {\em concatenate up} to geodesics for
$G$ over $X$ (or $\Geo(H,Y)$ concatenates up to $\Geo(G,X)$)
provided that whenever $w$ is a geodesic word over $Y$ and
$v_0$ is a word over $X$ that is a minimal length
representative of its coset, the word $wv_0$ is also geodesic.
\end{definition}

Note that this property implies that any element of $G$ has a geodesic 
representative of this form $wv_0$. 

The property of geodesics concatenating up has been used 
by Alonso~\cite{Alonso} and Chiswell~\cite{chiswellHNN}
to study the growth functions of amalgamated free products, HNN extensions,
and fundamental groups of graphs of groups.  
Examples of subgroups in groups with generating sets for which
geodesics concatenate up include any sub-graph product of
a graph product of groups (including a direct factor in a direct product, or
a free factor in a free product)~\cite{chiswellgp},~\cite[Prop.~14.4]{mann}.
Alonso's article~\cite{Alonso} provides many other examples.

In Section~\ref{subsec:coxartin}
we prove that Coxeter groups and 
sufficiently large Artin groups have the property of geodesics
concatenating up with respect to special 
subgroups (over the standard Coxeter and Artin generating sets), 
and  hence amalgamated products, HNN extensions,
and more generally fundamental groups of graphs of these
groups over parabolic subgroups are automatic.

\subsection{Crossover and strong sychnonous coset automaticity for graphs of groups when geodesics concatenate up}
\label{subsec:concatmain}

In order to obtain, in Theorem~\ref{thm:concat_synch}, 
\SSCA for fundamental groups of graphs
of groups in the case that geodesics concatenate up,
we begin by describing a situation that yields a
geodesic \SSCA and a $1$-limited crossover condition for a subgroup in a group.

\begin{proposition}\label{prop:concatenate}
Let $G=\langle X\mid R \rangle$,
let $H=\langle Y \rangle$ for some $Y \subset X$, and 
suppose that geodesics for $H$ over $Y$ concatenate up to geodesics for $G$ over $X$.
Suppose that $G$ is shortlex automatic with respect to some ordering of
$X^\pm$ in which all elements of $Y^\pm$ precede all
elements of $X^\pm \setminus Y^\pm$, and let the
languages $\SLex$ and $\SLex^H$ be defined with respect to that ordering. 
Then the coset language $\SLex^H$ has $1$-limited crossover with respect to
$(Y,Z)$ for any $Z \subseteq X$,
and defines a strong synchronous (shortlex) automatic coset system for
$(G,H)$.
\end{proposition}

\begin{proof}
Let $u \in \SLex^H$ and $x \in X^{\pm}$. The conclusions will follow once we
have proved that either
$ux =_G v$ with $v \in \SLex^H$ or $ux=_G yv$ with $v \in \SLex^H$ and 
$y \in Y^\pm$.
In the first case, since $u$ and $v$ are both in $\SLex^H$, we must have
$||u|-|v|| \leq 1$.  In the second case, we shall show that $|u|=|v|$.
These restrictions on $u$ and $v$ will be used later in the proof of
Proposition \ref{prop:synchegs}.

Note that $v$ or $yv$ will then be proved to be the unique representative
of $ux$ in $\SLex$, since we put all the letters of $Y^\pm$ first in the
ordering, provided that in the case where there is more than one choice for $yv$
we choose that one with $y$ earliest in the ordering of $Y^\pm$.
So the synchronous fellow travelling of the path ${}_1u$ with the path ${}_1v$
or ${}_yv$ then follows from the (synchronous) shortlex automaticity of $G$.

{\bf Case 1.}
Suppose first that the word $ux$ is not geodesic, and let $v_1 \in \SLex$
represent the element $ux$ of $G$.
Then $|v_1|$ is equal to either $|u| - 1$ or $|u|$.
We claim that $v_1 \in \SLex^H$, which will complete the proof in this case. 

If not, then let $v_0$ be the representative of $Hv_1$ in $\SLex^H$,
so $v_0 <_\SL v_1$, with $v_1 =_G hv_0$ for some $h \in H$.
Since geodesics for $H$ concatenate up to $G$, whenever $w$ is a geodesic
representative for $h$, the word $wv_0$ must be geodesic. It follows that
$|v_0|<|v_1|$.  But now $u = _G v_1x^{-1} =_G hv_0x^{-1}$ with
$|v_0x^{-1}| < |v_1| + 1 \leq |u|+1$, and so $v_0x^{-1} =_G h^{-1}u$.
Now (since geodesics concatenate up) $w^{-1}u$ must be geodesic, 
but we have $|v_0x^{-1}| < |u|+1 \leq |w^{-1}u|$, and so we have a contradiction.

{\bf Case 2.} Suppose now that $ux$ is geodesic.  Let $ux =_G v_1$ with $v_1
\in \SLex$ representing the element $ux$ of $G$. So $|v_1| = |ux|$.
If $v_1 \in \SLex^H$ then we are done. 

If not, then again let $v_0$ be the representative of $v_1$ in $\SLex^H$, so
$v_0 <_\SL v_1$ with $v_1 =_G hv_0$ for some
$h \in H$. Let $w$ be a geodesic representative of $h$.
Then, again, since geodesics concatenate up, $wv_0$ must be geodesic, of the same length as $v_1$, and so 
$|v_0| < |v_1|$ and $|w|=|v_1|-|v_0|$.

If $|v_0| = |v_1| - 1$, then $|w|=1$, so $w \in Y^\pm$, 
and so $ux =_G wv_0$ with $|v_0| = |v_1| - 1 = |u|$, which proves the result.

Otherwise, $|v_0| \leq |v_1| - 2$. 
Then $|w|\geq 2$, and
$ux =_G v_1 =_G wv_0$.
Then $v_0x^{-1}=_G w^{-1}u$, and, since geodesics concatenate up,
$w^{-1}u$ must be geodesic.
But $|w^{-1}u|>|v_0x^{-1}|$, so we have a contradiction.
This contradiction completes the proof of the proposition.
\end{proof}

\begin{theoremA}
\label{thm:concat_synch}
Let $\GG = (\lam=(V,\dE), \{ G_v: v \in V) \}, \{ G_e: e \in \dE \},
 \{ \phi_e: e \in \dE \} )$
be a graph of groups over a finite connected graph $\lam$.
Let $X_v$ and $Y_e$ be finite generating sets of the groups
$G_v$ and $G_e$, respectively.
Suppose that the
following conditions hold for each $e \in \dE$.
\begin{mylist}
\item[(i)]  $Y_e \subseteq X_{\tau(e)}$.
\item[(ii)] $\Geo(G_e,Y_e)$ concatenates up to 
$\Geo(G_{\tau(e)},X_{\tau(e)})$.
\item[(iii)] The triple $(G_e,G_\be,\phi_e)$ is 1-stable 
  with respect to $(Y_e,Y_\be)$.
\item[(iv)]  $G_{\tau(e)}$ is shortlex automatic
with respect to an ordering of $X_{\tau(e)}$ in which all letters of $Y_e^\pm$
precede all letters of $X_{\tau(e)}^\pm \setminus Y_e^\pm$.
\end{mylist}
Let $\LL$ be the set of coset languages $\SLex_{G_{\tau(e)}}^{G_e}$,
for $e \in \dE$, and let $\TT$ be any maximal tree in $\lam$.
Then, for each $e_0 \in \dE$, the pair $(\pi_1(\GG),G_{e_0})$ is strongly synchronously coset automatic, 
with the Higgins coset language $L:= L(\GG,\LL,e_0,\TT)$.
Furthermore
$L \subseteq \Geo^{G_{e_0}}$,
and the group $\pi_1(\GG)$ is automatic.
\end{theoremA}

\begin{proof}
Define $v_0:=\tau(e_0)$, $H:=G_{e_0}$ and 
$X := \cup_{v \in V} X_v \cup \{ s_e: e \in \dE \setminus \dE_\TT\}$.
We apply Proposition~\ref{prop:concatenate}
in order to verify for each $e \in \dE$ that the pairs $(G_{\tau(e)},G_e)$ 
satisfy conditions (i) and (iii) of Theorem~\ref{thm:gog_coset}.
Since we are already assuming hypothesis (ii) of that theorem, we can apply
it to conclude that
the pair $(G,H)$ is \SACA, with the language as described. 

Our next step is to prove that
$L \subseteq \Geo^H$.
So suppose that $w \in L$, and let $u$ be a representative of $Hw$ of minimal
$X$-length. We create a word $\hat{u}$ 
from $u$ by inserting into $u$
symbols $s_e$ for $e \in \dE_\TT$, so that
$\hat{u} = u_0s_{e_1}u_1\cdots s_{e_k}u_k$, where $e_1\cdots e_k$ is a path within
$\lam$ that starts at $v_0$, and
where $u_0 \in (X_{\tau(e_0)}^\pm)^*$ and
$u_i \in (X_{\tau(e_i)}^\pm)^*$ for each $i$. 
Note that $\Deflation{\hat{u}}{\TT}=u$.
Next we construct an
$(\LL,\SLex_{G_{e_0}}\SLex_{G_{\tau(e_0)}}^{G_{e_0}})$--cascade 
of $u$, as follows.
We define $u'_k$ to be the shortlex minimal representative word
over $X_{\tau(e_k)}$ in the coset $G_{e_k}u_k$, and suppose that
$u_k=_{G_{\tau(e_k)}} h_ku'_k$. 
Our hypothesis (ii) ensures that $|u_k| = |u'_k| + |h_k|_{Y_{e_k}}$.
Now we define $h'_k \in G_{\be_k}$ to be the element 
$\phi_{e_k}(h_k) =_G s_{e_k}h_ks^{-1}_{e_k}$.
The 1-stability condition implies that 
$|h'_k|_{Y_{\be_k}} \leq |h_k|_{Y_{e_k}}$. 
We repeat this procedure, but using $u_{k-1}h'_k$ rather than $u_k$,
and so define elements $h_k,\ldots,h_1,h_0$, words
$u'_k,\ldots,u'_0 \in \SLex_{G_{\tau(e_j)}}^{G_{e_j}}$, and elements $h'_k,\ldots,h'_1$. 
Let $w_0$ be a geodesic word over $Y_{e_0}$ that represents $h_0$.
The deflation $u'$ of the word $w_0u_0's_{e_1}u'_1s_{e_2}u'_2\cdots s_{e_k}u'_k$,
which represents the same element as $u$, is no longer than $u$. Hence $u'$ must
be geodesic, and since the deletion of $w_0$ results in a word in the same
coset, $w_0$ must be the empty word (and so $h_0=_G 1$).
Now $u' \in L$. Since $L$ has uniqueness, we have $u'=w$.

Now synchronicity of $L$ follows by Proposition~\ref{prop:synchsub}.
Then application of Theorem~\ref{thm:coset2auto} proves that $\pi_1(G)$
is automatic.
\end{proof}

\subsection{Application to graphs of Coxeter and sufficiently large type Artin groups}
\label{subsec:coxartin}

We assume that the reader is familiar with 
the definitions of Coxeter groups and Artin groups  
(also known as Artin--Tits groups, of which Coxeter groups are natural quotients) 
and with the presentations of these
groups over their standard generating sets; 
for Coxeter groups,~\cite{Bourbaki} is a standard reference.

The following lemma is noted in~\cite[Example 1]{Alonso}, 
and an outline of the proof is given 
in~\cite[Exercise Ch.IV~{\S}1(26)]{Bourbaki}.
It is also proved in \cite[Proposition 7.11]{Antolin}.
 
\begin{lemma} \label{lem:cox}
Let $G=\langle X \rangle$ be a Coxeter group, 
defined over its standard generating set $X$, and let $H=\langle Y \rangle$
be a subgroup of $G$, for some
$Y \subseteq X$. Then geodesics for $H$ concatenate up to $G$.
\end{lemma}

\begin{proof}
The proof of the lemma uses the {\em Exchange Lemma}
\cite[Chapter IV.1.4, Lemma 3]{Bourbaki} for
Coxeter groups, which says that, in any non-geodesic word  over $X$,
we can get a shorter representative of the same group element by removing
two of the letters in the word.

We let $w$ be a geodesic word over $Y$, and 
$v$ a geodesic word over $X$ such that $wv$ is non-geodesic, and
prove that in that case $v$ cannot be of minimal length within its coset.

Let $w'v$ be a minimal non-geodesic word with $w'$ a suffix of $w$.
Since $v$ is geodesic, $w'$ is nonempty. Let $w' = tw''$ with $t \in Y^\pm$.
Then, by the Exchange Lemma, we can  remove two of the letters of the
non-geodesic word $tw''v$ to get a shorter representative of the same group
element. Since $w''v$ and $tw''$ are both geodesic, one of these
removed letters must be $t$ and the other must lie in $v$. So the result of
removing this letter from $v$ is a shorter representative of the coset $Hv$.
\end{proof}

Given a Coxeter graph $\artgraph$ (that is, a finite simple graph whose edges
are labelled by parameters $m_{ij}$ each from the set 
$(\N \setminus \{0,1\}) \cup \{ \infty \}$),
we denote by $A_\artgraph$ the associated Artin group.
Suppose that $G$ is an Artin group with standard generating set $X$,
and that the integers $m_{ij}$ are the parameters of the standard presentation
(which label the edges of $\artgraph$).
The group $G$, as well as the Coxeter graph $\artgraph$,
is said to be of {\em large type} if for all $i\neq j$, $m_{ij}\geq 3$,
and (following \cite{HR2}) of {\em sufficiently large type} if for any triple
$i,j,k$ either none of $m_{ij},m_{ik},m_{jk}$ are equal to $2$,
or all three of $m_{ij},m_{ik},m_{jk}$ are equal to $2$,
or at least one $m_{ij},m_{ik},m_{jk}$ is infinite.

The following lemma is proved in~\cite{chiswellgp} 
(see also~\cite[Prop.~14.4]{mann})
for the special case of right-angled Artin groups.
In order to prove the result for Artin groups of sufficiently large type,
we need to use knowledge of geodesics in these groups, and of a process
that reduces any word to geodesic form, which is described in
\cite{HR1, HR2}. In particular, some familiarity with the concept
of {\em critical sequences of moves} applied to words over the generators
is required in order to understand the following proof.

\begin{lemma} \label{lem:artin}
Let $G=\langle X \rangle$ be an  Artin group of sufficiently large type, 
defined over its standard generating set $X$, and let $H=\langle Y \rangle$
be a subgroup of $G$, for some
$Y \subseteq X$. Then geodesics for $H$ concatenate up to $G$.
\end{lemma}
\begin{proof}
We prove the contrapositive, as follows.
As in the proof of Lemma~\ref{lem:cox}, we
let $w$ be a geodesic word over $Y$, and $v$ a geodesic word over $X$
such that $wv$ is non-geodesic, and 
prove that in that case $v$ cannot be of minimal length within its coset.
We prove this by showing that $v$ must be equal in $G$ to a geodesic
word that starts with a letter of $Y^\pm$.

Let $w'v$ be a minimal non-geodesic word with $w'$ a suffix of $w$.
So, if $u$ is a geodesic word with $u =_G w'v$, then $|w'v| - |u| \le 2$
but, since all defining relators of Artin groups have even length,
we must have $|u| =  |w'v| - 2$.  Since $v$ is geodesic, $w'$ is nonempty.
Let $w' = tw''$ with $t \in Y^\pm$.
Then $w''v =_G t^{-1}u$ with both words geodesic.

If $w''$ is empty then $v =_G t^{-1} u$, which proves the lemma.
Otherwise, since $w$ is geodesic, $w''$ does not start with $t^{-1}$,  but it
starts with a letter in $Y^\pm$. By \cite[Prop 3.2 (1)]{HR2} (applied with
`left' in place of `right'), a single leftward critical sequence 
(a sequence of overlapping replacements of 2-generator subwords by words
of the same length on the same 2 generators) can be
applied to $w''v$ to transform it to a word starting with $t^{-1}$.
The moves in the sequence cannot all take place within $w''$ because that
would contradict $w$ being geodesic.
If some of the moves in the critical sequence take place within $v$, then we
can just change $v$ to the result of these moves. So we can assume that
the first move $\mu_1 \to \nu_1$ in the sequence overlaps both $w''$ and $v$.

We claim that the two generators involved in this move must both be in $Y$.
So suppose that one of them, $s$ say, is not.
The first letter of $\mu_1$ lies in $w''$ and hence in $Y^\pm$, and
so $\nu_1$ begins with $s$ or $s^{-1}$.
If there is a second move $\mu_2 \to \nu_2$ in the sequence, then $\mu_2$
begins with a letter in $w''$ and hence in $Y^\pm$ and ends with $s^{\pm 1}$,
so $\nu_2$ must also begin with $s$ or $s^{-1}$.
Then we see by induction that, for all moves $\mu_i \to \nu_i$ in the
sequence, $\nu_i$ begins with $s$ or $s^{-1}$ and hence, after applying the
complete sequence, the resulting word begins with $s$ or $s^{-1}$. But we know
already that it begins with $t^{-1}$ with $t \in Y^\pm$, so we have a
contradiction, which proves the claim.

Since one of the two generators involved in the first move $\mu_1 \to \nu_1$
is the first letter of $v$ or the inverse of that generator, it follows that
the first letter of $v$ is in $Y^{\pm}$, and the lemma is proved.
\end{proof}

\begin{corollary} \label{cor:artincox}
A fundamental group of a graph of groups in which each vertex group is either 
a Coxeter group or a sufficiently large type Artin group, and in which each
edge group is a special subgroup in its adjacent vertex group, is automatic.
\end{corollary}

We note that a fundamental group of such a graph of groups built 
only out of Coxeter groups, or only out of Artin groups, must itself be such 
a group. And conversely, by \cite{MS}, any Coxeter group that arises as the 
fundamental group of a graph of groups must arise in a similar way.

The following gives a number of examples. 
\begin{corollaryA} 
\label{cor:newartin}
Let $\artgraph$ be a Coxeter graph of sufficiently large type.
Given arbitrary subgraphs $\Lambda_1,\Lambda_2,\ldots,\Lambda_k$ 
of $\artgraph$, suppose that the
Coxeter graph $\artgraph'$ is formed by adjoining new vertices 
$v_1,v_2,\ldots,v_k$ to $\artgraph$ together with 
the following edges from each $v_i$:
\begin{mylist}
\item[] to each vertex of $\Lambda_i$, with the label $2$,
\item[] to each vertex of $\artgraph \setminus \Lambda_i$, with the label $\infty$,
\item[] to each vertex
$v_j$ with $j \neq i$, with the label $\infty$.
\end{mylist}
Then the Artin group $A_{\artgraph'}$ is automatic. 
\end{corollaryA}

\begin{proof} 
The Artin group $A_{\artgraph'}$ is a multiple HNN-extension of 
$G_v=A_\artgraph$ over the subgroups $G_{e_i}=A_{\Lambda_i}$, 
where $\phi_{e_i}\colon A_{\Lambda_i}\to A_{\Lambda_i}$ is the identity map. 
Thus, this graph of groups satisfies condition (iii) of 
Theorem~\ref{thm:concat_synch}. 
Further, $X_v=V(\artgraph)^{\pm}$ and 
$Y_{e_i}=V(\Lambda_i)^{\pm}$ and so condition (i) is satisfied. 
By Lemma~\ref{lem:artin}, geodesics from $A_{\Lambda_i}$ concatenate up 
to $A_\artgraph$, and so condition (ii) is satisfied.  Condition (iv) 
follows 
from~\cite{HR2}. Thus, $(A_{\artgraph'},A_{\Lambda_1})$ is \SSCA.
Moreover, $A_{\Lambda_i}$ is also shortlex automatic, by~\cite{HR2}. 
Thus, by Theorem~\ref{thm:coset2auto}, $A_{\artgraph'}$ is 
automatic.
\end{proof}

\begin{example}
A 4-generator example is provided by extending the Artin group of
type $\widetilde{A_2}$ by one generator $y_1$, defined to commute with two 
of the existing generators. This is the Artin group defined by the Coxeter 
diagram shown in Figure~\ref{Artinfig}.
\end{example}
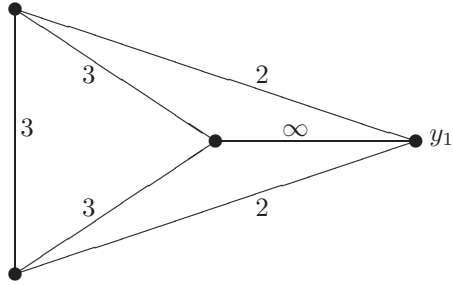
\begin{figure}
\begin{center}
\begin{picture}(100,100)
\put(0,0){\circle*{5}}
\put(75,50){\circle*{5}}
\put(150,50){\circle*{5}} \put(155,50){$y_1$}
\put(0,100){\circle*{5}}
\put(0,0){\line(0,1){100}}\put(2,52){$3$}
\put(0,0){\line(3,2){75}} \put(25,22){$3$}
\put(0,0){\line(3,1){150}} \put(90,22){$2$}
\put(0,100){\line(3,-2){75}} \put(25,72){$3$}
\put(0,100){\line(3,-1){150}} \put(90,72){$2$}
\put(75,50){\line(1,0){75}} \put(100,52){$\infty$}
\end{picture}
\end{center}
\caption{An Artin group not previously known to be automatic}\label{Artinfig}
\end{figure}

As far as the authors know, automaticity of the family
of Artin groups covered by Corollary~\ref{cor:newartin} was 
previously unknown, since this family includes groups
that are not of sufficiently large type. 
On the other hand, it was clear that (as fundamental groups of graphs of groups)
they had solvable word problem, 
though quadratic Dehn function was (probably) unknown. 

Further, it can be shown that these Artin groups satisfy Dehornoy's property H, 
introduced in \cite{Dehornoy1}, which implies that their word problem is 
solvable via padded multifraction reduction \cite[Proposition 1.14]{DHR}.

\section{Further strong synchronous coset automatic structures} \label{sec:misc} 


For our next example of a family of groups and subgroups with
limited crossover, we consider the case in which the group is abelian. 
It is proved in \cite[Theorem 4.3.1]{ECHLPT} that finitely
generated abelian groups are shortlex automatic over all finite
generating sets. 
The following proposition expands the result to coset systems relative 
to any subgroup.

\begin{proposition}\label{prop:ab}
Let $G=\langle X \rangle$ be a finitely generated abelian group, and let
$H = \langle Y \rangle$ be a subgroup. Then $(G,H)$ is strongly synchronously
coset automatic with $1$-limited crossover with respect to $Y$. Furthermore,
for any ordering of $X^{\pm}$, we can choose the coset automatic
structure to consist of the shortlex least representatives of the cosets of $H$.
\end{proposition}
\begin{proof}
We suppose that $X^{\pm} = \{ x_1,\ldots,x_n\}$, with
$x_1<x_2<\ldots<x_n$.
For each $x \in X \cup X^{-1}$, write $\bar{x}$ for the coset $Hx$.
Then $\bar{X}:=\{\bar{x}: x \in X\}$ is a generating set for $\bar{G}:= G/H$,
and each word $w$ over $X$ has an image $\bar{w}$ over $\bar{X}$.
Let $\rho: X \rightarrow \bar{X}$ be the map $x \mapsto \bar{x}$. Note that
$\rho$ might not be injective, but we may choose an injective
map  $\rho': \bar{X} \to X$ such that $\rho\rho'(\bar{x}) = \bar{x}$ for all
$\bar{x} \in \bar{X}$,
and then extend $\rho'$ to an injective monoid homomorphism
from words over $\bar{X}^{\pm}$ to words over $X^{\pm}$.

By \cite[Theorem 4.3.1]{ECHLPT}, $\bar{G}$ is shortlex automatic with respect to
the generating set $\bar{X}$; we define
$L_{\bar{G}}$ to be the shortlex language for $\bar{G}$.
Now we choose $L$ to be the set $\rho'(L_{\bar{G}})$.
Then, as the image of a regular set under a monoid homomorphism, $L$ is regular.
The words in $L_{\bar{G}}$ all have the form
$\bar{x}_1^{r_1}\cdots \bar{x}_n^{r_n}$, with $r_i \geq 0$,  and those in $L$
the form $x_1^{r_1}\cdots x_n^{r_n}$.

Now suppose that $v=x_1^{a_1}\cdots x_n^{a_n}$, $w=x_1^{b_1}\cdots x_n^{b_n}$
are elements of $L$ and $h \in H$ with $d_{\gam(G)}(v,hw) \leq 1$.
Then $d_{\gam(\bar{G})}(\bar{v},\bar{w}) \leq 1$,
and it follows from the proof of
\cite[Theorem 4.3.1]{ECHLPT} that $\Sigma_i | b_i-a_i | \leq B$,
for some constant $B$.
Hence, since $h= _G vxw^{-1}$, with $x \in X \cup X^{-1} \cup \{ \emptyword \}$,
the length of $h$ is bounded by $B+1$.  We can also see that,
where $v(j), w(j)$ denotes the prefixes of $v,w$ of length $j$,
each of the elements $v(j)^{-1}w(j)$ is represented by a product
$x_1^{r_1}\cdots x_n^{r_n}$ with $\Sigma |r_i|\leq 2B$.
It follows that the differences
$v(j)^{-1}hw(j)=_G hv(j)^{-1}w(j)$  are bounded in length.

Since $G$ is abelian, it is straightforward to show that $L$ has limited crossover
with respect to the pair $(Y,Y)$.
\end{proof}

\begin{proposition}\label{prop:vab}
Let $G$ be a finitely generated virtually abelian group, and let $H$ be a
subgroup.  Then $(G,H)$ is strongly synchronously coset automatic.
\end{proposition}
\begin{proof}
We shall construct first a coset language $L_1^H$ for $G$.
We have already considered the case when $G$ itself is abelian in
Proposition~\ref{prop:ab}. It will be convenient here, however, in
this special case to define a second language $L_2^H$,
which has additional properties that we shall need in the proof of
Proposition~\ref{prop:amalg_relhyp_ab}.

Let $F \unlhd G$ with $F$ free abelian and $|G:F|$ finite.
By Lemma~\ref{lem:basic_properties}(iii) it is sufficient to prove our
result for the subgroup $FH$ of $G$. But it will not be convenient to make
that assumption in the aforementioned application to
Proposition~\ref{prop:amalg_relhyp_ab} so, in the case when $G$ is nonabelian,
we shall assume from now on that $G=FH$, but not when $G$ is abelian.

Let $J := H \cap F$. Then, since either $G=FH$ or $G$ is abelian, we have
$J \unlhd G$.  We can find a subgroup $E \ge J$ of $F$ such that $E/J$ is
torsion-free, and $E$ is characteristic of finite index in $F$,
and hence normal of finite index in $G$.
(We can define $E \le F$ as the inverse image under the natural map
$F \rightarrow F/J$ of the $e$-th power of a complement in $F/J$ of the torsion 
subgroup $T$ of $F/J$, where $e$ is the exponent of $T$.)

If $G$ is abelian, then we choose $C$ to be any complement of $J$ in $E$.
Otherwise we apply Lemma~\ref{lem:bymaschke} below with $\widehat{G}=G/E$ to the
$\Z G/E$-module $E$ and its submodule $J$;
the submodule $U$ guaranteed by the lemma corresponds to a normal subgroup
$C$ of $G$ within $E$, with $J \cap C = 1$ and $|E:JC|$ finite.
Then, in either case, the free abelian group $JC$ is a direct product
$J \times C$, and has finite index in $G$ with $J$ and $C$ both normal in $G$.

We shall define both of our coset languages with respect to
a finite generating set $X$ for $G$ that is a union $X_J \cup X_C \cup X_T$, 
where $X_J$ and $X_C$ are finite generating sets for $J$ and $C$, and
$X_T$ is a set of (not necessarily unique) representatives
of the nontrivial cosets of $JC$ in $G$, satisfying the condition that
whenever a coset has nontrivial intersection with $H$, the representatives in
$X_T$ are all within $H$.

We describe first the construction of $L_1^H$.  The quotient $G/J$ is
virtually abelian with free abelian subgroup $JC/J$ of finite index.
For each $g \in G$, write $\bar{g}$ for the coset $Jg$. 
By \cite[Proof of Corollary 4.2.4]{ECHLPT}, there is an
automatic structure (actually a geodesic biautomatic structure with uniqueness)
for $\bar{G}:= G/J$ with language $L_{\bar{G}}$  consisting of words 
over a finite generating set $Z$ for $\bar{G}$ of the form $Z_C \cup Z_T$,
where $Z_C$ is a particular generating set for $\overline{CJ}$ and $Z_T$
is a set of (unique) representatives of the nontrivial cosets of
$\overline{CJ}$ in $\bar{G}$, satisfying the condition that whenever 
a coset has nontrivial intersection with $\bar{H}$ its representative is
chosen to be in $\bar{H}$. We let $K$ be the fellow traveller constant
associated with this automatic structure.
The subsets $X_C,X_T$ of $G$ that we need to define $X$ are chosen
to be subsets of $G$ that map bijectively under the map 
$g\mapsto \bar{g}$ to $Z_C,Z_T$, and such that $X_C \subseteq C$,
while $X_J$ can be any finite generating set of $J$.
So we have a bijection $\rho:X_C \cup X_T$ to $Z_C \cup Z_T$, and
we extend $\rho^{-1}$ to a monoid homomorphism that maps words over
$Z_C \cup Z_T$  to the corresponding words over
$X_C \cup X_T$.

The language $L_{\bar{G}}$ is defined in
\cite[Proof of Corollary 4.2.4]{ECHLPT}
to be a set of words of the form $\bar{w}$ or 
$\bar{w}\bar{t}$, where $\bar{w}$ is a word over $Z_C$, and $\bar{t} \in Z_T$.
We define $L_1^H := \rho^{-1}(L_{\bar{G}})$.
So, as the image of a monoid homomorphism, $L_1^H$ is regular,
and its elements have the form
$wt$, where $w$ is a word over $X_C$ and $t \in X_T \cup \{ \epsilon \}$.
We observe also that the set $Z_C$ and also the set of words $\bar{w}$ that
arise in this language are invariant under conjugation by elements of $\bar{G}$.

We claim that the language $L_1^H$ is a strong automatic coset system
for $(G,H)$.
We have seen that it is regular, and it contains a full set of coset
representatives of $H$ in $G$
(recall that $J \subseteq H$).

It remains to prove the fellow traveller property.
So suppose that $w_1t_1, w_2t_ 2\in L_1^H$, and
$x \in X \cup X^{-1} \cup \{ 1 \}$, $h \in H$
with \[ w_1t_1x =_G hw_2t_2.\quad (\ast)\]
We need to show that the paths ${}_1 w_1$ and ${}_h w_2$ fellow travel
(and hence so do ${}_1 w_1t_1x$ and ${}_h w_2t_2$).

Our first step towards this proof is to define $c_1 \in C$, $j_1 \in J$ and
$t_3 \in X_T$ such that $t_1x =_G j_1c_1t_3$.
Since there are only finitely many possible choices for each of $t_1$, and $x$,
we see that $c_1$ and $j_1$ are bounded in length.
Let $B$ be an upper bound on their lengths.
Now we find $w_3 \in L_1^H \cap (X_C^{\pm})^*$, with
$w_1 c_1 =_G w_3$ (and so $w_1 j_1c_1 =_G j_1w_3$) and, since $L_{\bar{G}}$ is
an automatic structure for $\bar{G}$,
we see that ${}_1w_1$ and ${}_1w_3$ fellow travel at distance $K|c_1| \leq KB$.
We now have $w_1t_1x =_G j_1w_3t_3$.

Now we consider the right hand side $hw_2t_2$ of the equation $(\ast)$.
We can find $j' \in J, t_4 \in X_T$,
with $h=_G j't_4$ and so $hw_2t_2 =_G j't_4w_2t_2 = _G j't_4 w_2t_4^{-1}t_4t_2$.
As we observed above, the
generating set $Z_C$ of $\overline{CJ}$ is closed under conjugation by
elements of $\bar{G}$ and so, for each generator $y$ that occurs
in the word $w_2$, the image $t_4yt_4^{-1}$ in $G/J$ is in $Z_C$.
The normality of $C$ in $G$ ensures that $t_4yt_4^{-1}$ is a generator in
the set $X_C$ that consists of inverse images under $\rho$ of the
elements of $Z_C$.
Let $w_4$ be the word formed from $w_2$ by replacing each of its generators $y$
by the generator in $X_C$ that represents $t_4yt_4^{-1}$. Then
$w_4 =_G t_4w_2t_4^{-1}$ and, by the invariance property of the language
$L_{\bar{G}}$ mentioned earlier,
the image of $w_4$ in $\bar{G}$ lies in  $L_{\bar{G}}$.
Now we define $c_2 \in C,\,j_2 \in J$
(also each bounded above in length by $B$),
and $t_5 \in X_T$ such that $t_4t_2=_G j_2c_2t_5$, and hence
$j't_4w_2t_2 =_G j'w_4j_2c_2t_5 =_G j'j_2w_5t_5$, where $w_5=_G w_4c_2$.
Just as for the words ${}_1w_1$ and ${}_1w_3$ discussed above,
we find that ${}_1 w_4$ and ${}_1 w_5$ fellow travel at distance $KB$.

Now recall that we have $w_1t_1x =_G hw_2t_2 \quad (\ast)$.
The left hand side of $(\ast)$ is equal in $G$ to $j_1w_3t_3$,
and the right hand side to $j'j_2w_5t_5$,
Since $X_T$ is a transversal of $JC$ in $G$ we have $t_3=t_5$
and, since $JC$ is a direct product of $C$ and $J$, we have $w_3=_G w_5$
and $j_1=_G j'j_2$. Since $L_{\bar{G}}$ is an automatic structure with 
uniqueness, we also have $w_3=w_5$ (as words).
Further, $j'$ is bounded in length by $2B$. 

So, since  the pairs of words $({}_1w_1,{}_1w_3)$, and
$({}_1 w_4,{}_1 w_5) = ({}_1 w_4,{}_1w_3)$ both $KB$-follow travel,
to complete our proof it suffices to show that ${}_1 w_4$ and ${}_h w_2$
fellow travel. We recall that for generators
$a_1,\ldots,a_n,b_1,\ldots,b_n \in X_C^{\pm 1}$ with
$b_i=_G t_4a_it_4^{-1}$, we have $w_2=a_1\cdots a_n$ and
$w_4=b_1\cdots b_n$. For each $i$, the word difference
$(a_1\cdots a_i)^{-1}t_4(b_1\cdots b_i)$ is equal in $G$ to $t_4^{-1}$,
and so ${}_1 w_4$ and ${}_{t_4} w_2$ $1$-fellow travel. 
Since $h=_G j't_4$, $j'$ is bounded in length and $JC$ is abelian,
it follows that ${}_1 w_4$ and ${}_h w_2$ fellow travel, and we are done.

We turn now to the definition of our second synchronous automatic coset
system $L_2^H$ in the case when $G$ is abelian.
In this case, we allow $X_T$ to be any finite set of elements from
$G \setminus JC$ that contains at least one representative of each nontrivial
coset of $JC$ in $G$. 
Further, the conditions on the generating sets $X_J$ and $X_C$ of $J$ and $C$
are different from those of $L_1^H$; they are chosen to ensure that
for all equations of the form  $t_1t_2t_3=_G jct_4$ with
$t_1,t_2,t_3,t_4 \in X^\pm_T$, $j \in J$ and $c \in C$,
the elements $j$ and $c$ and included in $X_J$ and $X_C$.
(This property is not used in the current proof, but it is required in the
proof of Proposition~\ref{prop:amalg_relhyp_ab} below.)

For our coset language $L_2^H$, we take the set of words of the form $wv$,
where $w \in \SLex(C,X_C)$, and $v$ is a word of length at most $2$ over $X_T$.
This language is regular, as it is the concatenation of two regular
languages, and this language contains representatives of all cosets of
$H$ within $G$.

It remains to prove the fellow traveller property,
and we can do this very much as we did for $L_1^H$.  We suppose 
that $w_1v_1, w_2v_2\in L_2^H$ with $w_1,w_2 \in \SLex(C,X_C)$ and $v_1,v_2$
words of length at most $2$ over $X_T$, and $x \in X \cup X^{-1} \cup \{ 1 \}$,
$h \in H$ with \[ w_1v_1x =_G hw_2v_2.\quad (\ast)\]
We can write $h = j't$ with $j' \in J$ and $t \in X_T \cup \{ 1 \}$, and so
$w_1v_1x =_G j'w_2tv_2$. Then $JCv_1x = JCtv_2$, so  $v_1x =_G jc tv_2$,
for some $j \in J$ and $c \in C$, and hence $j' =_J j$ and $w_1c =_C w_2$.
Since there are only finitely many possible $v_1$, $v_2$ and $t$, the lengths
of $j$ and $c$ are bounded. So, since $w_1,w_2 \in \SLex(C,X_c)$, they
$K$-fellow travel for some constant $K$ and hence $_{1}w_1v_1$ and $_{h}w_2v_2$
$K'$-fellow travel for some (larger) constant $K'$.
\end{proof}

\begin{lemma}
\label{lem:bymaschke}
Let $\widehat{G}$ be a finite group, let $V$ be a finite dimensional
torsion-free $\Z\widehat{G}$-module, and $W$ a submodule.
Then there exists a $\Z\widehat{G}$-submodule $U$ of $V$
with $U \cap W = \{0\}$ such that $V/(U \oplus W)$ is finite.
\end{lemma}
\begin{proof}
Let $\widehat{V}=V \otimes \Q$ and $\widehat{W}=W \otimes \Q$
be the corresponding $\Q\widehat{G}$-modules.
By Maschke's theorem, there exists a $\Q\widehat{G}$-submodule
$\widehat{U}$ of $\widehat{V}$ with
$\widehat{V}=\widehat{U} \oplus \widehat{W}$.
Let $e_1,\ldots,e_n$ be a $\Z$-basis of $V$, which we may consider also
as a $\Q$-basis of $\widehat{V}$. We can choose a basis $u_1,\ldots,u_k$
of $\widehat{U}$ such that
the matrices representing the action of $\widehat{G}$ have integer entries.
Define $\lambda_{ij} \in \Q$ by $u_i=\sum_{j=1}^n \lambda_{ij} e_j$.
Let $m$ be a common multiple of the denominators of all the $\lambda_{ij}$,
and define $U \subseteq V$ to be the $\Z$-module generated by the elements 
$m u_i,\ldots,m u_k$ of $\widehat{U}$. Then $U \oplus W$ has rank $n$,
and so must
have finite index in $V$.
\end{proof}

By Corollary \ref{cor:artincox},
the hypotheses of the following result are satisfied in particular when
$G_1$ is a Coxeter group or an Artin group of large type on its natural
generating set $X_1$, and $H$ is a parabolic subgroup.

\begin{proposition}\label{prop:synchegs}
Suppose that $H = \langle Y \rangle \le G_1 = \langle X_1 \rangle$ with
$Y \subset X_1$, where $H$ and $G_1$ satisfy the hypotheses of
Proposition \ref{prop:concatenate} with $G_1$ and $X_1$ in place of
$G$ and $X$.  Suppose that $H$ is also a subgroup of the finitely generated
abelian group $G_2$. Then $(G_1 *_H G_2,H)$ is strongly synchronously
coset automatic.
\end{proposition}

\begin{proof}
We use the coset language $L_1^H$
constructed in the proof of Proposition~\ref{prop:concatenate} for $(G_1,H)$
(where this language is $\SLex^H$).
We extend $Y$ to a generating set $X_2$ of $G_2$
such that $H \cap X_2 = Y$ and define the coset language $L_2^H$ for
$(G_2,H)$ with respect to $X_2$ as in the proof of Proposition \ref{prop:ab},
where it is called $L^H$.

We claim that the language $L^H$ over $X := X_1 \cup X_2$ constructed for
$(G_1 \ast_H G_2,H)$ in the proof of Theorem~\ref{thm:gog_coset} is
synchronous, so we need to work through that proof in our current context, and
we shall adopt the notation used in that proof without further comment.
Note first that the graph $\lam$ has two vertices and a single edge,
which lies in the maximal tree $\TT$, so no letter $s_e$ appears in any
deflated words in the language, and Case (4) in the proof of
Theorem~\ref{thm:gog_coset} does not arise.

Since both $L_1^H$ and $L_2^H$ contain unique representatives of each coset of
$H$, so does $L^H$. So Case (1) in the proof of Theorem~\ref{thm:gog_coset},
where the two different words $w,w' \in L^H$ lie in the same coset, does
not arise. But the arguments used in the proof of that case are applied to
each of the cases (i.e. cases (2) and (3)) in which $Hwx = Hw'$ with $x \in X$.
Recall that $w = u_0u_1 \cdots u_k$, where $e_1e_2 \cdots e_k$ is the path
in $\lam$ associated with $w$.
We shall just consider Case (2), in which $u_k$ and $x$ lie in the same
subgroup $G_{\tau(e_k)}$. The argument in Case (3) is similar.

In Case~(2) we have $u_kx =_G h_ku_k'$ for some $h_k \in H$, where the
$X$-length of $h_k$,
which is the same as its $Y$-length, is bounded above by some constant $K$.
If $k=0$, then the synchronous fellow travelling of ${}_1w$ and
${}_{h}w'$ (where $h = h_0$) follows from the
fact that $L_1^H$ and $L_2^H$ are both synchronous automatic coset systems.
So we may assume that $k>0$.

Now, for $1 \le i \le k$, we have  
$u_{i-1}h_i =_G  h_{i-1}u_{i-1}'$ for some $h_{i-1} \in H$.
Then $w' = u_0'u_1' \cdots u_{\ell}'$ for some $\ell \le k$, where
$u_{\ell}' \ne \emptyword$, and $u_i = \emptyword$ for $\ell < i \le k$. 
Note that if $h_i=_G 1$ for some $i$ then, by the uniqueness property of $L^H$,
we have $u_{j}=u_{j}'$ and $h_{j}=_G 1$ for all $j < i$.

Suppose next that $|h_i|=1$ for some $i$, so $h_i$ is a
generator in $Y^{\pm}$. If $u_{i-1} \in L_2^H$, then $h_{i-1}=h_i$ and
$u_{i-1}=u_{i-1}'$.  If $u_{i-1} \in L_1^H$ then, as stated in the first
paragraph of the proof of Proposition~\ref{prop:concatenate}, we have either
\begin{mylist}
\item[(a)] $||u_{i-1}| - |u_{i-1}'|| \le 1$ and $h_{i-1} =_G 1$; or
\item[(b)] $|u_{i-1}| = |u_{i-1}'|$ and $|h_{i-1}|=1$.
\end{mylist}
So in fact one of (a) and (b) must apply irrespective of whether
$u_{i-1}$ is in $L_1^H$ or $L_2^H$.

Now if $|h_i| = m > 1$, then $h_i$ is a product $x_1 \cdots x_m$
of $m$ elements of $Y^{\pm}$. We can then apply the above argument to each
of $x_1, \ldots, x_m$ in turn, yielding equations
$u_{i-1}x_1 =_G y_1v_1$,  $v_1x_2 =_G y_2v_2, \ldots, v_{m-1}x_m=_G y_mv_m$,
where each $v_i \in L^H$, each $y_i \in Y^{\pm 1} \cup \{\emptyword\}$,
$v_m = u_{i-1}'$, and $h_{i-1} =_G y_1 \cdots y_k$.
So, since (a) or (b) applies to each of these equations, we have either
\begin{mylist}
\item[(i)] $|h_{i-1}| < |h_i|$ and  $||u_{i-1}| - |u_{i-1}'|| \le |h_i|$; or
\item[(ii)] $|h_{i-1}| = |h_i|$ and  $|u_{i-1}| = |u_{i-1}'|$.
\end{mylist}

In particular, since $|h_k| \le K$, we have $|h_i| \le K$ for all $i$.
So Case (i) can occur for at most $K$ values of $i$, and hence
\[\sum_{i=1}^{k} ||u_{i-1}| - |u_{i-1}'|| \le K^2.\]
It is proved in Theorem~\ref{thm:gog_coset} that the paths labelled $w$ and
$w'$ asynchronously $L$-fellow travel for some constant $L$ and that, for each
$i$, the beginnings and ends of the subpath labelled $u_i$ correspond
to those of $u_i'$ in the fellow travelling. From the above inequality, we
see that, if the beginnings of these subpaths labelled $u_i$ and $u_i'$ are
at distances $i_1$ and $i_2$ from the basepoint, then $|i_1-i_2| \le K^2$.
It follows that ${}_1w$ and ${}_hw'$ synchronously fellow travel
with constant at most $L+K^2$, which completes the proof.
\end{proof}

\begin{proposition}
\label{prop:amalg_relhyp_ab}
Suppose that the group $G_1$ is finitely generated
and hyperbolic relative to a collection of
abelian subgroups, and let $H$ be one of those subgroups.
Suppose that $H$ is also a subgroup of the finitely generated
abelian group $G_2$. Then $(G_1 *_H G_2,H)$ is strongly synchronously
coset automatic.
\end{proposition}

\begin{proof}
The idea of the proof is first to find coset languages $L_1^H$ and $L_2^H$ for
$(G_1,H)$ and $(G_2,H)$ with respect to suitable generating sets $X_1$ and
$X_2$, then to use Theorem \ref{thm:gog_coset} to find a strong
asynchronous automatic coset system $L^H$ for $(G_1 *_H G_2,H)$,
and finally to apply Proposition \ref{prop:synchsub} to find a synchronous
subsystem within $L^H$.
For $L_1^H$ we use the language $L_H$ constructed in the proof of
Proposition~\ref{prop:relhyp}, and for $L_2^H$ we use the language also
called $L_2^H$ from the proof of Proposition~\ref{prop:vab}.

For the application of Proposition \ref{prop:synchsub}, we need to choose the
generating sets $X_1$ and $X_2$ for $G_1$ and $G_2$ such that
$Y := X_1 \cap H = X_2 \cap H$.
It is not a problem to find generating sets $X_1$, $X_2$ for $G_1$, $G_2$ satisfying
this condition. But the constructions of $L_1^H$ and $L_2^H$ both involve
the addition of new generators to $Y$. We can handle this situation as
follows.  First we extend $X_1$ (and so also $X_2$ and $Y$) during the
construction of $L_1^H$. Then we further extend $X_2$ (and so also $X_1$ and
$Y$) during the construction of $L_2^H$.  Since there is a geodesic
biautomatic structure for $H$ on any finite generating set, Lemma \ref{lem:AC}
allows us to reconstruct $L_1^H$ using the new generating set $X_1$.

We see from the proof of Proposition~\ref{prop:relhyp} that $L_1^H$
consists of those words in the geodesic biautomatic language $L_1$ for $G_1$
that do not begin with a letter in $Y^{\pm}$.

As stated earlier, for the language $L_2^H$ we use the second coset automatic
structure in the  Proposition~\ref{prop:vab}, which is also named $L_2^H$ there.
So $X_2 = X_J \cup X_C \cup X_T$, where $C$ and $J$ are disjoint free abelian
subgroups of $G_2$ with $|G_2:JC|$ finite, $JC \cap H = J$, and $X_T$ contains
a transversal for $JC$ in $G_2$, The elements of $L_2^H$ are words of the form
$wv$, where $w \in \SLex(C,X_C)$ and $v$ is a word of length at most 2 over $X_T$.

Now let $X = X_1 \cup X_2$ and let
$w \in (X^{\pm})^*$ be a shortest representative of its coset
of $H$ in $G = G_1 *_H G_2$. Then we can write $w$ as $w_1w_2 \cdots w_k$, where each
$w_i$ lies alternately in $(X_1^{\pm})^*$ or in $(X_2^{\pm})^*$,
and $w_i$ is a nonempty word that does not begin with a
generator from $Y^{\pm}$. We aim to replace $w$ with a word $v$ of the
same length in the same coset of $H$ such that, in the corresponding
decomposition $v=v_1v_2 \cdots v_k$, each $v_i$ is in $L_1^H$ or in $L_2^H$.
If we can do this, then $v \in L^H \cap \Geo^H$ representing $Hw$, and we can
apply Proposition \ref{prop:synchsub} to deduce the existence of a synchronous
subsystem of $L^H$.

Since the words $w_i$ that lie in $G_1$ must be geodesic words over $X_1$,
we may replace them if necessary by words of the same length representing
the same group elements that lie in the geodesic biautomatic language $L_1$ for
$G_1$. Then, since we are assuming that $w_i$ does not begin with a
letter in $Y$, we have $w_i \in L_1^H$. (This replacement may decrease
$|w_i|$ and increase $|w_{i-1}|$, but provided that we replace
the words $w_i$ in order of decreasing $i$, this is not a problem.)

We may assume that the words $w_i$ in $G_2$ contain no generators in $Y$,
since these could be moved to the left of the word. We may also assume that
the letters in $w_i$ from $X_T^{\pm}$ lie at the end of the word. If
we have three or more such letters then, from our choice of $X_2$, we can
replace them with a word of the same length containing a generator from
$Y$. So we may replace $w_i$ by a word $v_i = u_1u_2$, where $u_1$ and $u_2$
are words over $X_C$ and $X_T$, respectively, and $|u_2| \le 2$.
We may also assume that $u_1$ is the shortlex least representative over
$X_C$ of the element that it represents, and hence $v_i \in L_2^H$,
which completes the proof.
\end{proof}

\begin{example}
In this example we note that there is a pair $(G,H)$ that is strongly 
synchronously coset automatic, but
computer experiments suggest that it does not have  $\lambda$-limited crossover
for any $\lambda$ with respect to any generating set $Y$ of $H$. (But we
have no means of proving that.)
The group $G$ is the trefoil knot group (or the $3$-string braid group)
$\langle x,y \mid xyx = yxy \rangle$ and $H$ is the free abelian rank $2$
subgroup $\langle x,d \rangle$, where $d$ is the central element $(xyx)^2$.
\end{example}

\Addresses


\begin{thebibliography}{1}

\bibitem{Alonso} J.M. Alonso, Growth functions of amalgams. Arboreal group theory
(Berkeley, CA, 1988), 1--34, 
\emph{Math. Sci. Res. Inst. Publ.}, 19, Springer, New York, 1991. 
 
\bibitem{Antolin} Y.\ Antol{\'i}n, Counting subgraphs in fftp graphs with symmetry,
to appear in
Mathematical Proceedings of the Cambridge Philosophical Society,
{\tt https://doi.org/10.1017/S0305004119000422}

\bibitem{AC}
Y.\ Antol{\'i}n and L.\ Ciobanu,
Finite generating sets of relatively hyperbolic groups and
applications to geodesic languages.
{\em Trans. Amer. Math. Soc.}, 368(11):7965--8010, 2016.

\bibitem{BGSS}
G. Baumslag, S.M. Gersten, M. Shapiro and H. Short,
Automatic groups and amalgams.
{\em J.~Pure Appl. Algebra} 76(3):229--316, 1991.

\bibitem{Bourbaki} N. Bourbaki, 
{\em \'{E}l\'{e}ments de math\'{e}matique. {F}asc. {XXXIV}. {G}roupes et
              alg\`ebres de {L}ie. {C}hapitre {IV}: {G}roupes de {C}oxeter et
              syst\`emes de {T}its. {C}hapitre {V}: {G}roupes engendr\'{e}s par
              des r\'{e}flexions. {C}hapitre {VI}: syst\`emes de racines},
   Actualit\'{e}s Scientifiques et Industrielles, No. 1337,
 Hermann, Paris, 1968.

\bibitem{Brady}
N. Brady, Sol geometry groups are not asynchronously
automatic. {\em Proc. London Math. Soc. (3)}, 83(1):93--119, 2001.

\bibitem{BH}
B. Brink and R.\,B. Howlett,
A finiteness property and an automatic structure for {C}oxeter groups.
{\em Math. Ann.}, 296:179--190, 1993.

\bibitem{BHS}
M. Brittenham, S. Hermiller and T. Susse, 
The geometry of the word problem for $3$--manifold groups. 
{\em J. Algebra}, 499:111--150, 2018.

\bibitem{chiswellHNN}
I.M. Chiswell,
The growth series of HNN extensions. 
{\em Comm. Algebra}, 22(8):2969--2981, 1994.

\bibitem{chiswellgp}
I.M. Chiswell,
The growth series of a graph product. 
{\em Bull. London Math. Soc.}, 26(3):268--272, 1994.  


\bibitem{CH}
L. Ciobanu and S. Hermiller, Conjugacy growth series and languages
in groups. {\em Trans. Amer. Math. Soc.}, 366(5):2803--2825, 2014.

\bibitem{Dahmani}
F. Dahmani, Combination of convergence groups.
 {\em Geom. Topol.}, 7:933--963, 2003.

\bibitem{Dehornoy1}
P. Dehornoy, The subword reversing method. {\em Internat. J. Algebra Comput.}, 
21(1-2):71--118, 2011.

\bibitem{DHR}
P. Dehornoy,  D.F. Holt and S. Rees, 
Multifraction reduction IV: Padding and Artin-Tits monoids of 
sufficiently large type.
{\em J.~Pure Appl. Algebra}, 222(12):4082--4098, 2018.


\bibitem{Elder}
M. Elder, Finiteness and the falsification by fellow traveler
property. {\em Geom. Dedicata}, 95:103--113, 2002.

\bibitem{ECHLPT}
D.B.A.~Epstein, J.W.~Cannon, D.F.~Holt, S.~Levy, M.S.~Paterson and
W.P.~Thurston,
{\em Word Processing in groups}.
Jones and Bartlett, 1992.

\bibitem{HermillerMeier}
S. Hermiller and J. Meier, 
Algorithms and geometry for graph products of groups.
{\em J.~Algebra} 171:230--257, 1995.


\bibitem{Higgins} P.J. Higgins, The fundamental groupoid of a graph of groups.
{\em J. London Math. Soc.} 13(2): 145--149, 1976.

\bibitem{HoltHurt}
D.F. Holt and D. Hurt, Computing automatic coset systems
and subgroup presentations. {\em J. Symbolic Comput.} 27(1):1--19,
1999.

\bibitem{HRRbook}
D.F. Holt, S. Rees and C.E. R\"over. 
{\em Groups, Languages and Automata}.
London Mathematical Society Student Texts, 
Cambridge University Press, 2017.

\bibitem{HR1}
D.F. Holt and S. Rees,
Artin groups of large type are shortlex automatic with regular geodesics,
{\em Proc. London Math. Soc.}, 104:486--512, 2012.

\bibitem{HR2}
D.F. Holt and S. Rees,
Shortlex automaticity and geodesic regularity in Artin groups, 
{\em Groups, Complex. Cryptol.}, 5(1):1--23, 2013.

\bibitem{mann}
A. Mann, 
{\em How groups grow}.
London Mathematical Society Lecture Note Series, 395, 
Cambridge University Press, Cambridge, 2012.

\bibitem{MS}
M. Mihalik, S. Tschantz,
Visual decompositions of Coxeter groups.
{\em Groups Geom. Dyn.}, 3(1):173--198, 2009.

\bibitem{osin}
Denis~V. Osin.
\newblock Relatively hyperbolic groups: intrinsic geometry, algebraic
properties, and algorithmic problems.
\newblock {\em Mem. Amer. Math. Soc.}, 179(843):vi+100, 2006.

\bibitem{Redfern}
Ian Redfern, Automatic coset systems.
PhD thesis, University of Warwick, 1993.


\bibitem{ScottWall}
P. Scott and T. Wall,
Topological methods in group theory.
In {\em Homological group theory ({P}roc. {S}ympos., {D}urham,
  1977)}, volume~36 of {\em London Math. Soc. Lecture Note Ser.}, pages
  137--203. Cambridge Univ. Press, Cambridge-New York, 1979.

\bibitem{Serre}
J.-P. Serre,
{\em Trees}.
Springer Monographs in Mathematics, Springer-Verlag, Berlin, 2003.

\bibitem{Shapiro}
M. Shapiro,
Automatic structure and graphs of groups.
Topology '90 (Columbus, OH, 1990), 355-380,
Ohio State Univ. Math. Res. Inst. Publ., 1, de Gruyter, Berlin, 1992.


\end{thebibliography}
\end{document}